\documentclass[11pt,leqno,a4paper]{amsart}
\usepackage{amssymb}
\usepackage[pagebackref]{hyperref}
\usepackage{exscale}
\hypersetup{colorlinks, linkcolor={blue},citecolor={red}}

\usepackage[utf8]{inputenc}

\calclayout
\allowdisplaybreaks

\DeclareMathOperator{\supp}{supp}
\DeclareMathOperator{\dist}{dist}
\DeclareMathOperator{\diam}{diam}
\DeclareMathOperator{\interior}{int}
\DeclareMathOperator*{\esssup}{ess\,sup}

\def\div{\mathop{\operatorname{div}}}

\def\Xint#1{\mathchoice
{\XXint\displaystyle\textstyle{#1}}%
{\XXint\textstyle\scriptstyle{#1}}%
{\XXint\scriptstyle\scriptscriptstyle{#1}}%
{\XXint\scriptscriptstyle\scriptscriptstyle{#1}}%
\!\int}
\def\XXint#1#2#3{{\setbox0=\hbox{$#1{#2#3}{\int}$ }
\vcenter{\hbox{$#2#3$ }}\kern-.585\wd0}}
\def\barint{\Xint-}

\newcommand{\bariint}{\barint\mkern-11.5mu\barint}
\newcommand{\re}{\mathbb{R}}
\newcommand{\rest}[1]{\rule[-4pt]{.38pt}{10.5pt}_{\,#1}}

\renewcommand{\iint}{\int\mkern-13.5mu\int}
\renewcommand{\emptyset}{\mbox{\textup{\O}}}

\theoremstyle{plain}
\newtheorem{theorem}[equation]{Theorem}
\newtheorem{lemma}[equation]{Lemma}
\newtheorem{corollary}[equation]{Corollary}
\newtheorem{proposition}[equation]{Proposition}

\theoremstyle{definition}
\newtheorem{definition}[equation]{Definition}

\theoremstyle{remark}
\newtheorem{remark}[equation]{Remark}

\numberwithin{equation}{section}

\newcommand{\vertiii}[1]{{\left\vert\kern-0.25ex\left\vert\kern-0.25ex\left\vert #1 
    \right\vert\kern-0.25ex\right\vert\kern-0.25ex\right\vert}}

\begin{document}
\title[Perturbations of elliptic operators in 1-sided chord-arc domains]
{Perturbations of elliptic operators in 1-sided chord-arc domains. Part I: Small and large perturbation for symmetric operators}

\author{Juan Cavero}

\address{Juan Cavero
\\
Instituto de Ciencias Matem\'{a}ticas CSIC-UAM-UC3M-UCM
\\
Consejo Superior de Investigaciones Cient\'{\i}ficas
\\
C/ Nicol\'{a}s Cabrera, 13-15
\\
E-28049 Madrid, Spain} \email{juan.cavero@icmat.es}

\author{Steve Hofmann}

\address{Steve Hofmann
\\
Department of Mathematics
\\
University of Missouri
\\
Columbia, MO 65211, USA} \email{hofmanns@missouri.edu}

\author{Jos\'{e} Mar\'{\i}a Martell}

\address{Jos\'{e} Mar\'{\i}a Martell
\\
Instituto de Ciencias Matem\'{a}ticas CSIC-UAM-UC3M-UCM
\\
Consejo Superior de Investigaciones Cient\'{\i}ficas
\\
C/ Nicol\'{a}s Cabrera, 13-15
\\
E-28049 Madrid, Spain} \email{chema.martell@icmat.es}

\thanks{ 
The first and third authors acknowledge financial
  support from the Spanish Ministry of Economy and Competitiveness,
  through the “Severo Ochoa” Programme for Centres of Excellence in
  R\&D” (SEV-2015-0554). They also acknowledge that the research
  leading to these results has received funding from the European
  Research Council under the European Union's Seventh Framework
  Programme (FP7/2007-2013)/ ERC agreement no. 615112 HAPDEGMT. 
	The second author was supported by NSF grant DMS-1664047.
	}

\date{\today}



\begin{abstract}
Let $\Omega\subset\re^{n+1}$, $n\ge 2$, be a 1-sided chord-arc domain, that is, a domain which satisfies interior Corkscrew and Harnack Chain conditions (these are respectively scale-invariant/quantitative versions of the openness and path-connectedness), and whose boundary $\partial\Omega$ is $n$-dimensional Ahlfors regular. Consider $L_0$ and $L$ two real symmetric divergence form elliptic operators and let $\omega_{L_0}$, $\omega_L$ be the associated elliptic measures. We show that if $\omega_{L_0}\in A_\infty(\sigma)$, where $\sigma=H^n\rest{\partial\Omega}$, and $L$ is a perturbation of $L_0$ (in the sense that the discrepancy between $L_0$ and $L$ satisfies certain Carleson measure condition), then $\omega_L\in A_\infty(\sigma)$. Moreover, if $L$ is a sufficiently small perturbation of $L_0$, then one can preserve the reverse Hölder classes, that is, if for some $1<p<\infty$, one has $\omega_{L_0}\in RH_p(\sigma)$ then $\omega_{L}\in RH_p(\sigma)$. Equivalently, if the Dirichlet problem with data in $L^{p'}(\sigma)$ is solvable for $L_0$ then so it is for $L$.
These results can be seen as extensions of the perturbation theorems obtained by  Dahlberg,
Fefferman-Kenig-Pipher, and Milakis-Pipher-Toro in more benign settings. As a consequence of our methods we can show that for any perturbation of the Laplacian (or, more in general, of any elliptic symmetric operator with Lipschitz coefficients satisfying certain Carleson condition) if its elliptic measure belongs to $A_\infty(\sigma)$ then necessarily $\Omega$ is in fact an NTA domain (and hence chord-arc) and therefore its boundary is uniformly rectifiable.
\end{abstract}

\maketitle

\tableofcontents

\section{Introduction and Main results}
In the last years there has been a renewed interest in understanding the behavior of the harmonic measure, or 
more  generally of elliptic measures, in very rough domains. Part of the effort consisted of establishing a connection between the ``regularity'' of the boundary of the domain, expressed in terms of some rectifiability, and the good behavior of the harmonic or elliptic measures, written in terms of absolute continuity with respect to the surface measure. 

A 1-sided chord-arc domain $\Omega\subset \mathbb{R}^{n+1}$, $n\ge 2$, is a set whose boundary $\partial\Omega$ is $n$-dimensional Ahlfors regular, and which satisfies interior Corkscrew and Harnack Chain conditions (these are are respectively scale-invariant/quantitative versions of the openness and 
path-connectedness; see Definitions \ref{defCS} and \ref{defHC}
below). The papers \cite{hofmartell,  hofmartelltuero} show that in the setting of 1-sided chord-arc domains, harmonic measure is in $A_\infty(\sigma)$, where $\sigma=H^n\rest{\partial\Omega}$ is the surface measure, if and only if $\partial\Omega$ is uniformly rectifiable (a quantitative version of rectifiability). It was shown later in \cite{MR3626548} that under the same background hypothesis, $\partial\Omega$ is uniformly rectifiable if and only if $\Omega$ satisfies an exterior corkscrew condition and hence $\Omega$ is a chord-arc domain. All these together, and additionally, \cite{MR3626548} in conjunction with
\cite{DJe} or \cite{Se}, give a characterization of chord-arc domains, or a characterization of the uniform rectifiablity of the boundary,  in terms of the membership of harmonic measure to the class $A_\infty(\sigma)$. For other elliptic operators $Lu=-\div(A\nabla u)$ with variable coefficients it was shown recently in \cite{HMT-var} that the same characterization holds provided $A$ is locally Lipschitz and has appropriately controlled oscillation near the boundary. 

This paper is the first part of a series of two articles where we consider perturbation of real elliptic operators in the setting of 1-sided chord-arc domains. Here we work with symmetric operators and study perturbations that preserve the $A_\infty(\sigma)$ property. We extend the work of \cite{MR1114608, MR3107693} (see also \cite{MR1828387}, \cite{MR2833577, MR2655385}) to the setting of 1-sided chord-arc domains and show that if the disagreement between two elliptic symmetric matrices satisfies certain Carleson measure condition then one of the the associated elliptic measures is in $A_\infty(\sigma)$ if and only if the other is in $A_\infty(\sigma)$. In other words, the property that the elliptic measure belongs to $A_\infty(\sigma)$ is stable under Carleson  measure type perturbations. As an immediate consequence of this we can see that the above characterization  of the fact that a domain is chord-arc, or its boundary is uniformly rectifiable, extends to any perturbation of the Laplacian or more in general to any perturbation of the operators considered in \cite{HMT-var}. In particular, our result allows a characterization with operators whose coefficients are not even continuous.

Our method to obtain the perturbation result differs from that in \cite{MR1114608, MR3107693}, and uses the so-called extrapolation of Carleson measures  which originated in  \cite{MR1020043} (see also \cite{MR1828387, MR1879847, MR1934198}).  We shall utilize this technique
in the form developed in \cite{MR2833577, MR2655385} (see also \cite{hofmartell}). 
The method is a bootstrapping argument, based on
the Corona construction of Carleson \cite{Car} and Carleson and Garnett \cite{CG}, that, roughly speaking, allows one to reduce matters to the case in which the perturbation is small in some sawtooth subdomains.  
In particular, in the course of the proof we are implicitly treating  the case in which the perturbation is small, and this allows us to fine-tune the argument in order to obtain that  for sufficiently small perturbations, we not only preserve the class $A_\infty$ but we also we can keep the same exponent in the corresponding reverse Hölder class. More precisely, assume that $\omega_{L_0}$, the elliptic measure associated with $L_0$, belongs to the class $A_\infty(\sigma)$, then $\omega_{L_0}\in RH_p(\sigma)$ for some $p>1$ (that is, the Radon-Nikodym derivative of $\omega_{L_0}$ with respect to the surface measure satisfies a scale-invariant estimate in $L^p$). We  obtain that if $L$ is a sufficiently small perturbation (in a Carleson measure sense) of $L_0$, then $\omega_L\in RH_p(\sigma)$. This result can be seen as an extension of \cite{MR859772, MR3204862} (see also \cite{MR1394581}), where the small perturbation case is considered in the unit ball.  It is worth mentioning that in the present scenario, $\omega_{L_0}\in RH_p(\sigma)$ if and only if the $L^{p'}$-Dirichlet problem for $L_0$ is solvable (in a non-tangential fashion). Thus, the small perturbation case says that if the $L^{q}$-Dirichlet problem for $L_0$ is solvable for a given $1<q<\infty$ then so is the corresponding Dirichlet problem for $L$ provided $L$ is small perturbation of $L_0$. Analogously, saying that $\omega_{L_0}\in A_\infty(\sigma)$ is equivalent to the fact that $L^{q}$-Dirichlet problem for $L_0$ is solvable for some (large) $q$. Thus if we just assume that $L_0$ is an arbitrary perturbation of $L$ we conclude that the $L^{\tilde{q}}$-Dirichlet problem for $L$ is solvable for some (possibly larger) $\tilde{q}$.

In the second part of this series of papers \cite{CHMT}, together with Tatiana Toro, we consider the non-symmetric case and present another approach, interesting on its own right, to treat the ``large'' constant case in the same setting of 1-sided chord-arc domains. There, we start with $L_0$ and $L$ two real elliptic operators, non-necessarily symmetric,  whose disagreement satisfies a Carleson measure condition. Our method decouples the proof in two independent steps. The first one, which is the real perturbation result, shows that if $\omega_{L_0}\in A_\infty(\sigma)$  then all bounded null-solutions of $L$ satisfy Carleson measure estimates. The second step establishes that for any real elliptic operator non-necessarily symmetric $L$, the property that all bounded null-solutions of $L$ satisfy Carleson measure estimates yields that $\omega_L\in A_\infty(\sigma)$. This extends the work \cite{KKiPT} where they treated bounded Lipschitz domains and domains above the graph of a Lipschitz function. Let us point out that it is also shown in \cite{hofmartellgeneral} that the converse is true, namely, that  $\omega_L\in A_\infty(\sigma)$ implies that all bounded null-solutions of $L$ satisfy Carleson measure estimates. Hence, eventually both properties are equivalent. Finally, an interesting application of the method developed in \cite{CHMT} allows one to obtain that if $L=-\div(A\nabla)$ with $A$ locally Lipschitz such that $|\nabla A|\delta\in L^\infty(\Omega)$ (here $\delta$ is the distance to $\partial\Omega$) and $|\nabla A|^2\delta$ satisfies a Carleson condition then $\omega_L\in A_\infty(\sigma)$ if and only if $\omega_{L^\top}\in A_\infty(\sigma)$, where $L^\top$ is the adjoint operator, that is $L^\top=-\div(A^\top\nabla)$ with $A^\top$ being the adjoint matrix of $A$.

\medskip

Let us now state the main results of this paper, the precise definitions can be found in Section \ref{section:prelim}.

\begin{theorem}\label{perturbationtheo}
Let $\Omega\subset\re^{n+1}$, $n\ge 2$, be a 1-sided $\mathrm{CAD}$ (cf. Definition \ref{def-1CAD}). Let $Lu=-\div(A\nabla u)$ and $L_0u=-\div(A_0\nabla u)$ be real symmetric elliptic operators (cf. Definition \ref{ellipticoperator}). Define the disagreement between $A$ and $A_0$ in $\Omega$ by
\begin{equation}\label{discrepancia}
  \varrho(A, A_0)(X):=\sup_{Y\in B(X,\delta(X)/2)}|A(Y)-A_0(Y)|,\qquad X\in\Omega,
\end{equation}
where $\delta(X):=\dist(X,\partial\Omega)$, and write
\begin{equation}\label{eq:defi-vertiii}
\vertiii{\varrho(A, A_0)}:=\sup_{\substack{x\in\partial\Omega \\ 0<r<\diam(\partial\Omega)}}\frac{1}{\sigma(B(x,r)\cap\partial\Omega)}\iint_{B(x,r)\cap\Omega}\frac{\varrho(A, A_0)(X)^2}{\delta(X)}\,dX.
\end{equation}
Suppose that there exists $p$, $1<p<\infty$, such that the elliptic measure $\omega_{L_0}\in RH_p(\partial\Omega)$  (cf. Definition \ref{defrhp}). The following hold:

\begin{list}{$(\theenumi)$}{\usecounter{enumi}\leftmargin=1cm \labelwidth=1cm \itemsep=0.1cm \topsep=.2cm \renewcommand{\theenumi}{\alph{enumi}}}

\item If $\vertiii{\varrho(A, A_0)}<\infty$, then there exists $1<q<\infty$ such that $\omega_L\in RH_q(\partial\Omega)$. Here, $q$ and the implicit constant depend only on dimension, $p$, the 1-sided $\mathrm{CAD}$ constants, the ellipticity of $L_0$ and $L$, $\vertiii{\varrho(A, A_0)}$, and the constant in $\omega_{L_0}\in RH_p(\partial\Omega)$.

\item There exists $\varepsilon_0>0$ (depending only on dimension, $p$, the 1-sided $\mathrm{CAD}$ constants, the ellipticity of $L_0$ and $L$, and the constant in $\omega_{L_0}\in RH_p(\partial\Omega)$) such that if one has $\vertiii{\varrho(A, A_0)} \leq\varepsilon_0$, then $\omega_L\in RH_p(\partial\Omega)$, with the implicit constant depending only on dimension, $p$, the 1-sided $\mathrm{CAD}$ constants, the ellipticity of $L_0$ and $L$, and the constant in $\omega_{L_0}\in RH_p(\partial\Omega)$.
\end{list}
\end{theorem}

To present the characterization of chord-arc domains advertised above we need to introduce some notation. Let $\mathbb{L}_0$ be the collection of real symmetric elliptic operators $Lu=-\div(A\nabla u)$ (cf. Definition \ref{ellipticoperator}) such that $A\in \mathrm{Lip}_{\rm loc}(\Omega)$, $\big\||\nabla A|\,\delta\big\|_{L^\infty(\Omega)}<\infty$, and 
\begin{equation}\label{car-A-new}
\sup_{\substack{x\in\partial\Omega\\0<r<\diam(\partial\Omega)}} \frac1{\sigma(B(x,r)\cap\partial\Omega)}\iint_{B(x,r)\cap\Omega} |\nabla A(X)|dX<\infty.
\end{equation}
We also introduce $\mathbb{L}$, the collection of real symmetric elliptic operators $Lu=-\div(A\nabla u)$ (cf. Definition \ref{ellipticoperator}) for which there exists $L_0=-\div(A\nabla u)\in \mathbb{L}_0$ in such a way that $\vertiii{\varrho(A, A_0)}<\infty$. Note that all constant coefficient operators belong to $\mathbb{L}_0$ and also that $\mathbb{L}_0\subset \mathbb{L}$. 

\begin{corollary}
Let $\Omega\subset\re^{n+1}$, $n\ge 2$, be a 1-sided $\mathrm{CAD}$ (cf. Definition \ref{def-1CAD}). Let $L\in \mathbb{L}$
be a real symmetric elliptic operators. 
Then
$$
\omega_L\in A_\infty(\sigma)
\qquad\Longleftrightarrow\qquad
\Omega\mbox{ is a $\mathrm{CAD}$ (cf. Definition \ref{def-CAD})}.
$$
\end{corollary}

In the previous result the backward implication is well-known and follows from \cite{MR3107693, MR1829584} (see also \cite[Appendix A]{HMT-var}). For the forward implication, (that is, the fact that $A_\infty(\sigma)$ gives the existence of exterior corkscrews), the case when $L$ is the Laplacian was proved combining \cite{MR3626548, hofmartelltuero}. The case of operators in $\mathbb{L}_0$ is the main result of \cite{HMT-var}. Our contribution here is to extend $\mathbb{L}_0$ and to
be able to consider operators in $\mathbb{L}$ whose coefficients may not posses any regularity. The proof is as follows. Let $L=-\div(A\nabla u)\in\mathbb{L}$ be such that $\omega_L\in A_\infty(\sigma)
$. By the definition of the class $\mathbb{L}$, there exists $L_0=-\div(A_0\nabla u)\in \mathbb{L}_0$ such that $\vertiii{\varrho(A, A_0)}<\infty$. This and the fact that $\omega_L\in A_\infty(\sigma)$ allow us to invoke Theorem \hyperref[perturbationtheo]{\ref*{perturbationtheo}(a)} (note that we are switching the roles of $L_0$ and $L$) to conclude that $w_{L_0}\in A_\infty(\sigma)$. In turn, since $L_0\in\mathbb{L}_0$ we can invoke the main result in \cite{HMT-var} to conclude that $\Omega$ satisfies the exterior corkscrew condition, and therefore $\Omega$ is CAD as desired. 

\medskip

The plan of the paper is as follows. In Section \ref{section:prelim} we present some preliminaries, definition, and some background results that will be used throughout the paper. Section \ref{section:aux} contains some auxiliary results. The proof of 
Theorem \ref{perturbationtheo} is given in Sections \ref{section:main-a} and \ref{section:main-b}. Finally in Section \ref{section:apps} we present some applications of Theorem \hyperref[perturbationtheo]{\ref*{perturbationtheo}(b)}.

\section{Preliminaries}\label{section:prelim}
\subsection{Notation and conventions}
\begin{list}{$\bullet$}{\leftmargin=0.4cm  \itemsep=0.2cm}

\item  Our ambient space is $\re^{n+1}$, $n\ge 2$.

\item We use the letters $c$, $C$ to de
note harmless positive constants, not necessarily the same at each occurrence, which depend only on dimension and the constants appearing in the hypotheses of the theorems (which we refer to as the ``allowable parameters''). We shall also sometimes write $a\lesssim b$ and $a\approx b$ to mean, respectively, that $a\leq C b$ and $0<c\leq a/b\leq C$, where the constants $c$ and $C$ are as above, unless explicitly noted to the contrary. At times, we shall designate by $M$ a particular constant whose value will remain unchanged throughout the proof of a given lemma or proposition, but which may have a different value during the proof of a different lemma or proposition.

\item Given a domain (i.e., open and connected) $\Omega\subset\re^{n+1}$, we shall use lower case letters $x,y,z$, etc., to denote points on $\partial\Omega$, and capital letters $X,Y,Z$, etc., to denote generic points in $\re^{n+1}$ (especially those in $\Omega$).

%\item $E\subset\re^{n+1}$ will denote a closed set, and most of the times $E=\partial\Omega$ with $\Omega\subset\re^{n+1}$ being an open connected set.

\item The open $(n+1)$-dimensional Euclidean ball of radius $r$ will be denoted $B(x,r)$ when the center $x$ lies on $\partial\Omega$, or $B(X,r)$ when the center $X\in\re^{n+1}\setminus \partial\Omega$. A ``surface ball'' is denoted $\Delta(x,r):=B(x,r)\cap \partial\Omega$, and unless otherwise specified it is implicitly assumed that $x\in\partial\Omega$.  Also if $\partial\Omega$ is bounded, we typically assume that $0<r\lesssim\diam(\partial\Omega)$, so that $\Delta=\partial\Omega$ if $\diam(\partial\Omega)<r\lesssim\diam(\partial\Omega)$.

\item Given a Euclidean ball $B$ or surface ball $\Delta$, its radius will be denoted $r(B)$ or $r(\Delta)$
respectively.

\item Given a Euclidean ball $B=B(X,r)$ or surface ball $\Delta=\Delta(x,r)$, its concentric dilate by a factor of $\kappa>0$ will be denoted by $\kappa B=B(X,\kappa r)$ or $\kappa\Delta=\Delta(x,\kappa r)$.

\item For $X\in\re^{n+1}$, we set $\delta_{\partial\Omega}(X):=\dist(X,\partial\Omega)$. Sometimes, when clear from the context we will omit the subscript $\partial\Omega$ and simply write $\delta(X)$.

\item We let $H^n$ denote the $n$-dimensional Hausdorff measure, and let $\sigma_{\partial\Omega}:=H^n\rest{\partial\Omega}$ denote the ``surface measure'' on $\partial\Omega$. For a closed set $E\subset\re^{n+1}$ we will use the notation $\sigma_{E}:=H^n\rest{E}$. When clear from the context we will also omit the subscript and simply write $\sigma$.

\item For a Borel set $A\subset\re^{n+1}$, we let $\mathbf{1}_A$ denote the usual indicator function of $A$, i.e., $\mathbf{1}_A(x)=1$ if $x\in A$, and $\mathbf{1}_A(x)=0$ if $x\notin A$.

\item For a Borel set $A\subset\re^{n+1}$, we let $\interior(A)$ denote the interior of $A$, and $\overline{A}$ denote the closure of $A$. If $A\subset \partial\Omega$, $\interior(A)$ will denote the relative interior, i.e., the largest relatively open set in $\partial\Omega$ contained in $A$. Thus, for $A\subset \partial\Omega$, the boundary is then well defined by $\partial A:=\overline{A}\setminus\interior(A)$.

\item For a Borel set $A\subset\re^{n+1}$, we denote by $C(A)$ the space of continuous functions on $A$ and by $C_c(A)$ the subspace of $C(A)$ with compact support in $A$. Note that if $A$ is compact then $C(A)\equiv C_c(A)$.

\item For a Borel set $A\subset \partial\Omega$ with $0<\sigma(A)<\infty$, we write $\barint_{A}f\,d\sigma:=\sigma(A)^{-1}\int_A f\,d\sigma$.

\item We shall use the letter $I$ (and sometimes $J$) to denote a closed $(n+1)$-dimensional Euclidean cube with sides parallel to the co-ordinate axes, and we let $\ell(I)$ denote the side length of $I$. We use $Q$ to denote a dyadic ``cube'' on $E\subset\re^{n+1}$. The latter exists, given that $E$ is $\mathrm{AR}$ (cf. \cite{davidsemmes1}, \cite{christ}), and enjoy certain properties which we enumerate in Lemma \ref{dyadiccubes} below.
\end{list}
\medskip

\subsection{Some definitions}
\begin{definition}[\textbf{Corkscrew condition}]\label{defCS}
Following \cite{MR676988}, we say that an open set $\Omega\subset\re^{n+1}$ satisfies the ``Corkscrew condition'' if for some uniform constant $c\in(0,1)$ and for every surface ball $\Delta:=\Delta(x,r)=B(x,r)\cap\partial\Omega$ with $x\in\partial\Omega$ and $0<r<\diam(\partial\Omega)$, there is a ball $B(X_{\Delta},cr)\subset B(x,r)\cap\Omega$. The point $X_\Delta\in\Omega$ is called a ``corkscrew point'' relative to $\Delta$. Note that we may allow $r<C\diam(\partial\Omega)$ for any fixed $C$, simply by adjusting the constant $c$.
\end{definition}

\begin{definition}[\textbf{Harnack Chain condition}]\label{defHC}
Again following \cite{MR676988}, we say that $\Omega\subset\re^{n+1}$ satisfies the Harnack Chain condition if there is a uniform constant $C$ such that for every $\rho>0$, $\Theta\geq 1$, and every pair of points $X,X'\in\Omega$ with $\delta(X),\delta(X')\geq\rho$ and $|X-X'|<\Theta\rho$, there is a chain of open balls $B_1,\dots,B_N\subset\Omega$, $N\leq C(\Theta)$, with $X\in B_1$, $X'\in B_N$, $B_k\cap B_{k+1}\neq\emptyset$ and $C^{-1}\diam(B_k)\leq\dist(B_k,\partial\Omega)\leq C\diam(B_k)$. The chain of balls is called a ``Harnack Chain''.                                                                               \end{definition}

\begin{definition}[\textbf{Ahlfors regular}]
We say that a closed set $E\subset\re^{n+1}$ is $n$-dimensional $\mathrm{AR}$ (or simply $\mathrm{AR}$), if there is some uniform constant $C=C_{\mathrm{AR}}$ such that
$$
C^{-1}r^n\leq H^n(E\cap B(x,r))\leq C r^n,\quad 0<r< \diam(E),\quad x\in E.
$$
%When $E=\partial\Omega$, the boundary of a domain $\Omega$, we shall sometimes for convenience simply say that ``$\Omega$ has the $\mathrm{AR}$ property'' to mean that $\partial\Omega$ is $\mathrm{AR}$.
\end{definition}

\begin{definition}[\textbf{1-sided chord-arc domain}]\label{def-1CAD}
A connected open set $\Omega\subset\re^{n+1}$ is a ``1-sided chord-arc domain'' (1-sided CAD for short) if it satisfies the Corkscrew and Harnack Chain conditions and if $\partial\Omega$ is $\mathrm{AR}$.
\end{definition}

\begin{definition}[\textbf{Chord-arc domain}]\label{def-CAD}
A connected open set $\Omega\subset\re^{n+1}$ is a ``chord-arc domain'' (CAD for short) if it is a 1-sided CAD and moreover $\Omega$ satisfies the exterior Corkscrew condition (that is, the domain $\Omega{\rm ext}$ satisfies the Corkscrew condition).
\end{definition}

\begin{definition}
Given $E\subseteq\re^{n+1}$ and $n$-dimensional $\mathrm{AR}$ set, let 
$H^{1/2}(E)$ be the set of functions $f\in L^2(E)$ such that 
$$
\|f\|_{H^{1/2}(E)}:=\|f\|_{L^2(E)}+\bigg(\int_E\int_E\frac{|f(x)-f(y)|^2}{|x-y|^{n+1}}\,d\sigma(x)\,d\sigma(y)\bigg)^{1/2}<\infty.
$$
\end{definition}
\medskip

\subsection{Dyadic grids and sawtooths}

We give a lemma concerning the existence of a ``dyadic grid'':

\begin{lemma}[\textbf{Existence and properties of the ``dyadic grid''}, {\cite{davidsemmes1, davidsemmes2}, \cite{christ}}]\label{dyadiccubes}
Suppose that $E\subset\re^{n+1}$ is $n$-dimensional $\mathrm{AR}$. Then there exist constants $a_0>0$, $\eta>0$ and $C_1<\infty$ depending only on dimension and the $\mathrm{AR}$ constant, such that for each $k\in\mathbb{Z}$ there is a collection of Borel sets (``cubes'')
$$
\mathbb{D}_k:=\big\{Q_j^k\subset \partial\Omega:\:j\in\mathcal{J}_k\big\},
$$
where $\mathcal{J}_k$ denotes some (possibly finite) index set depending on $k$, satisfying:
\begin{list}{$(\theenumi)$}{\usecounter{enumi}\leftmargin=1cm \labelwidth=1cm \itemsep=0.1cm \topsep=.2cm \renewcommand{\theenumi}{\alph{enumi}}}
\item $E=\bigcup_jQ_j^k$ for each $k\in\mathbb{Z}$.
\item If $m\geq k$ then either $Q_i^m\subset Q_j^k$ or $Q_i^m\cap Q_j^k=\emptyset$.
\item For each $j,k\in\mathbb{Z}$ and each $m>k$, there is a unique $i\in\mathbb{Z}$ such that $Q_j^k\subset Q_i^m$.
\item $\diam(Q_j^k)\leq C_1\,2^{-k}$.
\item Each $Q_j^k$ contains some ``surface ball'' $\Delta(x_j^k,a_02^{-k})=B(x_j^k,a_02^{-k})\cap E$.
\item $H^n\big(\big\{x\in Q_j^k:\,\dist(x,E\setminus Q_j^k)\leq\tau 2^{-k}\big\}\big)\leq C_1\tau^\eta H^n(Q_j^k)$, for all $j,k\in\mathbb{Z}$ and for all $\tau\in(0,a_0)$.
\end{list}
\end{lemma}

A few remarks are in order concerning this lemma.
\begin{list}{$\bullet$}{\leftmargin=0.4cm  \itemsep=0.2cm}

\item In the setting of a general space of homogeneous type, this lemma has been proved by Christ
\cite{christ}, with the dyadic parameter $1/2$ replaced by some constant $\delta \in (0,1)$.
In fact, one may always take $\delta = 1/2$ (cf.  \cite[Proof of Proposition 2.12]{HMMM}). In the presence of the Ahlfors regularity property, the result already appears in \cite{davidsemmes1, davidsemmes2}.

\item We shall denote by $\mathbb{D}(E)$ the collection of all relevant $Q_j^k$, i.e.,
$$
\mathbb{D}(E):=\bigcup_k\mathbb{D}_k,
$$
where, if $\diam(E)$ is finite, the union runs over those $k\in\mathbb{Z}$ such that $2^{-k}\lesssim\diam(E)$.
\item For a dyadic cube $Q\in\mathbb{D}_k$, we shall set $\ell(Q)=2^{-k}$, and we shall refer to this quantity as the ``length'' of $Q$. It is clear that $\ell(Q)\approx\diam(Q)$. Also, for $Q\in\mathbb{D}(E)$ we will set $k(Q)=k$ if $Q\in\mathbb{D}_k$.

\item Properties $(d)$ and $(e)$ imply that for each cube $Q\in\mathbb{D}$, there is a point $x_Q\in E$, a Euclidean ball $B(x_Q,r_Q)$ and a surface ball $\Delta(x_Q,r_Q):=B(x_Q,r_Q)\cap E$ such that $c\ell(Q)\leq r_Q\leq\ell(Q)$, for some uniform constant $c>0$, and
    \begin{equation}\label{deltaQ}
    \Delta(x_Q,2r_Q)\subset Q\subset\Delta(x_Q,Cr_Q)
    \end{equation}
    for some uniform constant $C>1$. We shall denote these balls and surface balls by
    \begin{equation}\label{deltaQ2}
    B_Q:=B(x_Q,r_Q),\qquad\Delta_Q:=\Delta(x_Q,r_Q),
    \end{equation}
    \begin{equation}\label{deltaQ3}
    \widetilde{B}_Q:=B(x_Q,Cr_Q),\qquad\widetilde{\Delta}_Q:=\Delta(x_Q,Cr_Q),
    \end{equation}
    and we shall refer to the point $x_Q$ as the ``center'' of $Q$.

\item Let $\Omega\subset\re^{n+1}$ be an open set satisfying the Corkscrew condition and such that $\partial\Omega$ is $\mathrm{AR}$. Given $Q\in\mathbb{D}(\partial\Omega)$ we define the ``corkscrew point relative to $Q$'' as $X_Q:=X_{\Delta_Q}$. We note that
    $$
    \delta(X_Q)\approx\dist(X_Q,Q)\approx\diam(Q).
    $$
\end{list}
\medskip

Following \cite[Section 3]{hofmartell} we next introduce the notion of ``Carleson region'' and ``discretized sawtooth''. Given a cube  $Q\in\mathbb{D}(E)$, the ``discretized Carleson region'' $\mathbb{D}_Q$ relative to $Q$ is defined by
$$
\mathbb{D}_Q:=\big\{Q'\in\mathbb{D}(E):\,Q'\subset Q\big\}.
$$
Let $\mathcal{F}=\{Q_i\}\subset\mathbb{D}(E)$ be a family of disjoint cubes. The ``global discretized sawtooth'' relative to $\mathcal{F}$ is the collection of cubes $Q\in\mathbb{D}(E)$ that are not contained in any $Q_i\in\mathcal{F}$, that is,
$$
\mathbb{D}_\mathcal{F}:=\mathbb{D}(E)\setminus\bigcup_{Q_i\in\mathcal{F}}\mathbb{D}_{Q_i}.
$$
For a given $Q\in\mathbb{D}(E)$, the ``local discretized sawtooth'' relative to $\mathcal{F}$ is the collection of cubes in $\mathbb{D}_Q$ that are not contained in any $Q_i\in\mathcal{F}$ or, equivalently,
$$
\mathbb{D}_{\mathcal{F},Q}:=\mathbb{D}_{Q}\setminus\bigcup_{Q_i\in\mathcal{F}}\mathbb{D}_{Q_i}=\mathbb{D}_\mathcal{F}\cap\mathbb{D}_Q.
$$

We also introduce the ``geometric'' Carleson regions and sawtooths. In the sequel, $\Omega\subset\re^{n+1}$ ($n\geq 2$) will be a 1-sided $\mathrm{CAD}$. Given $Q\in\mathbb{D}(\partial\Omega)$ we want to define some associated regions which inherit the good properties of $\Omega$. Let $\mathcal{W}=\mathcal{W}(\Omega)$ denote a collection of (closed) dyadic Whitney cubes of $\Omega\subset\re^{n+1}$, so that the cubes in $\mathcal{W}$ form a pairwise non-overlapping covering of $\Omega$, which satisfy
\begin{equation}\label{constwhitney}
4\diam(I)\leq\dist(4I,\partial\Omega)\leq\dist(I,\partial\Omega)\leq 40\diam(I),\qquad\forall I\in\mathcal{W},
\end{equation}
and
$$
\diam(I_1)\approx\diam(I_2),\,\text{ whenever }I_1\text{ and }I_2\text{ touch}.
$$
Let $X(I)$ denote the center of $I$, let $\ell(I)$ denote the sidelength of $I$, and write $k=k_I$ if $\ell(I)=2^{-k}$.

Given $0<\lambda<1$ and $I\in\mathcal{W}$ we write $I^*=(1+\lambda)I$ for the ``fattening'' of $I$. By taking $\lambda$ small enough, we can arrange matters, so that, first, $\dist(I^*,J^*)\approx\dist(I,J)$ for every $I,J\in\mathcal{W}$, and secondly, $I^*$ meets $J^*$ if and only if $\partial I$ meets $\partial J$ (the fattening thus ensures overlap of $I^*$ and $J^*$ for any pair $I,J\in\mathcal{W}$ whose boundaries touch, so that the Harnack Chain property then holds locally in $I^*\cup J^*$, with constants depending upon $\lambda$). By picking $\lambda$ sufficiently small, say $0<\lambda<\lambda_0$, we may also suppose that there is $\tau\in(1/2,1)$ such that for distinct $I,J\in\mathcal{W}$, we have that $\tau J\cap I^*=\emptyset$. In what follows we will need to work with dilations $I^{**}=(1+2\lambda)I$ or $I^{***}=(1+4\lambda)I$, and in order to ensure that the same properties hold we further assume that $0<\lambda<\lambda_0/4$.

For every $Q\in\mathbb{D}(\partial\Omega)$ we can construct a family $\mathcal{W}_Q^*\subset\mathcal{W}(\Omega)$, and define
$$
U_Q:=\bigcup_{I\in\mathcal{W}_Q^*}I^*,
$$
satisfying the following properties: $X_Q\in U_Q$ and there are uniform constants $k^*$ and $K_0$ such that
\begin{gather*}
k(Q)-k^*\leq k_I\leq k(Q)+k^*,\quad\forall I\in\mathcal{W}_Q^*,
\\[4pt]
X(I)\rightarrow_{U_Q} X_Q,\quad\forall I\in\mathcal{W}_Q^*,
\\[4pt]
\dist(I,Q)\leq K_0 2^{-k(Q)},\quad\forall I\in\mathcal{W}_Q^*.
\end{gather*}
Here, $X(I)\rightarrow_{U_Q} X_Q$ means that the interior of $U_Q$ contains all balls in a Harnack Chain (in $\Omega$) connecting $X(I)$ to $X_Q$, and moreover, for any point $Z$ contained in any ball in the Harnack Chain, we have $\dist(Z,\partial\Omega)\approx\dist(Z,\Omega\setminus U_Q)$ with uniform control of the implicit constants. The constants $k^*, K_0$ and the implicit constants in the condition $X(I)\rightarrow_{U_Q} X_Q$, depend on at most allowable parameters and on $\lambda$. Moreover, given $I\in\mathcal{W}(\Omega)$ we have that $I\in\mathcal{W}_{Q_I}^*$, where $Q_I\in\mathbb{D}(\partial\Omega)$ satisfies $\ell(Q_I)=\ell(I)$, and contains any fixed $\widehat{y}\in\partial\Omega$ such that $\dist(I,\partial\Omega)=\dist(I,\widehat{y})$. The reader is referred to \cite{hofmartell} for full details.

For a given $Q\in\mathbb{D}(\partial\Omega)$, the ``Carleson box'' relative to $Q$ is defined by
$$
T_Q:=\interior\bigg(\bigcup_{Q'\in\mathbb{D}_Q}U_{Q'}\bigg).
$$
For a given family $\mathcal{F}=\{Q_i\}$ of pairwise disjoint cubes and a given $Q\in\mathbb{D}(\partial\Omega)$, we define the ``local sawtooth region'' relative to $\mathcal{F}$ by
\begin{equation}\label{defomegafq}
\Omega_{\mathcal{F},Q}=\interior\bigg(\bigcup_{Q'\in\mathbb{D}_{\mathcal{F},Q}}U_{Q'}\bigg)=\interior\bigg(\bigcup_{I\in\mathcal{W}_{\mathcal{F},Q}}I^*\bigg),
\end{equation}
where $\mathcal{W}_{\mathcal{F},Q}:=\bigcup_{Q'\in\mathbb{D}_{\mathcal{F},Q}}\mathcal{W}_Q^*$. Analogously, we can slightly fatten the Whitney boxes and use $I^{**}$ to define new fattened Whitney regions and sawtooth domains. More precisely, for every $Q\in\mathbb{D}(\partial\Omega)$,
$$
T_Q^*:=\interior\bigg(\bigcup_{Q'\in\mathbb{D}_Q}U_{Q'}^*\bigg),\qquad\Omega^*_{\mathcal{F},Q}:=\interior\bigg(\bigcup_{Q'\in\mathbb{D}_Q}U_{Q'}^*\bigg), \qquad U_{Q}^*:=\bigcup_{I\in\mathcal{W}_Q^*}I^{**}.
$$
Similarly, we can define $T_Q^{**}$, $\Omega^{**}_{\mathcal{F},Q}$ and $U^{**}_{Q}$ by using $I^{***}$ in place of $I^{**}$.

To define  the ``Carleson box'' $T_\Delta$ associated to a surface ball $\Delta=\Delta(x,r)$, let $k(\Delta)$ denote the unique $k\in\mathbb{Z}$ such that $2^{-k-1}<200r\leq 2^{-k}$, and set
$$
\mathbb{D}^{\Delta}:=\big\{Q\in\mathbb{D}_{k(\Delta)}:\:Q\cap 2\Delta\neq\emptyset\big\}.
$$
We then set
$$
T_{\Delta}:=\interior\bigg(\bigcup_{Q\in\mathbb{D}^\Delta}\overline{T_Q}\bigg).
$$
We can also consider slight dilations of $T_\Delta$ given by
$$
T_{\Delta}^*:=\interior\bigg(\bigcup_{Q\in\mathbb{D}^\Delta}\overline{T_Q^*}\bigg),\qquad T_{\Delta}^{**}:=\interior\bigg(\bigcup_{Q\in\mathbb{D}^\Delta}\overline{T_Q^{**}}\bigg).
$$

Following \cite{hofmartell}, one can easily see that there exist constants $0<\kappa_1<1$ and $\kappa_0\geq 2C$ (with $C$ the constant in \eqref{deltaQ3}), depending only on the allowable parameters, so that
\begin{gather}\label{definicionkappa12}
\kappa_1B_Q\cap\Omega\subset T_Q\subset T_Q^*\subset T_Q^{**}\subset \overline{T_Q^{**}}\subset\kappa_0B_Q\cap\overline{\Omega}=:\tfrac{1}{2}B_Q^*\cap\overline{\Omega},
\\[6pt]
\label{definicionkappa0}
\tfrac{5}{4}B_\Delta\cap\Omega\subset T_\Delta\subset T_\Delta^*\subset T_\Delta^{**}\subset\overline{T_\Delta^{**}}\subset\kappa_0B_\Delta\cap\overline{\Omega}=:\tfrac{1}{2}B_\Delta^*\cap\overline{\Omega},
\end{gather}
and also
\begin{equation}\label{propQ0}
Q\subset\kappa_0B_\Delta\cap\partial\Omega=\tfrac{1}{2}B_\Delta^*\cap\partial\Omega=:\tfrac{1}{2}\Delta^*,\qquad\forall\,Q\in\mathbb{D}^{\Delta},
\end{equation}
where $B_Q$ is defined as in \eqref{deltaQ2}, $\Delta=\Delta(x,r)$ with $x\in\partial\Omega$, $0<r<\diam(\partial \Omega)$, and $B_{\Delta}=B(x,r)$ is so that $\Delta=B_\Delta\cap\partial\Omega$.
\medskip

\subsection{\texorpdfstring{$A_\infty$}{Ainf} weights and Carleson measures}
Throughout this section, $E \subset\re^{n+1}$ will be an $n$-dimensional $\mathrm{AR}$ set and $\sigma=H^{n}\rest{E}$.

\begin{definition}[$A_\infty$ and $A_{\infty}^{\rm dyadic}$]
Given a surface ball $\Delta_0=B_0\cap E$, with $B_0=B(x_0,r_0)$, $x_0\in E$, $0<r<\diam(E)$, a Borel measure $\omega$ defined on $\Delta_0$ is said to belong to $A_\infty(\Delta_0)$ if there exist constants $0<\alpha,\beta<1$ such that for every surface ball $\Delta=B\cap E$ centered at $E$ with $B\subset B_0$, and for every Borel set $F\subset\Delta$, we have that
$$
\frac{\sigma(F)}{\sigma(\Delta)}>\alpha\implies
\frac{\omega(F)}{\omega(\Delta)}>\beta.
$$

Given $Q_0\in\mathbb{D}(E)$, a Borel measure $\omega$ defined on $Q_0$ is said to belong to $A_\infty^{\rm dyadic}(Q_0)$ if there exist constants $0<\alpha,\beta<1$ such that for every $Q\in\mathbb{D}_{Q_0}$ and for every Borel set $F\subset Q$, we have that
$$
\frac{\sigma(F)}{\sigma(Q)}>\alpha\implies
\frac{\omega(F)}{\omega(Q)}>\beta.
$$
\end{definition}

It is well known (see \cite{MR807149}, \cite{coifman1974}) that since $\sigma$ is a doubling measure (recall that $E$ satisfies the $\mathrm{AR}$ condition), $\omega\in A_\infty(\Delta_0)$ if and only if $\omega\ll\sigma$ in $\Delta_0$  and there exists $1<p<\infty$ such that $\omega\in RH_p(\Delta_0)$, that is, there is a constant $C_1>1$ such that
$$
\bigg(\barint_{\Delta} k(x)^{p}\,d\sigma(x)\bigg)^{\frac{1}{p}}\leq
C_1\barint_{\Delta}k(x)\,d\sigma(x),
$$
for every $\Delta=B\cap E$ centered at $E$ with $B\subset B_0$, and where $k=d\omega/d\sigma$ is the Radon-Nikodym derivative.  Analogously, $\omega\in A_\infty^{\rm dyadic}(Q_0)$
 if and only if $\omega\ll\sigma$ in $Q_0$  and there exists $1<p<\infty$ such that $\omega\in RH_p^{\rm dyadic}(Q_0)$, that is,  there is a constant $C_1>1$ such that
$$
\bigg(\barint_{Q} k(x)^{p}\,d\sigma(x)\bigg)^{\frac{1}{p}}\leq
C_1\barint_{Q}k(x)\,d\sigma(x),
$$
for every  $Q\in\mathbb{D}_{Q_0}$, where again $k=d\omega/d\sigma$.

%
%
%
%
%
%
%Analogously, we can define a dyadic version of the $A_\infty$ condition. Let now $\omega$ be a Borel measure on $Q_0\in\mathbb{D}(E)$, we say that $\omega$ is dyadically doubling on $Q_0$ if there exists $C_\omega> 1$ such that
%$\omega(Q)\leq C_\omega\omega(Q')$ for every $Q\in\mathbb{D}_{Q_0}$ and every
%$Q'\in\mathbb{D}_Q$ such that $\ell(Q')=\ell(Q)/2$. Moreover, if there exist constants
%$0<\alpha,\beta<1$ such that for every $Q\in\mathbb{D}_{Q_0}$ and every Borel set $F\subset Q$, we have that
%$$
%\frac{\sigma(F)}{\sigma(Q)}>\alpha\implies
%\frac{\omega(F)}{\omega(Q)}>\beta,
%$$
%we say that $\omega\in A_{\infty}^{\rm dyadic}(Q_0)$.
\medskip

Fix $Q_0\in\mathbb{D}(E)$. For each $\mathcal{F}=\{Q_i\}\subset\mathbb{D}_{Q_0}$, a family of pairwise disjoint dyadic cubes, and each $f$ locally integrable, we define
$$
\mathcal{P}_{\mathcal{F}} f(x)=f(x)\mathbf{1}_{E\setminus\big(\bigcup_i
Q_i\big)}(x)+\sum_{Q_i\in\mathcal{F}}\Big(\barint_{Q_i}f(y)\,d\sigma(y)\Big)\mathbf{1}_{Q_i}(x).
$$
If $\omega$ is a non-negative Borel measure on $Q_0$, we may naturally then define the measure $\mathcal{P}_\mathcal{F}\omega$ as $\mathcal{P}_{\mathcal{F}}\omega(F)=\int_{E}\mathcal{P}_{\mathcal{F}}\mathbf{1}_F\,d\omega$, that is,
\begin{equation}\label{defprojection}
\mathcal{P}_{\mathcal{F}}\omega(F)=\omega\Big(F\setminus\bigcup_{Q_i\in\mathcal{F}}Q_i\Big)+\sum_{Q_i\in\mathcal{F}}\frac{\sigma(F\cap Q_i)}{\sigma(Q_i)}\omega(Q_i),
\end{equation}
for each Borel set $F\subset Q_0$.

The next result follows easily  by combining the arguments in  \cite[Lemma B.1]{hofmartell} and \cite[Lemma 4.1]{MR2655385}

\begin{lemma}\label{lemm_w-Pw-:properties}
Let $\omega$ be a non-negative Borel measure on $Q_0\in\mathbb{D}(E)$.

\begin{list}{$(\theenumi)$}{\usecounter{enumi}\leftmargin=1cm \labelwidth=1cm \itemsep=0.1cm \topsep=.2cm \renewcommand{\theenumi}{\alph{enumi}}}

\item If $\omega$ is dyadically doubling on $Q_0$ ---that is, there exists $C_\omega> 1$ such that
$\omega(Q)\leq C_\omega\,\omega(Q')$ for every $Q\in\mathbb{D}_{Q_0}$ and every $Q'\in\mathbb{D}_Q$ such that $\ell(Q')=\ell(Q)/2$--- then $\mathcal{P}_{\mathcal{F}}\omega$  is dyadically doubling on $Q_0$.

\item If  $\omega\in A_\infty^{\rm dyadic}(Q_0)$  then $\mathcal{P}_{\mathcal{F}}\omega\in A_\infty^{\rm dyadic}(Q_0)$.
\end{list}

\end{lemma}

Let $\{\gamma_Q\}_{Q\in\mathbb{D}(E)}$ be a sequence of non-negative real numbers. We define the ``measure'' $\mathfrak{m}$ (acting on collections of dyadic cubes) by
$$
\mathfrak{m}(\mathbb{D}')=\sum_{Q\in\mathbb{D'}}\gamma_Q,\qquad \mathbb{D'}\subset\mathbb{D}(E).
$$
Let $Q_0\in\mathbb{D}(E)$, we say that $\mathfrak{m}$ is a discrete ``Carleson measure'' on $Q_0$ (with respect to $\sigma$) or, equivalently, $\mathfrak{m}\in\mathcal{C}(Q_0)$ if
$$
\|\mathfrak{m}\|_{\mathcal{C}(Q_0)}:=\sup_{Q\in\mathbb{D}_{Q_0}}\frac{\mathfrak{m}(\mathbb{D}_{Q})}{\sigma(Q)}<\infty.
$$
Given $\mathcal{F}=\{Q_i\}\subset\mathbb{D}_{Q_0}$, a family of pairwise disjoint dyadic cubes, we define $\mathfrak{m}_\mathcal{F}$ by
$$
\mathfrak{m}_{\mathcal{F}}(\mathbb{D}')=\mathfrak{m}(\mathbb{D}'\cap\mathbb{D}_\mathcal{F})=\sum_{Q\in\mathbb{D}'\cap\mathbb{D}_\mathcal{F}}\gamma_Q,\qquad \mathbb{D}'\subset\mathbb{D}_{Q_0}.
$$
Equivalently, the measure $\mathfrak{m}_\mathcal{F}$ is given by the sequence $\{\gamma_{\mathcal{F},Q}\}_{Q\in\mathbb{D}_{Q_0}}$, where
\begin{equation}\label{gammauxiliar}
\gamma_{\mathcal{F},Q}=\left\{
                   \begin{array}{ll}
                     \gamma_Q & \hbox{ if $Q\in\mathbb{D}_{\mathcal{F},Q_0}$,} \\[5pt]
                     0 & \hbox{ if $Q\in\mathbb{D}_{Q_0}\setminus\mathbb{D}_{\mathcal{F},Q_0}$.}
                   \end{array}
                 \right.
\end{equation}
\medskip

\begin{lemma}[{\cite[Lemma 8.5]{hofmartell}}]\label{extrapolation}
Suppose that $E\subset\re^{n+1}$ is $n$-dimensional $\mathrm{AR}$. Fix $Q_0\in\mathbb{D}(E)$, let $\sigma$, $\omega$ be a pair of non-negative dyadically doubling Borel measures on $Q_0$, and let $\mathfrak{m}$ be a discrete Carleson measure with respect to $\sigma$, with
$$
\|\mathfrak{m}\|_{\mathcal{C}(Q_0)}\leq M_0.
$$
Suppose that there exists $\gamma>0$ such that for every $Q\in\mathbb{D}_{Q_0}$ and every family of pairwise disjoint dyadic cubes $\mathcal{F}=\{Q_i\}\subset\mathbb{D}_Q$ verifying
$$
\|\mathfrak{m}_\mathcal{F}\|_{\mathcal{C}(Q)}=\sup_{Q'\in\mathbb{D}_{Q}}\frac{\mathfrak{m}(\mathbb{D}_{\mathcal{F},Q'})}{\sigma(Q')}\leq\gamma,
$$
we have that $\mathcal{P}_{\mathcal{F}}\omega$ satisfies the following property:
$$
\forall\varepsilon\in(0,1)\quad\exists\,C_\varepsilon>1\text{ such that }\Big(F\subset Q,\quad\frac{\sigma(F)}{\sigma(Q)}\geq\varepsilon\implies\frac{\mathcal{P}_{\mathcal{F}}\omega(F)}{\mathcal{P}_{\mathcal{F}}\omega(Q)}\geq\frac{1}{C_\varepsilon} \Big).
$$
Then, there exist $\eta_0\in(0,1)$ and $C_0<\infty$ such that, for every $Q\in\mathbb{D}_{Q_0}$
$$
F\subset Q,\quad\frac{\sigma(F)}{\sigma(Q)}\geq 1-\eta_0\implies\frac{\omega(F)}{\omega(Q)}\geq\frac{1}{C_0}.
$$
In other words, $\omega\in A_{\infty}^{\rm dyadic}(Q_0)$.
\end{lemma}
\medskip

\subsection{PDE estimates}
Next, we recall several facts concerning elliptic measure and Green functions. For our first results we will only assume that $\Omega\subset\re^{n+1}$ is an open set, not necessarily connected, with $\partial\Omega$ satisfying the $\mathrm{AR}$ property. Later we will focus on the case where $\Omega$ is a 1-sided $\mathrm{CAD}$.

\begin{definition}\label{ellipticoperator}
We say that $L$ is a real symmetric elliptic operator if $Lu=-\div(A\nabla u)$, with $A(X)=(a_{i,j}(X))_{i,j=1}^{n+1}$ being a real symmetric matrix such that $a_{i,j}\in L^{\infty}(\Omega)$ and there exists $\Lambda\geq 1$ such that the following uniform ellipticity condition holds
\begin{equation}\label{eq:elliptic}
\Lambda^{-1}|\xi|^2\leq A(X)\xi\cdot\xi\leq\Lambda|\xi|^2,\qquad\xi\in\re^{n+1},\quad\text{for a.e. }X\in\Omega.
\end{equation}
\end{definition}

In what follows we will only be working with this kind of operators, we will refer to them as ``elliptic operators'' for the sake of simplicity. Associated with $L$ one can construct an elliptic measure $\{\omega_L^X\}_{X\in\Omega}$ and a Green function $G_L$ (see \cite{hofmartellgeneral} for full details). Sometimes, in order to emphasize the dependence on $\Omega$, we will write $\omega_{L,\Omega}$ and $G_{L,\Omega}$.

\begin{lemma}\label{bourgain}
Suppose that $\Omega\subset\re^{n+1}$ is an open set such that $\partial\Omega$ satisfies the $\mathrm{AR}$ property. Let $L$ be an elliptic operator, there exist constants $c_1<1$ and $C_1>1$ (depending only on the $\mathrm{AR}$ constant and on the ellipticity of $L$) such that for every $x\in\partial\Omega$ and every $0<r<\diam(\partial\Omega)$, we have
$$
\omega_L^Y(\Delta(x,r))\geq \frac{1}{C_1},\qquad\forall\, Y\in B(x,c_1r)\cap\Omega.
$$
\end{lemma}

We refer the reader to \cite[Lemma 1]{bou} for the proof in the harmonic case and to \cite{hofmartellgeneral} for general elliptic operators. See also \cite[Theorem 6.18]{HKM} and \cite[Section 3]{ZHAO}.
\medskip

A proof of the next lemma may be found in \cite{hofmartellgeneral}. We note that, in particular, the $\mathrm{AR}$ hypothesis implies that $\partial\Omega$ satisfies the Capacity Density Condition, hence $\partial\Omega$ is Wiener regular at every point (see \cite[Lemma 3.27]{hoflemartellnystrom}).

\begin{lemma}
Suppose that $\Omega\subset\re^{n+1}$ is an open set such that $\partial\Omega$ satisfies the $\mathrm{AR}$ property. Given an elliptic operator $L$, there exist $C>1$ (depending only on dimension and on the ellipticity of $L$) and $c_\theta>0$ (depending on the above parameters and on $\theta\in (0,1)$) such that $G_L$,  the Green function associated with $L$, satisfies
\begin{gather}\label{sizestimate}
G_L(X,Y)\leq C|X-Y|^{1-n};
\\[0.15cm]
c_\theta|X-Y|^{1-n}\leq G_L(X,Y),\quad\text{if }\,|X-Y|\leq\theta\delta(X),\quad\theta\in(0,1);
\\[0.15cm]
G_L(X,\cdot)\in C\big(\overline{\Omega}\setminus\{X\}\big)\quad\text{and}\quad G(X,\cdot)\rest{\partial\Omega}\equiv 0\quad\forall X\in\Omega;
\\[0.15cm]
G_L(X,Y)\geq 0,\quad\forall X,Y\in\Omega,\quad X\neq Y;
\\[0.15cm]
G_L(X,Y)=G_L(Y,X),\quad\forall X,Y\in\Omega,\quad X\neq Y;
\end{gather}
and for every $\varphi\in C_c^{\infty}(\re^{n+1})$ we have that
$$
\int_{\partial\Omega}\varphi\,d\omega_L^X-\varphi(X)=-\iint_{\Omega}A(Y)\nabla_YG_L(Y,X)\cdot\nabla\varphi(Y)\,dY,\quad\text{for a.e. }X\in\Omega.
$$
\end{lemma}
\medskip

\begin{remark}
 If we also assume that $\Omega$ is bounded, following \cite{hofmartellgeneral} we know that the Green function $G_L$ coincides with the one constructed in \cite{gruterwidman}. Consequently, for each $X\in\Omega$ and $0<r<\delta(X)$, there holds
\begin{equation}\label{obsgreen1}
G_L(X,\cdot)\in W^{1,2}(\Omega\setminus B(X,r))\cap W_0^{1,1}(\Omega).
\end{equation}
Moreover, for every $\varphi\in C_c^{\infty}(\Omega)$ such that $\varphi\equiv 1$ in $B(X,r)$ with $0<r<\delta(X)$, we have that
\begin{equation}\label{obsgreen2}
(1-\varphi)G_L(X,\cdot)\in W_0^{1,2}(\Omega).
\end{equation}
\end{remark}
\medskip

\begin{lemma}[\cite{hofmartellgeneral}]\label{proppde}
Suppose that $\Omega\subset\re^{n+1}$ is a 1-sided $\mathrm{CAD}$. Let $L$ and $L_1$ be elliptic operators, there exist $C_1\ge 1$ (depending only on dimension, the 1-sided $\mathrm{CAD}$ constants and the ellipticity of $L$) and $C_2\ge 1$ (depending on the above parameters and on the ellipticity of $L_1$), such that for every $B_0=B(x_0,r_0)$ with $x_0\in\partial\Omega$, $0<r_0<\diam(\partial\Omega)$, and $\Delta_0=B_0\cap\partial\Omega$ we have the following properties:
\begin{list}{$(\theenumi)$}{\usecounter{enumi}\leftmargin=1cm \labelwidth=1cm \itemsep=0.1cm \topsep=.2cm \renewcommand{\theenumi}{\alph{enumi}}}

\item If $B=B(x,r)$ with $x\in\partial\Omega$ and $\Delta=B\cap\partial\Omega$ is such that $2B\subset B_0$, then for all $X\in\Omega\setminus B_0$ we have that
    $$
\frac{1}{C_1}\omega_L^X(\Delta)\leq r^{n-1} G_L(X_\Delta,X)\leq C_1\omega_L^X(\Delta).
    $$

\item  If $X\in\Omega\setminus 4B_0$, then
    $$
    \omega_{L}^X(2\Delta_0)\leq C_1\omega_{L}^X(\Delta_0).
    $$

\item  If $B=B(x,r)$ with $x\in\partial\Omega$ and $\Delta:=B\cap\partial\Omega$ is such that $B\subset B_0$, then for every $X\in\Omega\setminus 2\kappa_0B_0$ with $\kappa_0$ as in \eqref{definicionkappa0}, we have that
$$
\frac{1}{C_1}\omega_L^{X_{\Delta_0}}(\Delta)\leq \frac{\omega_L^X(\Delta)}{\omega_L^X(\Delta_0)}\leq C_1\omega_L^{X_{\Delta_0}}(\Delta).
$$
Moreover, if we also suppose that $\omega_L\ll\sigma$, then
$$
\frac{1}{C_1}k_L^{X_{\Delta_0}}(y)\leq \frac{k_L^X(y)}{\omega_L^X(\Delta_0)}\leq C_1 k_L^{X_{\Delta_0}}(y),\quad\text{ for }\sigma\text{-a.e. }y\in\Delta_0.
$$

\item  If $B=B(x,r)$ with $x\in\Delta_0$, $0<r<r_0/4$ and $\Delta=B\cap\partial\Omega$, then we have that
$$
\frac{1}{C_1}\omega_{L,\Omega}^{X_{\Delta}}(F)\leq\omega_{L,T_{\Delta_0}}^{X_{\Delta}}(F)\leq C_1\omega_{L,\Omega}^{X_{\Delta}}(F),\quad\text{ for every Borel set }F\subset\Delta.
$$
This implies that $\omega_{L,\Omega}\ll\sigma$ in $\Delta$ if and only if $\omega_{L,T_{\Delta_0}}\ll\sigma$ in $\Delta$ and, in such a case,
$$
\frac{1}{C_1}k_{L,\Omega}^{X_\Delta}(y)\leq k_{L,T_{\Delta_0}}^{X_\Delta}(y)\leq C_1k_{L,\Omega}^{X_{\Delta}}(y),\quad\text{ for }\sigma\text{-a.e. }y\in\Delta.
$$

\item  If $L\equiv L_1$ in $B(x_0,2\kappa_0r_0)\cap\Omega$ with $\kappa_0$ as in \eqref{definicionkappa0}, then
$$
\frac{1}{C_2}\omega_{L_1}^{X_{\Delta_0}}(F)\leq\omega_{L}^{X_{\Delta_0}}(F)\leq C_2\omega_{L_1}^{X_{\Delta_0}}(F),\quad\text{ for every Borel set }F\subset\Delta_0.
$$
This implies that $\omega_{L}\ll\sigma$ in $\Delta_0$ if and only if $\omega_{L_1}\ll\sigma$ in $\Delta_0$ and, in such a case,
$$
\frac{1}{C_2}k_{L_1}^{X_{\Delta_0}}(y)\leq k_{L}^{X_{\Delta_0}}(y)\leq C_2k_{L_1}^{X_{\Delta_0}}(y),\quad\text{ for }\sigma\text{-a.e. }y\in\Delta_0.
$$
\end{list}
\end{lemma}

\begin{remark}\label{dyadicchangepole}
As a consequence of Lemma \hyperref[proppde]{\ref*{proppde}(c)}, one can see that if  $\omega_L\ll\sigma$, there exists $C\ge 1$ (depending only on dimension, the 1-sided $\mathrm{CAD}$ constants and the ellipticity of $L$) such that for every $Q_0\in\mathbb{D}(\partial\Omega)$ and every $Q\in\mathbb{D}_{Q_0}$ we have that
$$
\frac{1}{C}k_L^{X_{Q}}(y)\leq \frac{k_L^{X_{Q_0}}(y)}{\omega_L^{X_{Q_0}}(Q)}\leq C k_L^{X_{Q}}(y),\quad\text{ for }\sigma\text{-a.e. }y\in Q.
$$
\end{remark}
\medskip

\begin{lemma}[\cite{hofmartellgeneral}]\label{sawtoothlemma}
Suppose that $\Omega\subset\re^{n+1}$ is a 1-sided $\mathrm{CAD}$. Given $Q_0\in\mathbb{D}(\partial\Omega)$ and $\mathcal{F}=\{Q_i\}\subset\mathbb{D}_{Q_0}$, a family of pairwise disjoint dyadic cubes, let $\mathcal{P}_{\mathcal{F}}$ be the corresponding projection operator defined in \eqref{defprojection}. Given an elliptic operator $L$, we denote by $\omega_L=\omega_{L,\Omega}^{A_{Q_0}}$ and $\omega_{L,\star}=\omega_{L,\Omega_{\mathcal{F},Q_0}}^{A_{Q_0}}$ the elliptic measures of $L$ with respect to $\Omega$ and $\Omega_{\mathcal{F},Q_0}$ with fixed pole at the corkscrew point $A_{Q_0}\in\Omega_{\mathcal{F},Q_0}$ (cf. \cite[Proposition 6.4]{hofmartell}). Let $\nu_L=\nu_L^{A_{Q_0}}$ be the measure defined by
\begin{equation}\label{eq:def-nu}
\nu_L(F)=\omega_{L,\star}\Big(F\setminus\bigcup_{Q_i\in\mathcal{F}}Q_i\Big)+\sum_{Q_i\in\mathcal{F}}\frac{\omega_L(F\cap Q_i)}{\omega_L(Q_i)}\omega_{L,\star}(P_i),\qquad F\subset Q_0,
\end{equation}
where $P_i$ is the cube produced by \cite[Proposition 6.7]{hofmartell}. Then $\mathcal{P}_{\mathcal{F}}\nu_L$ depends only on $\omega_{L,\star}$ and not on $\omega_L$. More precisely,
\begin{equation}\label{eq:def-nu:P}
\mathcal{P}_{\mathcal{F}}\nu_L(F)=\omega_{L,\star}\Big(F\setminus\bigcup_{Q_i\in\mathcal{F}}Q_i\Big)+\sum_{Q_i\in\mathcal{F}}\frac{\sigma(F\cap Q_i)}{\sigma(Q_i)}\omega_{L,\star}(P_i),\qquad F\subset Q_0.
\end{equation}
Moreover, there exists $\theta>0$ such that for all $Q\in\mathbb{D}_{Q_0}$ and all $F\subset Q$, we have
\begin{equation}\label{ainfsawtooth}
\bigg(\frac{\mathcal{P}_{\mathcal{F}}\omega_L(F)}{\mathcal{P}_{\mathcal{F}}\omega_L(Q)}\bigg)^\theta\lesssim\frac{\mathcal{P}_{\mathcal{F}}\nu_L(F)}{\mathcal{P}_{\mathcal{F}}\nu_L(Q)}\lesssim
\frac{\mathcal{P}_{\mathcal{F}}\omega_L(F)}{\mathcal{P}_{\mathcal{F}}\omega_L(Q)}.
\end{equation}
\end{lemma}
\medskip

\begin{definition}\label{defrhp}
Suppose that $\Omega\subset\re^{n+1}$ is a 1-sided $\mathrm{CAD}$, let $L$ be an elliptic operator and let $1<p<\infty$. We say that $\omega_L\in RH_p(\partial\Omega)$ if $\omega_L\ll\sigma$ and
$k_L^{X_{\Delta_0}}\in RH_p(\Delta_0)$ uniformly in $\Delta_0$ for every surface ball $\Delta_0\subset\partial\Omega$. That is,
there exists $C \ge 1$ such that for every $B_0:=B(x_0,r_0)$ with $x_0\in\partial\Omega$ and $0<r_0<\diam(\partial\Omega)$, and for every $B=B(x,r)\subset B_0$ with $x\in\partial\Omega$, we have that
$$
\bigg(\barint_{\Delta}k_L^{X_{\Delta_0}}(y)^p\,d\sigma(y)\bigg)^{1/p}\leq
C\barint_{\Delta}k_L^{X_{\Delta_0}}(y)\,d\sigma(y),\qquad \Delta=B\cap\partial\Omega.
$$
Analogously, we say that $\omega_L\in RH_p^{\rm dyadic}(\partial\Omega)$ if
$\omega_L\ll\sigma$ and $k_L^{X_{Q_0}}\in RH_p(Q_0)$ uniformly in $Q_0$ for every $Q_0\in\mathbb{D}(\partial\Omega)$. That is, there exists
$C\geq 1$ such that for every $Q_0\in\mathbb{D}(\partial\Omega)$ and every $Q\in\mathbb{D}_{Q_0}$, we have that
$$
\bigg(\barint_{Q}k_L^{X_{Q_0}}(y)^p\,d\sigma(y)\bigg)^{1/p}\leq
C\barint_{Q}k_L^{X_{Q_0}}(y)\,d\sigma(y).
$$
\end{definition}
\medskip

Before going further, let us introduce the following operators (see \cite{hofmartelltuero}):
$$
Su(x):=\bigg(\iint_{\Gamma(x)}|\nabla u(Y)|^2\delta(Y)^{1-n}\,dY\bigg)^{1/2},\qquad\widetilde{\mathcal{N}}_*u(x):=\sup_{Y\in\widetilde{\Gamma}(x)}|u(Y)|,
$$
where
$$
\Gamma(x):=\bigcup_{x\in Q\in\mathbb{D}(\partial\Omega)}U_Q,\qquad\widetilde{\Gamma}(x):=\bigcup_{x\in Q\in\mathbb{D}(\partial\Omega)}U_Q^*.
$$
Similarly, we can define localized versions of the above operators. For a fixed $Q_0\in\mathbb{D}(\partial\Omega)$, we define
$$
S_{Q_0}u(x):=\bigg(\iint_{\Gamma_{Q_0}(x)}|\nabla u(Y)|^2\delta(Y)^{1-n}\,dY\bigg)^{1/2},\qquad\widetilde{\mathcal{N}}_{Q_0,*}u(x):=\sup_{Y\in\widetilde{\Gamma}_{Q_0}(x)}|u(Y)|,
$$
for each $x\in Q_0$, where
$$
\Gamma_{Q_0}(x):=\bigcup_{x\in Q\in\mathbb{D}_{Q_0}}U_{Q},\qquad\widetilde{\Gamma}_{Q_0}(x):=\bigcup_{x\in Q\in\mathbb{D}_{Q_0}}U_{Q}^*.
$$

\begin{theorem}[\cite{hofmartellgeneral}]\label{solvability}
Suppose that $\Omega\subset\re^{n+1}$ is a 1-sided $\mathrm{CAD}$, let $L$ be an elliptic operator and let $1<p<\infty$, the following statements are equivalent:
\begin{list}{$(\theenumi)$}{\usecounter{enumi}\leftmargin=1cm \labelwidth=1cm \itemsep=0.1cm \topsep=.2cm \renewcommand{\theenumi}{\alph{enumi}}}
\item There exists $C\ge 1$ such that
$$
\|\widetilde{\mathcal{N}}_*u\|_{L^{p'}(\partial\Omega)}\leq C\|f\|_{L^{p'}(\partial\Omega)},
$$
whenever
\begin{equation}\label{eq:u-hm-sol}
u(X)=\int_{\partial\Omega}f(y)\,d\omega_L^X(y),
\qquad
f\in C_c(\partial\Omega).
\end{equation}

\item $\omega_L\in RH_p(\partial\Omega)$ (cf. Definition \ref{defrhp}).

\item $\omega_L\ll\sigma$ and there exists $C\ge 1$
such that for every $B:=B(x,r)$ with $x\in\partial\Omega$ and $0<r<\diam(\partial\Omega)$, we have that
\begin{equation}\label{reverseholderL}
\int_{\Delta}k_L^{X_{\Delta}}(y)^p\,d\sigma(y)\leq C\sigma(\Delta)^{1-p},\quad \Delta=B\cap\partial\Omega.
\end{equation}
\end{list}
Moreover, $(a)$, $(b)$ and/or $(c)$ yield that for every $0<q<\infty$ there exists $C$ (depending only on dimension, the 1-sided $\mathrm{CAD}$ constants, the ellipticity of $L$, the constants in $(a)$, $(b)$ and/or $(c)$, and on $q$) such that for every $Q_0\in\mathbb{D}$
\begin{equation}\label{eq:S<N}
\|S_{Q_0}u\|_{L^{q}(Q_0)}
\lesssim
\|\widetilde{\mathcal{N}}_{Q_0,*}u\|_{L^{q}(Q_0)}
\end{equation}
for every $u$ as in \eqref{eq:u-hm-sol}.
\end{theorem}
\medskip

\begin{remark}\label{obsrhpq0}
Note that $\omega_L\in RH_p(\partial\Omega)$, together with Lemma \hyperref[proppde]{\ref*{proppde}(b)} and Harnack's inequality, imply that $\omega_L\in RH_p^{\rm dyadic}(\partial\Omega)$. This in turn gives
\begin{equation}\label{reverseholderLdyadic}
\int_{Q}k_L^{X_{Q}}(y)^p\,d\sigma(y)\leq C\sigma(Q)^{1-p},\qquad Q\in\mathbb{D}(\partial\Omega).
\end{equation}
Moreover, from \eqref{reverseholderLdyadic} and Harnack's inequality, we can see that \eqref{reverseholderL} holds, and hence $\omega_L\in RH_p(\partial\Omega)$. Therefore, the conditions $\omega_L\in RH_p(\partial\Omega)$, $\omega_L\in RH_p^{\rm dyadic}(\partial\Omega)$, \eqref{reverseholderL} and \eqref{reverseholderLdyadic} are all equivalent.
\end{remark}

\section{Auxiliary results}\label{section:aux}

The following result is a generalization of \cite[Lemma B.7]{MR2833577} to our setting dyadic setting.

\begin{lemma}\label{lemmab7}
Suppose that $\Omega\subset\re^{n+1}$ is an open set such that $\partial\Omega$ satisfies the $\mathrm{AR}$ property. Fix $0<\eta<1$, $Q_0\in\mathbb{D}(\partial\Omega)$ and let $v\in L^1(Q_0)$ be such that $0<v(Q)\leq C_0 v(\eta\Delta_Q)$ for every
$Q\in\mathbb{D}_{Q_0}$, for some uniform $C_0\ge 1$. Suppose also that there exist $C_1\geq 1$ and $1<p<\infty$ such that
\begin{equation}\label{hypothesisB7}
\bigg(\barint_{\eta \Delta_Q}v(y)^p\,d\sigma(y)\bigg)^{1/p} \leq
C_1\barint_{\eta \Delta_Q}v(y)\,d\sigma(y),\qquad Q\in\mathbb{D}_{Q_0},
\end{equation}
then $v\in A_\infty^{\rm dyadic}(Q_0)$, with the implicit constants depending on dimension, $p$, $C_0$,
$C_1$, $\eta$ and the $\mathrm{AR}$ constant.
\end{lemma}

\begin{proof}
We first prove  that for every $Q\in\mathbb{D}_{Q_0}$ and every Borel set $F\subset\eta \Delta_Q$, there holds
\begin{equation}\label{lanjaron}
\frac{v(F)}{v(\eta \Delta_Q)}\leq
C_1\bigg(\frac{\sigma(F)}{\sigma(\eta \Delta_Q)}\bigg)^{1/p'}.
\end{equation}
Indeed, 
using Hölder's inequality together with \eqref{hypothesisB7},
we obtain
\begin{multline*}
\frac{v(F)}{\sigma(\eta \Delta_Q)}
=
\frac1{\sigma(\eta \Delta_Q)}
\int_Fv(y)\,d\sigma(y)
\leq
\bigg(\frac{\sigma(F)}{\sigma(\eta \Delta_Q)}\bigg)^{1/p'} 
\Big(\barint_{\eta \Delta_Q}v(y)^p\,d\sigma(y)\Big)^{1/p} \\
\leq
C_1
\bigg(\frac{\sigma(F)}{\sigma(\eta \Delta_Q)}\bigg)^{1/p'} 
\barint_{\eta \Delta_Q}v(y)\,d\sigma(y),
\end{multline*}
which is equivalent to \eqref{lanjaron}. 

To obtain that $v\in A_\infty^{\text{dyadic}}(Q_0)$, we observe that $\sigma(Q)\leq C\sigma(\eta\Delta_Q)$ with  $C>1$ depending only on AR and $n$. Fix then 
$0<\alpha<(CC_1^{p'})^{-1}$ and take $E\subset Q$ such that
$\sigma(E)>(1-\alpha)\sigma(Q)$. Writing $E_0=E\cap\eta\Delta_Q$ and $F_0=\eta\Delta_Q\setminus E$, it is clear that
$$
(1-\alpha)\frac{\sigma(Q)}{\sigma(\eta\Delta_Q)}
<
\frac{\sigma(E)}{{\sigma(\eta\Delta_Q)}}
\leq
\frac{\sigma(E_0)}{\sigma(\eta\Delta_Q)}+\frac{\sigma(Q\setminus\eta\Delta_Q)}{\sigma(\eta\Delta_Q)}
=
\frac{\sigma(E_0)}{\sigma(\eta\Delta_Q)}+\frac{\sigma(Q)}{\sigma(\eta\Delta_Q)}-1,
$$
and hence 
\begin{equation}\label{alsowehave}
\frac{\sigma(F_0)}{\sigma(\eta\Delta_Q)}=
1-\frac{\sigma(E_0)}{\sigma(\eta\Delta_Q)}
<
\alpha\,\frac{\sigma(Q)}{\sigma(\eta\Delta_Q)}
\le
C\,\alpha
.
\end{equation}
Combining \eqref{lanjaron} and \eqref{alsowehave} we obtain $v(F_0)/v(\eta\Delta_Q)< C_1(C\alpha)^{1/p'}$. This and the fact that $v(Q)\leq C_0 v(\eta\Delta_Q)$ yield
$$
\frac{v(E)}{v(Q)}
\geq
\frac{v(\eta\Delta_Q)}{v(Q)}\,\frac{v(E_0)}{v(\eta\Delta_Q)}
\ge
C_0^{-1}\,\left(1-\frac{v(F_0)}{v(\eta\Delta_Q)}\right)
>
C_0^{-1}\,(1-C_1(C\alpha)^{1/p'})=:1-\beta,
$$
with $0<\beta<1$ by our choice of $\alpha$. This eventually proves that $v\in A_\infty^{\text{dyadic}}(Q_0)$ and the proof is complete.
\end{proof}

The following auxiliary result is standard and its proof is left to the interested reader.

\begin{lemma}\label{densidad}
Let $\varphi\in C_c^{\infty}(\re)$ be such that $\mathbf{1}_{(0,1)}\leq\varphi\leq\mathbf{1}_{(0,2)}$. For each $t>0$ and  $h\in L^1_{\rm{loc}}(\partial\Omega)$ define 
\begin{equation}\label{defgt}
P_t h(x):=\int_{\partial\Omega}\varphi_t(x,y)h(y)\,d\sigma(y),\qquad x\in \partial\Omega,
\end{equation}
where
$$
\varphi_t(x,y):=\frac{\varphi\big(\frac{|x-y|}{t}\big)}{\int_{\partial\Omega}\varphi\big(\frac{|x-z|}{t}\big)\,d\sigma(z)},\qquad x,y\in \partial\Omega.
$$
The following hold:

\begin{list}{$(\theenumi)$}{\usecounter{enumi}\leftmargin=1cm \labelwidth=1cm \itemsep=0.1cm \topsep=.2cm \renewcommand{\theenumi}{\alph{enumi}}}

\item $P_t$ is uniformly bounded on $L^q(\partial\Omega)$ for every $1\le q\le\infty$.

\item If $h\in L^q(\partial\Omega)$, $1\le q\le \infty$, and $t>0$ then $P_t h\in L^\infty(\partial\Omega)\cap \mathrm{Lip}(\partial\Omega)$.

\item If $h\in L^q(\partial\Omega)$, $1\le q< \infty$,  then $P_t h\longrightarrow h$ in $L^q(\partial\Omega)$ as $t\to 0^+$.

\item If $h\in C_c(\partial\Omega)$, then $P_t h(x)\longrightarrow h(x)$ as $t\to 0^+$ for every $x\in\partial\Omega$.

\item If $h\in L^q(\partial\Omega)$, $1\le q\le \infty$, with $\supp h \subset \Delta(x_0,r_0)$  then $\supp P_th\subset \Delta(x_0,r_0+2\,t)$.

\end{list}
\end{lemma}

Fix $Q_0\in\mathbb{D}(E)$ and consider the operators $\mathcal{A}_{Q_0}$, $\mathfrak{C}_{Q_0}$ defined by
\begin{equation}\label{definicionA-C}
\mathcal{A}_{Q_0}\alpha(x):=\bigg(\sum_{x\in Q\in\mathbb{D}_{Q_0}}\frac{1}{\ell(Q)^{n}}\,\alpha_{Q}^2\bigg)^{1/2},\quad
\mathfrak{C}_{Q_0}\alpha(x):=\sup_{x\in Q\in\mathbb{D}_{Q_0}}\bigg(\frac{1}{\sigma(Q)}\sum_{Q'\in\mathbb{D}_{Q}}\alpha_{Q'}^2\bigg)^{1/2},
\end{equation}
where $\alpha=\{\alpha_Q\}_{Q\in\mathbb{D}_{Q_0}}$ is a sequence of real numbers. Note that these operators are discrete analogues of those used in \cite{coifmanmeyerstein} to develop the theory of tent spaces. Sometimes, we use a truncated version of $\mathcal{A}_{Q_0}$, defined for each $k\geq 0$ by
$$
\mathcal{A}^k_{Q_0}\alpha(x):=\Bigg(\sum_{x\in Q\in\mathbb{D}_{Q_0}^k}\frac{1}{\ell(Q)^{n}}\,\alpha_{Q}^2\Bigg)^{1/2},
$$
where $\mathbb{D}_{Q_0}^k$ is the collection of $Q\in\mathbb{D}_{Q_0}$ such that $\ell(Q)\leq 2^{-k}\ell(Q_0)$. The following proposition is a discrete version of \cite[Theorem 1]{coifmanmeyerstein}.

\begin{lemma}\label{tentspaces}
Suppose that $E\subset\re^{n+1}$ is $n$-dimensional $\mathrm{AR}$, fix $Q_0\in\mathbb{D}(E)$, let $\mathcal{A}_{Q_0}$ and $\mathfrak{C}_{Q_0}$ be the operators defined in \eqref{definicionA-C} respectively. There exists $C$, depending only on dimension and the $\mathrm{AR}$ constant, such that for every $\alpha=\{\alpha_{Q}\}_{Q\in\mathbb{D}_{Q_0}}$, $\beta=\{\beta_{Q}\}_{Q\in\mathbb{D}_{Q_0}}$ sequences of real numbers, we have that
\begin{equation}\label{cotatentspaces}
\sum_{Q\in\mathbb{D}_{Q_0}}\,|\alpha_{Q}\beta_{Q}|\leq C\int_{Q_0}\mathcal{A}_{Q_0}\alpha(x)\mathfrak{C}_{Q_0}\beta(x)\,d\sigma(x).
\end{equation}
\end{lemma}
\begin{proof}
It is easy to see that, without loss of generality, we may assume that $\beta_Q=0$ when $\ell(Q)\leq 2^{-N}\ell(Q_0)$ for some $N\in\mathbb{N}$. In that scenario, we need to establish \eqref{cotatentspaces} with $C$ independent of $N$. Fix then such a $\beta$ and for $Q\in\mathbb{D}_{Q_0}$, let $k_Q\geq 0$ be so that $\ell(Q)=2^{-k_Q}\ell(Q_0)$. Suppose that $Q'\in\mathbb{D}_{Q_0}$ satisfies $\ell(Q')\leq 2^{-k_Q}\ell(Q_0)=\ell(Q)$ and $Q'\cap Q\neq\emptyset$, then necessarily $Q'\in\mathbb{D}_{Q}$. Therefore, using the $\mathrm{AR}$ property we obtain
$$
\int_{Q}\mathcal{A}_{Q_0}^{k_Q}\beta(y)^2\,d\sigma(y)=\int_{Q}\sum_{Q'\in\mathbb{D}_{Q}}\mathbf{1}_{Q'}(y)\frac{1}{\ell(Q')^{n}}\,\beta_{Q'}^2\,d\sigma(y)
\lesssim\sum_{Q'\in\mathbb{D}_{Q}}\,\beta_{Q'}^2.
$$
Dividing both sides by $\sigma(Q)$, we have proved that for every $Q\in\mathbb{D}_{Q_0}$ and every $x\in Q$ we have that
\begin{equation}\label{acotacionk0}
\eta_Q:=\barint_{Q}\mathcal{A}_{Q_0}^{k_Q}\beta(y)^2\,d\sigma(y)\leq C_0\,\mathfrak{C}_{Q_0}\beta(x)^2,\qquad\ell(Q)=2^{-k_Q}\ell(Q_0),
\end{equation}
with $C_0$ depending only on the $\mathrm{AR}$ constant. Since $\beta_Q=0$ for $\ell(Q)\leq 2^{-N}\ell(Q_0)$, we have that $\mathcal{A}_{Q_0}\beta(x)\leq C_N<\infty$ and hence $\eta_Q\leq C_N^2<\infty$. Now, we set $C_1=2\sqrt{C_0}$ and define
$$
F_0:=\big\{x\in Q_0:\:\mathcal{A}_{Q_0}^k\beta(x)> C_1\mathfrak{C}_{Q_0}\beta(x),\;\forall k\geq 0\big\}.
$$
In particular, using \eqref{acotacionk0}, we have $\mathcal{A}_{Q_0}^{k_Q}\beta(x)>2\eta_Q^{1/2}$ for each $x\in Q\cap F_0$. We claim that $4\sigma(Q\cap F_0)\leq\sigma(Q)$. Indeed, if $\eta_Q=0$ the estimate is trivial since $Q\cap F_0=\emptyset$. On the other hand, if $\eta_Q>0$, we have
$$
4\eta_Q\sigma(Q\cap F_0)\leq\int_{Q\cap F_0}\mathcal{A}_{Q_0}^{k_Q}\beta(y)^2\,d\sigma(y)\leq\eta_Q\sigma(Q),
$$
and the desired estimate follows since $0<\eta_Q<\infty$. Let us now consider
\begin{equation}\label{defkx}
k(x):=\min\big\{k\geq 0:\:\mathcal{A}_{Q_0}^k\beta(x)\leq C_1\mathfrak{C}_{Q_0}\beta(x)\big\},\qquad x\in Q_0\setminus F_0.
\end{equation}
Setting $F_{1,Q}:=\{x\in Q\setminus F_0:\,k(x)>k_Q\}$ and using \eqref{acotacionk0} we obtain
$$
F_{1,Q}\subset\{x\in Q\setminus F_0:\:\mathcal{A}_{Q_0}^{k_Q}\beta(x)>2\eta_Q^{1/2}\big\}.
$$
Applying Chebychev's inequality, it follows that
$$
\sigma(F_{1,Q})\leq\frac{1}{4\eta_Q}\int_{Q\setminus F_0}\mathcal{A}_{Q_0}^{k_Q}\beta(y)^2\,d\sigma(y)\leq\frac{1}{4}\sigma(Q).
$$
Setting $F_{2,Q}:=\{x\in Q\setminus F_0:\,k(x)\leq k_Q\}$, and gathering the above estimates, we have
$$
\sigma(F_{2,Q})=\sigma(Q)-\sigma(Q\cap F_0)-\sigma(F_{1,Q})\geq\frac{1}{2}\sigma(Q).
$$
Hence, the $\mathrm{AR}$ property, Cauchy-Schwarz's inequality and \eqref{defkx} yield
\begin{align*}
\sum_{Q\in\mathbb{D}_{Q_0}}|\alpha_Q\beta_Q|
&\lesssim
\sum_{Q\in\mathbb{D}_{Q_0}}\sigma(F_{2,Q})\frac{|\alpha_Q\beta_Q|}{\ell(Q)^{n}}
\leq
\int_{Q_0\setminus F_0}\sum_{Q\in\mathbb{D}_{Q_0}}\frac{|\alpha_Q\beta_Q|}{\ell(Q)^{n}}\mathbf{1}_{F_{2,Q}}(x)\,d\sigma(x)
\\
&\lesssim
\int_{Q_0\setminus F_0}\mathcal{A}_{Q_0}\alpha(x)\bigg(\sum_{Q\in\mathbb{D}_{Q_0}}\frac{1}{\ell(Q)^{n}}\,\beta_{Q}^2\mathbf{1}_{F_{2,Q}}(x)\bigg)^{1/2}\,d\sigma(x)
\\
&\lesssim
\int_{Q_0\setminus F_0}\mathcal{A}_{Q_0}\alpha(x)\mathcal{A}_{Q_0}^{k(x)}\beta(x)\,d\sigma(x)
\\
&\lesssim
\int_{Q_0}\mathcal{A}_{Q_0}\alpha(x)\mathfrak{C}_{Q_0}\beta(x)\,d\sigma(x),
\end{align*}
where we have used that $Q\in\mathbb{D}_{Q_0}^{k(x)}$ for each $x\in F_{2,Q}$. As the implicit constant does not depend on $N\in\mathbb{N}$, this completes the proof of \eqref{cotatentspaces}.
\end{proof}
\medskip

\begin{lemma}\label{lema272'}
Suppose that $\Omega\subset\re^{n+1}$ is a bounded open set such that $\partial\Omega$ satisfies the $\mathrm{AR}$ property. Let $L_0$, $L_1$ be elliptic operators, and let $u_0\in W^{1,2}(\Omega)$ be a weak solution of $L_0u_0=0$ in $\Omega$. Then,
\begin{equation}\label{272'}
\iint_{\Omega}A_0(Y)\nabla_YG_{L_1}(Y,X)\cdot\nabla u_0(Y)\,dY=0,\quad\text{for a.e. }X\in\Omega.
\end{equation}
\end{lemma}
\begin{proof}
Let us take a cut-off function $\varphi\in C_c([-2,2])$ such that $0\leq\varphi\leq 1$ and $\varphi\equiv 1$ in $[-1,1]$. Fix $X_0\in\Omega$, for each $0<\varepsilon<\delta(X_0)/16$ we set $\varphi_\varepsilon(X)=\varphi(|X-X_0|/\varepsilon)$ and $\psi_\varepsilon=1-\varphi_\varepsilon$. Using \eqref{obsgreen2} we have that $G_{L_1}(\cdot,X_0)\psi_\varepsilon\in W_0^{1,2}(\Omega)$, which together with the fact that $u_0\in W^{1,2}(\Omega)$ is a weak solution of $L_0u_0=0$ in $\Omega$, implies
$$
\iint_{\Omega}A_0(Y)\nabla \big(G_{L_1}(\cdot,X_0)\psi_\varepsilon\big)(Y)\cdot\nabla u_0(Y)\,dY=0.
$$
Hence, we can write
\begin{multline}
\iint_{\Omega}A_0\nabla G_{L_1}(\cdot,X_0)\cdot\nabla u_0\,dY=
\iint_{\Omega}A_0\nabla\big(G_{L_1}(\cdot,X_0)\varphi_\varepsilon\big)\cdot\nabla u_0\,dY
\\
=\iint_{\Omega}A_0\nabla G_{L_1}(\cdot,X_0)\cdot\nabla u_0\,\varphi_\varepsilon\,dY
+\iint_{\Omega}A_0\nabla\varphi_\varepsilon\cdot\nabla u_0\,G_{L_1}(\cdot,X_0)\,dY=:\mathrm{I}_\varepsilon+\mathrm{II}_\varepsilon.
\end{multline}
In order to simplify the notation we  set $C_j(X_0,\varepsilon):=\{Y\in\re^{n+1}:\,2^{-j+1}\varepsilon\leq|Y-X_0|<2^{-j+2}\varepsilon\}$ for $j\geq 1$. For the first term, we use Caccioppoli's inequality and \eqref{sizestimate}
\begin{align}\label{acotieps}
|\mathrm{I}_\varepsilon|
&\lesssim
\iint_{B(X_0,2\varepsilon)}|\nabla_YG_{L_1}(Y,X_0)||\nabla u_0(Y)|\,dY
\\
&\lesssim
\sum_{j=1}^{\infty}(2^{-j}\varepsilon)^{n+1}\bigg(\bariint_{C_j(X_0,\varepsilon)}|\nabla G_{L_1}(\cdot,X_0)|^2\,dY\bigg)^{1/2}\bigg(\bariint_{B(X_0,2^{-j+2}\varepsilon)}|\nabla u_0|^2\,dY\bigg)^{1/2}\nonumber
\\
&\lesssim
\sum_{j=1}^{\infty}2^{-j}\varepsilon M_2(|\nabla u_0|\mathbf{1}_{\Omega})(X_0)\lesssim\varepsilon M_2(|\nabla u_0|\mathbf{1}_{\Omega})(X_0),\nonumber
\end{align}
where  $M_2f(X):=M(|f|^2)(X)^{1/2}$, with $M$ being the Hardy-Littlewood maximal operator on $\re^{n+1}$. For the second term, using again \eqref{sizestimate},
\begin{multline}\label{acotiieps}
|\mathrm{II}_\varepsilon|\lesssim\varepsilon^{-1}\iint_{C_1(X_0,\varepsilon)}|G_{L_1}(Y,X_0)||\nabla u_0(Y)|\,dY
\\
\lesssim\varepsilon^{-n}\iint_{B(X_0,2\varepsilon)}|\nabla u_0(Y)|\,dY\lesssim \varepsilon M_2(|\nabla u_0|\mathbf{1}_{\Omega})(X_0).
\end{multline}
Combining \eqref{acotieps} and \eqref{acotiieps}, we have proved that, for every $X_0\in\Omega$ and for every $0<\varepsilon<\delta(X_0)/16$,
\begin{equation}\label{cotaM2}
\bigg|\iint_{\Omega}A_0(Y)\nabla_Y G_{L_1}(Y,X_0)\cdot\nabla u_0(Y)\,dY\bigg|\lesssim\varepsilon M_2(|\nabla u_0|\mathbf{1}_{\Omega})(X_0).
\end{equation}
Recall that $M_2(|\nabla u_0|\mathbf{1}_{\Omega})\in L^{1,\infty}(\Omega)$ as $|\nabla u_0|\in L^2(\Omega)$, thus $M_2(|\nabla u_0|\mathbf{1}_{\Omega})(X)<\infty$ for almost every $X\in\Omega$. Taking limits as $\varepsilon\rightarrow 0$ in \eqref{cotaM2}, we obtain as desired \eqref{272'}.
\end{proof}

\begin{lemma}\label{proprepresentacotado}
Suppose that $\Omega\subset\re^{n+1}$ is a bounded open set such that $\partial\Omega$ satisfies the $\mathrm{AR}$ property. Let $L_0$ and $L_1$ be elliptic operators, and let $g\in H^{1/2}(\partial\Omega)\cap C_c(\partial\Omega)$. Consider the solutions $u_0$ and $u_1$ given by
$$
u_0(X)=\int_{\partial\Omega}g(y)\,d\omega_{L_0}^X(y),\qquad u_1(X)=\int_{\partial\Omega}g(y)\,d\omega_{L_1}^X(y),\qquad X\in\Omega.
$$
Then,
\begin{equation}\label{representacotado}
u_1(X)-u_0(X)=\iint_{\Omega}(A_0-A_1)(Y)\nabla_Y G_{L_1}(Y,X)\cdot\nabla u_0(Y)\,dY,\quad\text{for a.e. }X\in\Omega.
\end{equation}
\end{lemma}
\begin{proof}
Following \cite{hofmartellgeneral} we know that $u_0=\widetilde{g} - v_0$ and $u_1=\widetilde{g}-v_1$, where $\widetilde{g}=\mathcal{E}_{\partial\Omega}g\in W^{1,2}(\re^{n+1})$ is the Jonsson-Wallin extension (see \cite{jonsson1984function}), and $v_0,v_1\in W_0^{1,2}(\Omega)$ are the Lax-Milgram solutions of $L_0v_0=L_0\widetilde{g}$ and $L_1v_1=L_1\widetilde{g}$ respectively. Hence, we have that $u_1-u_0=v_0-v_1\in W_0^{1,2}(\Omega)$, and following again \cite{hofmartellgeneral} we obtain
$$
(u_1-u_0)(X)=\iint_{\Omega}A_1(Y)\nabla_YG_{L_1}(Y,X)\cdot\nabla (u_1-u_0)(Y)\,dY,\quad\text{for a.e. }X\in\Omega.
$$
For almost every $X\in\Omega$ we then have that
\begin{multline*}
(u_1-u_0)(X)-\iint_{\Omega}(A_0-A_1)(Y)\nabla_Y G_{L_1}(Y,X)\cdot\nabla u_0(Y)\,dY=
\\
=\iint_{\Omega}A_1(Y)\nabla_YG_{L_1}(Y,X)\cdot\nabla u_1(Y)\,dY
-\iint_{\Omega}A_0(Y)\nabla_Y G_{L_1}(Y,X)\cdot\nabla u_0(Y)\,dY.
\end{multline*}
Using Lemma \ref{lema272'} for both terms, the right side of the above equality vanishes almost everywhere, and this proves \eqref{representacotado}.
\end{proof}
\medskip

\begin{lemma}\label{proprepresent}
Suppose that $\Omega\subset\re^{n+1}$ is an open set such that $\partial\Omega$ satisfies the $\mathrm{AR}$ property. Let $L_0$, $L_1$ be elliptic operators such that $K:=\esssup(A_0-A_1)\cap\Omega$ is compact. For every $g\in H^{1/2}(\partial\Omega)\cap C_c(\partial\Omega)$, let
$$
u_0(X)=\int_{\partial\Omega}g(y)\,d\omega_{L_0}^X(y),\qquad u_1(X)=\int_{\partial\Omega}g(y)\,d\omega_{L_1}^X(y),\qquad X\in\Omega.
$$
Then, for almost every $X\in\Omega\setminus K$, there holds
\begin{equation}\label{represent}
u_1(X)-u_0(X)=\iint_{\Omega}(A_0-A_1)(Y)\nabla_Y G_{L_1}(Y,X)\cdot\nabla u_0(Y)\,dY.
\end{equation}
\end{lemma}
\begin{proof}
First, fix $x_0\in\partial\Omega$, following \cite{hofmartellgeneral} we consider the family of bounded increasing open subsets $\{\mathcal{T}_k\}_{k\in\mathbb{Z}}$ such that $\Omega=\bigcup_{k\in\mathbb{Z}}\mathcal{T}_k$, and $\partial\mathcal{T}_k$ satisfies the $\mathrm{AR}$ property, with constants possibly depending on $k$ and $\diam(\partial\Omega)$ (see \cite{hofmartellgeneral}). As we can see in \cite{jonsson1984function}, there exists an extension operator $\mathcal{E}_{\partial\Omega}$, which maps $H^{1/2}(\partial\Omega)$ continuously into $W^{1/2}(\re^{n+1})$, and a restriction operator $\mathfrak{R}_{\partial\Omega}$, which is bounded from $W^{1,2}(\re^{n+1})$ to $H^{1/2}(\partial\Omega)$, such that $\mathfrak{R}_{\partial\Omega}\circ\mathcal{E}_{\partial\Omega}=Id$ in $H^{1/2}(\partial\Omega)$. Moreover, we have that $\mathcal{E}_{\partial\Omega}f\in C_c(\re^{n+1})\cap L^{\infty}(\re^{n+1})$ for every $f\in H^{1/2}(\partial\Omega)\cap C_c(\partial\Omega)$.
Let $g\in H^{1/2}(\partial\Omega)\cap C_c(\partial\Omega)$ and $h=\mathcal{E}_{\partial\Omega}g\in W^{1,2}(\re^{n+1})\cap C_c(\re^{n+1})\cap L^{\infty}(\re^{n+1})$. Let $\eta\in C_c^{\infty}([-2,2])$ be such that $0\leq\eta\leq 1$, $\eta\equiv 1$ in $[-1,1]$, $\eta$ monotonously decreasing in $(1,2)$ and monotonously increasing in $(-2,-1)$. Let us consider $h_k(y)=h(y)\eta(|y-x_0|/2^k)$, as well as the solutions
$$
u_0^k(X)=\int_{\partial\mathcal{T}_k}h_k(y)\,d\omega_{L_0,\mathcal{T}_k}^X(y),\quad u_1^k(X)=\int_{\partial\mathcal{T}_k}h_k(y)\,d\omega_{L_1,\mathcal{T}_k}^X(y),\qquad X\in\mathcal{T}_k.
$$
We take $k_0\gg 1$ such that $\supp(g),\supp(h)\subset B(x_0,2^{k_0-1})$, in such a way that $h_k\equiv h$ for $k\geq k_0$. Note that by \cite{hofmartellgeneral}, $B(x_0,2^k)\cap\Omega\subset\mathcal{T}_k$, hence $h=g\mathbf{1}_{\partial\Omega}$ on $\partial\mathcal{T}_k$, and consequently $h\in H^{1/2}(\partial\mathcal{T}_k)\cap C_c(\partial\mathcal{T}_k)$ for $k\geq k_0$. Using Lemma \ref{proprepresentacotado}, we have that
\begin{equation}\label{u1ku0k}
(u_1^k-u_0^k)(X)=\iint_{\mathcal{T}_k}(A_0-A_1)(Y)\nabla_Y G_{L_1,\mathcal{T}_k}(Y,X)\cdot\nabla u_0^k(Y)\,dY,\qquad k\geq k_0,
\end{equation}
for almost every $X\in\mathcal{T}_k$. Let $\mathcal{G}_k$ be the set of points $X\in\mathcal{T}_k$ for which \eqref{u1ku0k} holds, and let $\mathcal{B}_k=\mathcal{T}_k\setminus\mathcal{G}_k$. We fix $X_0\in(\Omega\setminus K)\setminus\bigcup_{k\geq k_0}\mathcal{B}_k$ and take $k_0$ (possibly greater than before) such that $X_0\in B(x_0,2^{k_0-1})\cap\Omega \subset\mathcal{T}_k$ and $K\subset B(x_0,2^{k_0-1})\cap\Omega\subset \mathcal{T}_k$. Let us consider $v_k=G_{L_1,\mathcal{T}_k}(\cdot,X_0)$, which converge to $v=G_{L_1}(\cdot,X_0)$ uniformly on compacta in $\Omega\setminus\{X_0\}$ (see \cite{hofmartellgeneral}), and hence on $W^{1,2}_{\rm loc}(\Omega\setminus\{X_0\})$ by Caccioppoli's inequality. Also, note that for $i=0,1$, we have that $u_i^k\rightarrow u_i$ uniformly on compacta in $\Omega$ (see \cite{hofmartellgeneral}). In particular, Caccioppoli's inequality yields $u_0^k\rightarrow u_0$ in $W^{1,2}_{\rm loc}(\Omega)$. Thanks to these observations, using \eqref{u1ku0k} we obtain
\begin{multline*}
\bigg|(u_1^k-u_0^k)(X_0)-\iint_{\Omega}(A_0-A_1)(Y)\nabla_Y G_{L_1}(Y,X_0)\cdot\nabla u_0(Y)\,dY\bigg|
\\
\leq
\iint_K\big|(A_0-A_1)(Y)\big|\big|\nabla v_k(Y)\cdot\nabla u_0^k(Y)-\nabla v(Y)\cdot\nabla u_0(Y)\big|\,dY
\\
\lesssim\|\nabla v_k\|_{L^2(K)}\|\nabla u_0^k-\nabla u_0\|_{L^{2}(K)}+\|\nabla v_k-\nabla v\|_{L^2(K)}\|\nabla u_0\|_{L^2(K)}.
\end{multline*}
Taking limits as $k\to\infty$, \eqref{represent} is then proved.
\end{proof}

\begin{remark}\label{modific}
Note that
Lemma \ref{proprepresentacotado} ensures that there exists $\mathcal{G}\subset\Omega$ with $|\Omega\setminus\mathcal{G}|=0$ such that \eqref{representacotado} holds for all $X\in\mathcal{G}$. Let $\Delta=\Delta(x,r)$ with $x\in\partial\Omega$ and $0<r<\diam(\partial\Omega)$ be such that $X_\Delta\notin\mathcal{G}$. Take $\widetilde{X}_\Delta\in B(X_\Delta,cr/2)\cap\mathcal{G}$ where $0<c<1$ is the corkscrew constant. Taking into account that $B(\widetilde{X}_\Delta,cr/2)\subset B(X_\Delta,cr)$ and slightly modifying the constants, we can use $\widetilde{X}_\Delta$ as a corkscrew point associated with $\Delta$. Hence, we may assume that for every $\Delta$ as before, there exists a corkscrew point $X_\Delta\in\mathcal{G}$ for which \eqref{representacotado} holds with $X=X_\Delta$. Similarly, we may also assume that \eqref{represent} holds for $X_\Delta$, as long as $X_\Delta\notin K$. In particular, for every $Q\in\mathbb{D}(\partial\Omega)$, we can choose $X_Q$ so that \eqref{representacotado} and \eqref{represent} hold with $X=X_Q$ (the latter provided $X_Q\notin K$).
\end{remark}

\begin{lemma}\label{comp3}
Suppose that $\Omega\subset\re^{n+1}$ is a 1-sided $\mathrm{CAD}$. Fix $Q_0\in\mathbb{D}(\partial\Omega)$, let $L_1$ and $L_2$ be elliptic operators such that $\omega_{L_1}\ll\sigma$, $\omega_{L_2}\ll\sigma$, and $L_1\equiv L_2$ in $T_{Q_0}$. Given $0<\tau<1$, there exists $C_\tau>1$ such that
$$
\frac{1}{C_\tau}k_{L_2}^{X_{Q_0}}(y)\leq k_{L_1}^{X_{Q_0}}(y)\leq C_\tau k_{L_2}^{X_{Q_0}}(y),\quad\text{ for }\sigma\text{-a.e. }y\in Q_0\setminus\Sigma_{Q_0,\tau},
$$
where $\Sigma_{Q_0,\tau}$ is the region defined by $\Sigma_{Q_0,\tau}=\big\{x\in Q_0:\:\dist(x,\partial\Omega\setminus Q_0)<\tau\ell(Q_0)\big\}$.
\end{lemma}
\begin{proof}
Let $r=\tau\ell(Q_0)/M$ with $M> 1$ to be chosen. Using a Vitali-type covering argument, we construct a maximal collection of points $\{x_j\}_{j\in\mathcal{J}}\subset Q_0\setminus\Sigma_{Q_0,\tau}$ with respect to the property that $|x_j-x_k|>2r/3$ for every $j,k\in\mathcal{J}$, and a disjoint family $\{\Delta'_j\}_{j\in\mathcal{J}}$ given by $\Delta'_j=\Delta(x_j,r/3)$, in such a way that $Q_0\setminus\Sigma_{Q_0,\tau}\subset\bigcup_{j\in\mathcal{J}}3\Delta'_j$. Note that there exists $C$, depending only on dimension and on the 1-sided $\mathrm{CAD}$ constants, such that $\Delta_j'\subset\Delta(x_{Q_0},C\ell(Q_0))$ for every $j\in\mathcal{J}$. Hence,
$$
\#\mathcal{J}\Big(\frac{\tau\ell(Q_0)}{M}\Big)^n\approx\sum_{j\in\mathcal{J}}\sigma(\Delta'_j)=\sigma\Big(\bigcup_{j\in\mathcal{J}}\Delta'_j\Big)\leq
\sigma(\Delta(x_{Q_0},C\ell(Q_0)))\approx \ell(Q_0)^n.
$$
We have then obtained a covering $\{\Delta_j\}_{j=1}^{N_\tau}$ of $Q_0\setminus\Sigma_{Q_0,\tau}$ by balls $\Delta_j=\Delta(x_j,r)$ with $x_j\in Q_0\setminus\Sigma_{Q_0,\tau}$, $r=\tau\ell(Q_0)/M$ and $N_\tau\lesssim(M/\tau)^n$. We claim that for $M\gg 1$ we have $B_j^*\cap\Omega\subset T_{Q_0}$, with $B_j^*:=B_{\Delta_j}^*=B(x_j,2\kappa_0r)$ and $\kappa_0$  as in \eqref{definicionkappa0}. Let $Y\in B_j^*\cap\Omega$ and $I\in\mathcal{W}$ be such that $Y\in I$. Take $y_j\in\partial\Omega$ such that $\dist(I,\partial\Omega)=\dist(I,y_j)$ and pick $Q_j\in\mathbb{D}(\partial\Omega)$ the unique cube such that $y_j\in Q_j$ and $\ell(Q_j)=\ell(I)$. As  already observed, $I\in\mathcal{W}_{Q_j}^*$. We are going to see that $Q_j\in\mathbb{D}_{Q_0}$. First of all, note that
$$
\ell(Q_j)=\ell(I)\approx\dist(I,\partial\Omega)\leq|x_j-Y|<2\kappa_0\tau\ell(Q_0)/M<2\kappa_0\ell(Q_0)/M.
$$
Choosing $M\gg 1$ sufficiently large (independent of $\tau$) we may obtain $\ell(Q_j)<\ell(Q_0)/4$ and $\dist(I,\partial\Omega)\leq|x_j-Y|<\tau\ell(Q_0)/4$. Also, since $x_j\in Q_0\setminus\Sigma_{Q_0,\tau}$, we can write by \eqref{constwhitney}
\begin{multline*}
\tau\ell(Q_0)
\leq\dist(x_j,\partial\Omega\setminus Q_0)\leq
|x_j-Y|+\diam(I)+\dist(I,y_j)+\dist(y_j,\partial\Omega\setminus Q_0)
\\
\leq
\tfrac{1}{4}\tau\ell(Q_0)+\tfrac{5}{4}\dist(I,\partial\Omega)+\dist(y_j,\partial\Omega\setminus Q_0)\leq \tfrac{9}{16}\tau\ell(Q_0)+\dist(y_j,\partial\Omega\setminus Q_0),
\end{multline*}
and hence $y_j\in \interior(Q_0)$. Since $y_j\in Q_0\cap Q_j$ and $\ell(Q_j)<\ell(Q_0)/4$ it follows that $Q_j\in\mathbb{D}_{Q_0}$. This and the fact that $Y\in I\in \mathcal{W}_{Q_j}^*$ allow us to conclude that $Y\in T_{Q_0}$. Consequently, we have shown that $B_j^*\cap\Omega\subset T_{Q_0}$ and thus $L_1\equiv L_2$ in $B_j^*\cap\Omega$ for every $j=1,\dots, N_\tau$.

Next, we note that $\delta(X_{Q_0})\approx\ell(Q_0)\geq\tau\ell(Q_0)$, $\delta(X_{\Delta_j})\approx\tau\ell(Q_0)$, and $|X_{Q_0}-X_{\Delta_j}|\lesssim\ell(Q_0)$. Hence, we can use Harnack's inequality to move from $X_{Q_0}$ to $X_{\Delta_j}$ with constants depending on $\tau$, and Lemma \hyperref[proppde]{\ref*{proppde}(e)}, we obtain
$$
k_{L_1}^{X_{Q_0}}(y)\approx_{\tau}k_{L_1}^{X_{\Delta_j}}(y)\approx k_{L_2}^{X_{\Delta_j}}(y)\approx_{\tau}k_{L_2}^{X_{Q_0}}(y)
$$
for $\sigma$-almost every $y\in\Delta_j=B_j\cap\partial\Omega$. Since we know that $\{\Delta_j\}_{j=1}^{N_\tau}$ covers $Q_0\setminus\Sigma_{Q_0,\tau}$, the desired conclusion follows.
\end{proof}

\medskip

We will prove Theorem \hyperref[perturbationtheo]{\ref*{perturbationtheo}(a)} with the help of Lemma \ref{extrapolation}. In this way we consider the measure $\mathfrak{m}=\{\gamma_Q\}_{Q\in\mathbb{D}(\partial\Omega)}$, where
\begin{equation}\label{coefcarleson}
\gamma_{Q}:=\sum_{I\in\mathcal{W}_{Q}^*}\frac{\sup_{I^*}|\mathcal{E}|^2}{\ell(I)}|I|,\qquad Q\in\mathbb{D}(\partial\Omega),
\end{equation}
and $\mathcal{E}(Y)=A(Y)-A_0(Y)$. We are going to show that $\mathfrak{m}$ is indeed a discrete Carleson measure with respect to $\sigma$.

\begin{lemma}\label{carlesondisc}
Suppose that $\Omega\subset\re^{n+1}$ is a 1-sided $\mathrm{CAD}$, let $L_0$ and $L$ be elliptic operators whose disagreement in $\Omega$ is given by the function $a:=\varrho(A,A_0)$ defined in \eqref{discrepancia}, and suppose that $\vertiii{a}<\infty$, see \eqref{eq:defi-vertiii}. Then, there exists $\kappa>0$ (depending only on dimension and the 1-sided $\mathrm{CAD}$ constants) such that for every $Q_0\in\mathbb{D}(\partial\Omega)$ with $\ell(Q_0)<\diam(\partial\Omega)/\kappa_0$ (see \eqref{definicionkappa12}), the collection $\mathfrak{m}=\{\gamma_{Q}\}_{Q\in\mathbb{D}(\partial\Omega)}$ given by \eqref{coefcarleson}
defines a discrete Carleson measure $\mathfrak{m}\in\mathcal{C}(Q_0)$ with $\|\mathfrak{m}\|_{\mathcal{C}(Q_0)}\leq\kappa \vertiii{a}$.
\end{lemma}

\begin{proof}
Let $Q_0\in\mathbb{D}(\partial\Omega)$ with $\ell(Q_0)<\diam(\partial\Omega)/\kappa_0$. First, note that for every $I\in\mathcal{W}$ and every $Y\in I$ we have that $\sup_{I^*}|\mathcal{E}|\leq a(Y)$. Indeed, since $4\diam(I)\leq\dist(I,\partial\Omega)$ (see \eqref{constwhitney}), we know that $I^*\subset\{X\in\Omega:\,|X-Y|<\delta(Y)/2\}$. Given $Q\in\mathbb{D}_{Q_0}$ we can write
\begin{multline}\label{pruebacoefcarleson}
\mathfrak{m}(\mathbb{D}_{Q})
=
\sum_{Q'\in\mathbb{D}_{Q}}\gamma_{Q'}=\sum_{Q'\in\mathbb{D}_{Q}}\sum_{I\in\mathcal{W}_{Q'}^*}\frac{\sup_{I^*}|\mathcal{E}|^2}{\ell(I)}|I|
\\
\lesssim
\sum_{Q'\in\mathbb{D}_{Q}}\sum_{I\in\mathcal{W}_{Q'}^*}\iint_{I}\frac{a(Y)^2}{\delta(Y)}\,dY
\le
\sum_{Q'\in\mathbb{D}_{Q}}\iint_{U_{Q'}}\frac{a(Y)^2}{\delta(Y)}\,dY
\lesssim
\iint_{T_{Q}}\frac{a(Y)^2}{\delta(Y)}\,dY.
\end{multline}
where we have used that the family $\{U_{Q'}\}_{Q'\in\mathbb{D}_{Q}}$ has bounded overlap. Indeed,  if $Y\in U_{Q'}\cap U_{Q''}$ then $\ell(Q')\approx\delta(Y)\approx \ell(Q'')$ and $\dist(Q',Q'')\le \dist(Y,Q')+\dist(Y,Q'')\lesssim\ell(Q')+\ell(Q'')\approx\ell(Q')$. These readily imply that $Y$ can be only in a bounded number of $U_{Q'}$'s.

On the other hand, by \eqref{definicionkappa12} we know that $T_Q\subset B(x_Q,\kappa_0 r_Q)\cap\Omega$. Also, $\kappa_0r_Q\leq \kappa_0\ell(Q) \leq \kappa_0\ell(Q_0)<\diam(\partial\Omega)$. Using the $\mathrm{AR}$ property, from \eqref{pruebacoefcarleson} we conclude that
$$
\mathfrak{m}(\mathbb{D}_{Q})\lesssim \iint_{B(x_Q,\kappa_0 r_Q)\cap\Omega}\frac{a(Y)^2}{\delta(Y)}\,dY\leq \vertiii{a}\,\sigma(\Delta(x_Q,\kappa_0 r_Q))\lesssim \vertiii{a}\sigma(Q),
$$
Taking the supremum over $Q\in\mathbb{D}_{Q_0}$, we obtain $\|\mathfrak{m}\|_{\mathcal{C}(Q_0)}\leq \kappa \vertiii{a}$ with $\kappa$ depending on the allowable parameters. This completes the proof.
\end{proof}
\medskip

\section{Proof of Theorem \texorpdfstring{\hyperref[perturbationtheo]{\ref*{perturbationtheo}\textnormal{(a)}}}{1.1(a)}}\label{section:main-a} 

Before starting the proof we choose $M_0>2\kappa_0/c$ (which will remain fixed during the proof) where $c$ is the corkscrew constant and $\kappa_0$ as in  \eqref{definicionkappa12}. Given an arbitrary $Q_0\in\mathbb{D}(\partial\Omega)$ with $\ell(Q_0)<\diam(\partial\Omega)/M_0$ we let $B_{Q_0}=B(x_{Q_0},r_{Q_0})$ with $r_{Q_0}\approx\ell(Q_0)$ as in \eqref{deltaQ}. Let $X_{M_0\Delta_{Q_0}}$ be the corkscrew point relative to $M_0\Delta_{Q_0}$ (note that $M_0r_{Q_0}\leq M_0\ell(Q_0)<\diam(\partial\Omega)$). By our choice of $M_0$, it is clear that $\delta(X_{M_0\Delta_{Q_0}})\geq c M_0 r_{Q_0}>2\kappa_0 r_{Q_0}$. Hence, by \eqref{definicionkappa12},
\begin{equation}\label{eq:X0-TQ}
X_{M_0\Delta_{Q_0}}\in\Omega\setminus B_{Q_0}^*\subset\Omega\setminus T_{Q_0}^{**}.
\end{equation}

We will prove Theorem \hyperref[perturbationtheo]{\ref*{perturbationtheo}(a)} using Lemma \ref{extrapolation}. To do that we need to split the proof in several steps.

\subsection{Step 0}
We first make a reduction which will allow us to use some qualitative properties of the elliptic measure. Fix $j\in\mathbb{N}$ (large enough, as we eventually let $j\to\infty$) and $\widetilde{L}=L^j$ be the operator defined by $\widetilde{L}u=-\div(\widetilde{A}\nabla u)$, with
\begin{equation}\label{definicionLtilde}
\widetilde{A}(Y)=A^j(Y):=\left\{
             \begin{array}{ll}
               A(Y) & \hbox{$\text{if }\,\:Y\in\Omega,\:\delta(Y)\geq 2^{-j},$} \\
               A_0(Y) & \hbox{$\text{if }\,\:Y\in\Omega,\:\delta(Y)< 2^{-j}.$}
             \end{array}
           \right.
\end{equation}
Note that the matrix $A^j$ is uniformly elliptic with constant $\Lambda_j=\max\{\Lambda_A,\Lambda_{A_0}\}$, where $\Lambda_A$ and $\Lambda_{A_0}$ are the ellipticity constants of $A$ and $A_0$ respectively. Recall that $\omega_{L_0}\in RH_p(\partial\Omega)$ and that $\widetilde{L}\equiv L_0$ in $\{Y\in\Omega:\,\delta(Y)< 2^{-j}\}$. Therefore, applying Lemma \hyperref[proppde]{\ref*{proppde}(e)} we have that $\omega_{\widetilde{L}}\ll\sigma$ and there exists $k_{\widetilde{L}}^X:=d\omega_{\widetilde{L}}^X/d\sigma$. The fact that $\widetilde{L}$ verifies these qualitative hypotheses will be essential in the following steps. At the end of Step 4 we will have obtained the desired conclusion for the operator $\widetilde{L}=L^j$, with constants independent of $j\in\mathbb{N}$, and in Step 5 we will prove it for $L$ via a limiting argument. From now on, $j\in\mathbb{N}$ will be fixed and we will focus on the operator $\widetilde{L}=L^j$.
\medskip

\subsection{Step 1}
Let us fix $Q_0\in\mathbb{D}(\partial\Omega)$ with $\ell(Q_0)<\diam(\partial\Omega)/M_0$ and $M_0$ as chosen above, and set
$X_{0}:=X_{M_0\Delta_{Q_0}}$ so that \eqref{eq:X0-TQ} holds.
We also fix $\mathcal{F}=\{Q_i\}\subset\mathbb{D}_{Q_0}$ a family of disjoint dyadic subcubes such that
\begin{equation}\label{hipepsilon0}
\|\mathfrak{m}_{\mathcal{F}}\|_{\mathcal{C}(Q_0)}=\sup_{Q\in\mathbb{D}_{Q_0}}\frac{\mathfrak{m}(\mathbb{D}_{\mathcal{F},Q})}{\sigma(Q)}\leq\varepsilon_0,
\end{equation}
with $\varepsilon_0>0$ sufficiently small, to be chosen.
We modify the operator $L_0$ inside the region $\Omega_{\mathcal{F},Q_0}$ (see \eqref{defomegafq}), by defining $L_1=L_1^{\mathcal{F},Q_0}$ as $L_1u=-\div(A_1\nabla u)$, where
$$
A_1(Y):=\left\{
             \begin{array}{ll}
               \widetilde{A}(Y) & \hbox{$\text{if }\, Y\in\Omega_{\mathcal{F},Q_0},$} \\
               A_0(Y) & \hbox{$\text{if }\,Y\in\Omega\setminus \Omega_{\mathcal{F},Q_0},$}
             \end{array}
           \right.
$$
and $\widetilde{A}=A^j$ as in \eqref{definicionLtilde}. By construction, it is clear that $\mathcal{E}_1:=A_1-A_0$ verifies $|\mathcal{E}_1|\leq|\mathcal{E}|\mathbf{1}_{\Omega_{\mathcal{F},Q_0}}$ and also $\mathcal{E}_1(Y)=0$ if $\delta(Y)<2^{-j}$. Hence, the support of $A_1-A_0$ is contained in a compact subset contained in $\Omega$.

Our goal in Step 1 is to prove $\|k_{L_1}^{X_{Q_0}}\|_{L^p(Q_0)}\lesssim \sigma(Q_0)^{-1/p'}$ (uniformly in $j$), using that $\omega_{L_0}\in RH_p(\partial\Omega)$. Note that by Harnack's inequality and Lemma \hyperref[proppde]{\ref*{proppde}(e)}, we have that $\omega_{L_1}\ll\sigma$ and  $\|k_{L_1}^{X_{Q_0}}\|_{L^p(Q_0)}\leq C_j<\infty$ for $k_{L_1}^X:=d\omega_{L_1}^X/d\sigma$. We will use this qualitatively, and the point of this step is to show that we can actually remove the dependence on $j$.

Take an arbitrary $0\leq g\in L^{p'}(Q_0)$ such that $\|g\|_{L^{p'}(Q_0)}=1$. Without loss of generality we may assume that $g$ is defined in $\Omega$ with $g\equiv 0$ in $\Omega\setminus Q_0$. Let $\widetilde{\Delta}_{Q_0}:=\Delta(x_{Q_0},Cr_{Q_0})$ (see \eqref{deltaQ}) and take $0<t<C r_{Q_0}/2$. Set $g_t=P_tg$ (cf. Lemma \ref{densidad}) and consider the solutions
\begin{equation}\label{eq:defi-u0-u1}
u_0^t(X)=\int_{\partial\Omega}g_t(y)\,d\omega_{L_0}^X(y),\qquad u_1^t(X)=\int_{\partial\Omega} g_t(y)\,d\omega_{L_1}^X(y),\qquad X\in\Omega.
\end{equation}
By Lemma \ref{densidad}, $g_t\in\mathrm{Lip}(\partial\Omega)$ with $\supp(g_t)\subset 2\widetilde{\Delta}_{Q_0}$, hence $g_t\in\mathrm{Lip}_c(\partial\Omega)\subset H^{1/2}(\partial\Omega)\cap C_c(\partial\Omega)$.
Recall that  $\mathcal{E}_1=A_1-A_0$ verifies $|\mathcal{E}_1|\leq|\mathcal{E}|\mathbf{1}_{\Omega_{\mathcal{F},Q_0}}$ and also $\mathcal{E}_1(Y)=0$ if $\delta(Y)<2^{-j}$. This, \eqref{eq:X0-TQ} and \eqref{definicionkappa12}, allow us to invoke  Lemma \ref{proprepresent} (see Remark \ref{modific}) to obtain
\begin{align}\label{hastaqui}
&F_{Q_0}^t(X_0)
:=
|u_1^t(X_0)-u_0^t(X_0)|=\bigg|\iint_{\Omega}(A_0-A_1)(Y)\nabla_Y G_{L_1}(X_0,Y)\cdot\nabla u_0^t(Y)\,dY\bigg|
\\
&\ \ \leq
\sum_{Q\in\mathbb{D}_{\mathcal{F},Q_0}}\sum_{I\in\mathcal{W}_{Q}^*}\iint_{I^*}|\mathcal{E}(Y)||\nabla_Y G_{L_1}(X_0,Y)||\nabla u_0^t(Y)|\,dY,\nonumber
\\
&\ \ \le \sum_{Q\in\mathbb{D}_{\mathcal{F},Q_0}}\sum_{I\in\mathcal{W}_{Q}^*}
\sup_{I^*}|\mathcal{E}|\bigg(\iint_{I^*}|\nabla_Y G_{L_1}(X_0,Y)|^2\,dY\bigg)^{1/2}\bigg(\iint_{I^*}|\nabla u_0^t(Y)|^2\,dY\bigg)^{1/2}.\nonumber
\end{align}
Note that by our choice of $X_0=X_{M_0\Delta_{Q_0}}$, see \eqref{eq:X0-TQ}, the function $v(Y)=G_{L_1}(X_0,Y)$ is a weak solution of $L_1v=0$ in $I^{***}$ for every $I\in\mathcal{W}_{Q}^*$ with $Q\in\mathbb{D}_{Q_0}$. Therefore, by Caccioppoli's and Harnack's inequalities, and Lemma \hyperref[proppde]{\ref*{proppde}(a)}, we obtain
\begin{equation}\label{cotacac}
\iint_{I^*}|\nabla_Y G_{L_1}(X_0,Y)|^2\,dY
\approx
\ell(I)^{n-1} G_{L_1}(X_0,X_{Q})^2\approx \bigg(\frac{\omega_{L_1}^{X_0}(Q)}{\sigma(Q)}\bigg)^2|I|.
\end{equation}
Also, since  $\delta(Y)\approx\ell(I)\approx\ell(Q)$ for every $Y\in I^*$ such that $I\in \mathcal{W}_{Q}^*$,
\begin{equation}\label{eq:term-S}
\iint_{I^*}|\nabla u_0^t(Y)|^2\,dY
\approx
\ell(I)^{-1} \ell(Q)^n\iint_{I^*}|\nabla u_0^t(Y)|^2\delta(Y)^{1-n}\,dY.
\end{equation}
Recalling \eqref{coefcarleson}, \eqref{gammauxiliar}, we define the sequences $\alpha=\{\alpha_Q\}_{Q\in\mathbb{D}_{Q_0}}$, $\beta=\{\beta_Q\}_{Q\in\mathbb{D}_{Q_0}}$ by
\begin{equation}\label{definicionalphabeta}
\alpha_{Q}:=
\frac{\omega_{L_1}^{X_0}(Q)}{\sigma(Q)}\bigg(\ell(Q)^n\iint_{U_{Q}}|\nabla u_0^t(Y)|^2\delta(Y)^{1-n}\,dY\bigg)^{1/2}
\quad\mbox{and}\quad
\beta_{Q}:=\gamma_{\mathcal{F},Q}^{1/2}.
\end{equation}
Using Cauchy-Schwarz's inequality and the bounded overlap
of the cubes $I^*$, one can see that \eqref{cotacac}, \eqref{cotacac}, \eqref{eq:term-S}, and \eqref{definicionalphabeta} yield
\begin{multline}\label{est:FQ0-new}
F_{Q_0}^t(X_0)
\lesssim
\sum_{Q\in\mathbb{D}_{Q_0}}
\frac{\omega_{L_1}^{X_0}(Q)}{\sigma(Q)}\,\gamma_{\mathcal{F},Q}^{1/2}\,\bigg(\ell(Q)^n\iint_{U_Q}|\nabla u_0^t(Y)|^2\delta(Y)^{1-n}\,dY\bigg)^{1/2}
\\
=
\sum_{Q\in\mathbb{D}_{Q_0}} \alpha_Q \,\beta_Q
\lesssim \int_{Q_0}\mathcal{A}_{Q_0}\alpha(x)\mathfrak{C}_{Q_0}\beta(x)\,d\sigma(x).
\end{multline}
where in the last estimate we have used Lemma \ref{tentspaces}, and where we recall that  $\mathcal{A}_{Q_0}$, $\mathfrak{C}_{Q_0}$ were defined in \eqref{definicionA-C}.
Using the bounded overlap property of $U_{Q}$ with $Q\in\mathbb{D}_{Q_0}$, we have that
\begin{multline}\label{cotaA}
\mathcal{A}_{Q_0}\alpha(x)
=
\bigg(\sum_{x\in Q\in\mathbb{D}_{Q_0}}
\bigg(\frac{\omega_{L_1}^{X_0}(Q)}{\sigma(Q)}\bigg)^{2}
\iint_{U_{Q}}|\nabla u_0^t(Y)|^2\delta(Y)^{1-n}\,dY\bigg)^{1/2}
\\
\lesssim
M^d_{Q_0} k_{L_1}^{X_0}(x)\,
S_{Q_0}u_0^t(x),
\end{multline}
where
\begin{equation}\label{definicionmaximal}
  M_{Q_0}^df(x):=\sup_{x\in Q\in\mathbb{D}_{Q_0}}\barint_{Q}|f(y)|\,d\sigma(y).
\end{equation}

On the other hand,  \eqref{hipepsilon0} yields
\begin{align}\label{cotaC}
\mathfrak{C}_{Q_0} \beta(x)
&=
\sup_{x\in Q\in\mathbb{D}_{Q_0}}\bigg(\frac{1}{\sigma(Q)}\sum_{Q'\in\mathbb{D}_{Q}}\gamma_{\mathcal{F},Q'}\bigg)^{1/2}
\le
\|\mathfrak{m}_{\mathcal{F}}\|_{\mathcal{C}(Q_0)}^{1/2}
\le \varepsilon_0^{1/2}
.
\end{align}
Plugging \eqref{cotaA},  \eqref{cotaC} into \eqref{est:FQ0-new}, using Hölder's inequality           we conclude that
\begin{align}\label{eq:conc1}
F_{Q_0}^t(X_0)
\lesssim
\varepsilon_0^{1/2} \,\|S_{Q_0}u_0^t\|_{L^{p'}(Q_0)}\,
\| M_{Q_0}^d k_{L_1}^{X_0}\|_{L^{p}(Q_0)}
\lesssim
\varepsilon_0^{1/2}
\big\| k_{L_1}^{X_0}\big\|_{L^p(Q_0)},
\end{align}
where we have used that $M^d_{Q_0}$ is bounded in $L^{p}(Q_0)$  and that
$$
\|S_{Q_0}u_0^t\|_{L^{p'}(Q_0)}\lesssim \|\widetilde{\mathcal{N}}_{Q_0,*}u_0^t\|_{L^{p'}(Q_0)}\lesssim\|g_t\|_{L^{p'}(Q_0)}\lesssim\|g\|_{L^{p'}(Q_0)}=1,
$$
which follows from \eqref{eq:S<N}, Lemma \hyperref[solvability]{\ref*{solvability}(a)}, $\omega_{L_0}\in RH_p(\partial\Omega)$,    \eqref{eq:defi-u0-u1}, and Lemma \ref{densidad}.

From \eqref{eq:conc1}, and for all $0<t<C r_{Q_0}/2$,
$$
0\leq u_1^t(X_0)\leq F_{Q_0}^t(X_0)+u_0^t(X_0)\lesssim\varepsilon_0^{1/2}\|k_{L_1}^{X_0}\|_{L^p(Q_0)}+\|k_{L_0}^{X_0}\|_{L^{p}(2\widetilde{\Delta}_{Q_0})},
$$
where we have used that $\|g_t\|_{L^{p'}(\partial\Omega)}\lesssim 1$ and  Lemma \ref{densidad}, and the implicit constants do not depend on $t$. Next, using the previous estimate,
\begin{multline*}
\int_{\partial\Omega}g(y)k_{L_1}^{X_0}(y)\,d\sigma(y)=u_1^t(X_0)+\int_{\partial\Omega}(g(y)-g_t(y)) k_{L_1}^{X_0}(y)\,d\sigma(y)
\\
\lesssim \varepsilon_0^{1/2}\|k_{L_1}^{X_0}\|_{L^p(Q_0)}+\|k_{L_0}^{X_0}\|_{L^{p}(2\widetilde{\Delta}_{Q_0})}+\|g-g_t\|_{L^{p'}(\partial\Omega)}\|k_{L_1}^{X_0}\|_{L^p(2\widetilde{\Delta}_{Q_0})}.
\end{multline*}
Note that $\|k_{L_1}^{X_0}\|_{L^p(2\widetilde{\Delta}_{Q_0})}\leq C_j<\infty$ by Lemma \hyperref[proppde]{\ref*{proppde}(e)} and Harnack's inequality ($L_0\equiv L_1$ in $\{Y\in\Omega:\:\delta(Y)< 2^{-j}\}$). Recall that $\|g-g_t\|_{L^{p'}(\partial\Omega)}$ as $t\rightarrow 0$ (see Lemma \ref{densidad}) and hence
$$
\int_{\partial\Omega}g(y)k_{L_1}^{X_0}(y)\,d\sigma(y)\lesssim \varepsilon_0^{1/2}\|k_{L_1}^{X_0}\|_{L^p(Q_0)}+\|k_{L_0}^{X_0}\|_{L^{p}(2\widetilde{\Delta}_{Q_0})}.
$$
Taking the supremum over $0\leq g\in L^{p'}(Q_0)$ with $\|g\|_{L^{p'}(Q_0)}=1$, the latter implies
$$
\|k_{L_1}^{X_0}\|_{L^p(Q_0)}\leq C\varepsilon_0^{1/2}\|k_{L_1}^{X_0}\|_{L^p(Q_0)}+ C\|k_{L_0}^{X_0}\|_{L^p(2\widetilde{\Delta}_{Q_0})},
$$
with $C$ depending only on dimension, $p$, the 1-sided $\mathrm{CAD}$ constants, the ellipticity of $L_0$ and $L$, and the constant in $\omega_{L_0}\in RH_p(\partial\Omega)$.
As mentioned above, $\|k_{L_1}^{X_0}\|_{L^p(Q_0)}\leq C_j<\infty$, thus taking $\varepsilon_0<C^{-2}/4$ we can hide the first term in the left hand side, and consequently $\|k_{L_1}^{X_0}\|_{L^p(Q_0)}\lesssim\|k_{L_0}^{X_0}\|_{L^p(2\widetilde{\Delta}_{Q_0})}$. Recalling  that $X_0=X_{M_0\Delta_{Q_0}}$ we have that $\delta(X_{Q_0})\approx \ell(Q_0)$, $\delta(X_0)\approx M_0\ell(Q_0)\geq\ell(Q_0)$, $\delta(X_{2\widetilde{\Delta}_{Q_0}})\approx  \ell(Q_0)$. Also, $|X_0-X_{Q_0}|+ |X_0-X_{2\widetilde{\Delta}_{Q_0}}|\lesssim M_0\ell(Q_0)$. Hence, using Harnack's inequality
(with constants depending on $M_0$, which has been already fixed), and the fact that $\omega_{L_0}\in RH_p(\partial\Omega)$, we conclude that
\begin{multline}\label{rhpL1}
\int_{Q_0}k_{L_1}^{X_{Q_0}}(y)^p\,d\sigma(y)
\approx
\int_{Q_0}k_{L_1}^{X_0}(y)^p\,d\sigma(y)\lesssim\int_{2\widetilde{\Delta}_{Q_0}}k_{L_0}^{X_0}(y)^p\,d\sigma(y)
\\
\approx
\int_{2\widetilde{\Delta}_{Q_0}}k_{L_0}^{X_{2\widetilde{\Delta}_{Q_0}}}(y)^p\,d\sigma(y)\lesssim \sigma(2\widetilde{\Delta}_{Q_0})^{1-p}\approx \sigma(Q_0)^{1-p}.
\end{multline}
\smallskip

\subsection{Self-improvement of Step 1}
The goal of this section is to extend \eqref{rhpL1} and show that it holds with the integration taking place in an arbitrary $Q\in\mathbb{D}_{Q_0}$, but with the pole of the elliptic measure being $X_{Q_0}$. In doing this, we will lose the exponent $p$, showing that a $RH_q$ inequality holds for some fixed $q$.

Fix $Q\in\mathbb{D}_{Q_0}$, and let $L_1^Q$ be the operator defined by $L_1^Qu=-\div(A_1^Q\nabla u)$, where
$$
A_1^Q(Y):=\left\{
             \begin{array}{ll}
               \widetilde{A}(Y) & \hbox{$\text{if }\,Y\in\Omega_{\mathcal{F},Q}$}, \\
               A_0(Y) & \hbox{$\text{if }\,Y\in\Omega\setminus \Omega_{\mathcal{F},Q}$},
             \end{array}
           \right.
$$
with $\widetilde{A}=\widetilde{A}^j$ as in \eqref{definicionLtilde}. Since $L_1^Q\equiv L_0$ in $\{Y\in\Omega:\,\delta(Y)<2^{-j}\}$, Lemma \hyperref[proppde]{\ref*{proppde}(e)} implies that $\omega_{L_1^Q}\ll\sigma$, hence there exists $k_{L_1^Q}^X=d\omega_{L_1^{Q}}^X/d\sigma$. Our first goal is to obtain
\begin{equation}\label{rhpL1Q}
\int_{Q}k_{L_1^Q}^{X_{Q}}(y)^p\,d\sigma(y)\lesssim \sigma(Q)^{1-p}.
\end{equation}
We consider two cases. Suppose first that $Q\subset Q_i$ for some $Q_i\in\mathcal{F}$, then $\Omega_{\mathcal{F},Q}=\emptyset$, $L_1^Q\equiv L_0$ in $\Omega$, and \eqref{rhpL1Q} is a consequence of the fact that $\omega_{L_0}\in RH_p(\partial\Omega)$. In other case, that is, if $Q\in\mathbb{D}_{\mathcal{F},Q_0}$, we define $\mathcal{F}_Q=\{Q_i\in\mathcal{F}:\:Q_i\cap Q\neq\emptyset\}=\{Q_i\in\mathcal{F}:\:Q_i\subsetneq Q\}$. Note that $A_0-A_1^Q$ is supported in $\Omega_{\mathcal{F}_Q,Q}=\Omega_{\mathcal{F},Q}$, and clearly
$$
\|\mathfrak{m}_{\mathcal{F}_Q}\|_{\mathcal{C}(Q)}=\sup_{Q'\in\mathbb{D}_Q}\frac{\mathfrak{m}_{\mathcal{F}_Q}(\mathbb{D}_{Q'})}{\sigma(Q')}\leq
\sup_{Q'\in\mathbb{D}_{Q_0}}\frac{\mathfrak{m}_{\mathcal{F}}(\mathbb{D}_{Q'})}{\sigma(Q')}\leq\varepsilon_0.
$$
We can then repeat the argument of Step $1$ for the operator $L_1^Q$ replacing $L_1$, and with $Q$ and $\mathcal{F}_Q$ in place of respectively $Q_0$ and $\mathcal{F}$. Hence, the estimate \eqref{rhpL1} becomes \eqref{rhpL1Q}.

We next notice that using \cite[Lemma 3.55]{hofmartell}, there exists $0<\widehat{\kappa}_1<\kappa_1$ (see \eqref{definicionkappa12}), depending only on the allowable parameters, such that
$\widehat{\kappa}_1 B_Q\cap\Omega_{\mathcal{F},Q_0}=\widehat{\kappa}_1 B_Q\cap\Omega_{\mathcal{F},Q}$.
This easily gives that $L_1\equiv L_1^{Q}$ in $\widehat{\kappa}_1B_Q\cap\Omega$. Using now Lemma \hyperref[proppde]{\ref*{proppde}(e)} and Harnack's inequality, we have \begin{equation}\label{competa}
k_{L_1}^{X_Q}(y)\approx k_{L_1^Q}^{X_Q}(y),\quad\text{ for }\sigma\text{-a.e. }y\in \eta\Delta_Q,
\end{equation}
where $\eta=\widehat{\kappa}_1/(2\kappa_0)$ and $\kappa_0$ is as in \eqref{definicionkappa0}, and hence $\eta\Delta_Q\subset \Delta_Q\subset Q$. Combining \eqref{rhpL1Q}, \eqref{competa}, Lemma \ref{bourgain}, Lemma \hyperref[proppde]{\ref*{proppde}(b)} and Harnack's inequality we obtain
\begin{multline*}
\bigg(\barint_{\eta\Delta_Q}k_{L_1}^{X_Q}(y)^p\,d\sigma(y)\bigg)^{1/p}
\lesssim
\bigg(\barint_{Q}k_{L_1^Q}^{X_Q}(y)^p\,d\sigma(y)\bigg)^{1/p}
\lesssim
\sigma(Q)^{-1}
\\
\lesssim\sigma(Q)^{-1}\omega_{L_1}^{X_Q}(Q)\lesssim\barint_{\eta\Delta_Q}k_{L_1}^{X_Q}(y)\,d\sigma(y).
\end{multline*}
Now, using Remark \ref{dyadicchangepole} we have that
\begin{equation}\label{endselfimprovement}
\bigg(\barint_{\eta\Delta_Q}k_{L_1}^{X_{Q_0}}(y)^p\,d\sigma(y)\bigg)^{1/p}\leq C_1\barint_{\eta\Delta_Q}k_{L_1}^{X_{Q_0}}(y)\,d\sigma(y),
\end{equation}
with $C_1>1$ depending only on dimension, $p$, the 1-sided $\mathrm{CAD}$ constants, the ellipticity of $L_0$ and $L$, and the constant in $\omega_{L_0}\in RH_p(\partial\Omega)$. Note that \eqref{endselfimprovement} holds then for every $Q\in\mathbb{D}_{Q_0}$. Also, by means of Lemma \ref{bourgain}, Lemma \hyperref[proppde]{\ref*{proppde}(b)} and Harnack's inequality, there exists $C_\eta>1$ such that $0<\omega_{L_1}^{X_{Q_0}}(Q)\leq C_\eta\omega_{L_1}^{X_{Q_0}}(\eta\Delta_Q)$ for every $Q\in\mathbb{D}_{Q_0}$. Using Lemma \ref{lemmab7} we obtain that $\omega_{L_1}^{X_{Q_0}}\in A_\infty^{\rm dyadic}(Q_0)$. This and Lemma  \hyperref[lemm_w-Pw-:properties]{\ref*{lemm_w-Pw-:properties}(b)} yield $\mathcal{P}_{\mathcal{F}}\omega_{L_1}^{X_{Q_0}}\in A_{\infty}^{\rm dyadic}(Q_0)$.
\medskip

\subsection{Step 2}
We define a new operator $L_2$ by changing $L_1$ below the region $\Omega_{\mathcal{F},Q_0}$. More precisely, set $L_2u=-\div(A_2\nabla u)$ with
$$
A_2(Y):=\left\{
             \begin{array}{ll}
               \widetilde{A}(Y) & \hbox{$\text{if }\,Y\in T_{Q_0}\setminus\Omega_{\mathcal{F},Q_0}$,} \\
               A_1(Y) & \hbox{$\text{if }\,Y\in \Omega\setminus(T_{Q_0}\setminus\Omega_{\mathcal{F},Q_0}).$}
             \end{array}
           \right.
$$
Note that by construction, $A_2=\widetilde{A}$ in $T_{Q_0}$ and $A_2=A_0$ in $\Omega\setminus T_{Q_0}$. Our goal is to prove that $\mathcal{P}_{\mathcal{F}}\omega_{L_2}^{X_{Q_0}}\in A_{\infty}^{\rm dyadic}(Q_0)$ by using Lemma \ref{sawtoothlemma}. For $k=1,2$, we write $\omega_{L_k}=\omega_{L_k,\Omega}^{A_{Q_0}}$ and $\omega_{L_k,\star}=\omega_{L_k,\Omega_{\mathcal{F},Q_0}}^{A_{Q_0}}$ for the elliptic measures of $L_k$ with respect to the domains $\Omega$ and $\Omega_{\mathcal{F},Q_0}$, with fixed pole at $A_{Q_0}$ (see \cite[Proposition 6.4]{hofmartell}). Note that since $A_1= A_2$ in $\Omega_{\mathcal{F},Q_0}$ then $\omega_{L_1,\star}=\omega_{L_2,\star}$. Finally let $\nu_{L_k}=\nu_{L_k}^{A_{Q_0}}$ be the corresponding measures defined as in \eqref{eq:def-nu}, and observe that \eqref{eq:def-nu:P} imply that $\mathcal{P}_{\mathcal{F}}\nu_{L_1}=\mathcal{P}_{\mathcal{F}}\nu_{L_2}$.

In Step 1 we have shown that $\mathcal{P}_\mathcal{F}\omega_{L_1}^{X_{Q_0}}\in A_{\infty}^{\text{dyadic}}(Q_0)$, thus Harnack's inequality and \eqref{ainfsawtooth} give that $\mathcal{P}_{\mathcal{F}}\nu_{L_2}=\mathcal{P}_{\mathcal{F}}\nu_{L_1}\in A_{\infty}^{\text{dyadic}}(Q_0)$. Another use of \eqref{ainfsawtooth} and Harnack's inequality allows us to obtain that $\mathcal{P}_{\mathcal{F}}\omega_{L_2}^{X_{Q_0}} \approx   \mathcal{P}_{\mathcal{F}}\omega_{L_2}\in A_{\infty}^{\rm dyadic}(Q_0)$. Note that by Lemma \hyperref[proppde]{\ref*{proppde}(b)}, Harnack's inequality and Lemma \hyperref[lemm_w-Pw-:properties]{\ref*{lemm_w-Pw-:properties}(a)} it follows that $\mathcal{P}_{\mathcal{F}}\omega_{L_2}^{X_{Q_0}}$ is dyadically doubling in $Q_0$. Thus, \cite[Lemma B.7]{hofmartell} implies that there exist $\theta,\theta'>0$ such that
\begin{equation}\label{finpaso2}
\bigg(\frac{\sigma(E)}{\sigma(Q)}\bigg)^{\theta}\lesssim
\frac{\mathcal{P}_\mathcal{F}\omega_{L_2}^{X_{Q_0}}(E)}{\mathcal{P}_\mathcal{F}\omega_{L_2}^{X_{Q_0}}(Q)}\lesssim
\bigg(\frac{\sigma(E)}{\sigma(Q)}\bigg)^{\theta'},\qquad Q\in\mathbb{D}_{Q_0},\qquad E\subset Q.
\end{equation}
\medskip

\subsection{Step 3}
To complete the proof it remains to change the operator outside $T_{Q_0}$. Let us introduce $L_3u=-\div(A_3\nabla u)$, where
$$
A_3(Y):=\left\{
             \begin{array}{ll}
               A_2(Y) & \hbox{$\text{if }\,Y\in T_{Q_0}$,} \\
               \widetilde{A}(Y) & \hbox{$\text{if }\,Y\in \Omega\setminus T_{Q_0},$}
             \end{array}
           \right.
$$
and note that $L_3\equiv\widetilde{L}$ in $\Omega$.

We want to prove that for every $0<\varepsilon<1$, there exists $C_\varepsilon>1$ such that
\begin{equation}\label{concpaso3}
E\subset Q_0,\quad\frac{\sigma(E)}{\sigma(Q_0)}\geq\varepsilon\implies\frac{\mathcal{P}_\mathcal{F}\omega_{L_3}^{X_{Q_0}}(E)}{\mathcal{P}_\mathcal{F}\omega_{L_3}^{X_{Q_0}}(Q_0)}\geq\frac{1}{C_\varepsilon}.
\end{equation}
Let $0<\varepsilon<1$ and let $E\subset Q_0$ be such that
$\sigma(E)\geq\varepsilon\sigma(Q_0)$. First, we can disregard the trivial case
$\mathcal{F}=\{Q_0\}$:
$$
\frac{\mathcal{P}_\mathcal{F}\omega_{L_3}^{X_{Q_0}}(E)}{\mathcal{P}_\mathcal{F}\omega_{L_3}^{X_{Q_0}}(Q_0)}=
\frac{\frac{\sigma(E)}{\sigma(Q_0)}\,\omega_{L_3}^{X_{Q_0}}(Q_0)}{\frac{\sigma(Q_0)}{\sigma(Q_0)}\,\omega_{L_3}^{X_{Q_0}}(Q_0)}=\frac{\sigma(E)}{\sigma(Q_0)}\geq\varepsilon.
$$

Suppose then that $\mathcal{F}\subsetneq\mathbb{D}_{Q_0}\setminus\{Q_0\}$. For $\tau\ll 1$ we consider the sets
$$
\Sigma_\tau:=\big\{x\in Q_0:\:\dist(x,\partial\Omega\setminus Q_0)<\tau\ell(Q_0)\big\}
$$
and $\widetilde{Q_0}:=Q_0\setminus\bigcup_{Q'\in\mathcal{I}_\tau}Q'$, where
$$
\mathcal{I}_{\tau}=\big\{Q'\in\mathbb{D}_{Q_0}:\:\tau\ell(Q_0)<\ell(Q')\leq 2\tau\ell(Q_0),\ Q'\cap\Sigma_{\tau}\neq\emptyset\big\}.
$$
By construction, $\Sigma_{\tau}\subset \bigcup_{Q'\in\mathcal{I}_\tau}Q'$, and there exists $C=C(n,\mathrm{AR})>0$ such that every $Q'\in\mathcal{I}_\tau$ satisfies $Q'\subset\Sigma_{C\tau}$.  Using Lemma \hyperref[dyadiccubes]{\ref*{dyadiccubes}(f)}, for $\tau=\tau(\varepsilon)>0$ sufficiently small we have
$$
\sigma(Q_0\setminus\widetilde{Q_0})\leq\sigma(\Sigma_{C\tau})\leq C_1(C\tau)^\eta\sigma(Q_0)\leq\frac{\varepsilon}{2}\,\sigma(Q_0),
$$
and letting $F=E\cap\widetilde{Q_0}$, it follows that
$$
\varepsilon\sigma(Q_0)\leq\sigma(E)\leq\sigma(F)+\sigma(Q_0\setminus\widetilde{Q_0})\leq\sigma(F)+\frac{\varepsilon}{2}\sigma(Q_0).
$$
Hence $\sigma(F)/\sigma(Q_0)\geq\varepsilon/2$ and by \eqref{finpaso2}, we conclude that
\begin{equation}\label{utpasant}
\frac{\mathcal{P}_\mathcal{F}\omega_{L_2}^{X_{Q_0}}(F)}{\mathcal{P}_\mathcal{F}\omega_{L_2}^{X_{Q_0}}(Q_0)}\gtrsim
\bigg(\frac{\sigma(F)}{\sigma(Q_0)}\bigg)^\theta\geq
\Big(\frac{\varepsilon}{2}\Big)^\theta.
\end{equation}

We claim that there exists $c_\varepsilon>0$ such that $\mathcal{P}_\mathcal{F}\omega_{L_3}^{X_{Q_0}}(F)\geq
c_\varepsilon\mathcal{P}_\mathcal{F}\omega_{L_2}^{X_{Q_0}}(F)$. Assuming this momentarily, we easily obtain \eqref{concpaso3}:
$$
\frac{\mathcal{P}_\mathcal{F}\omega_{L_3}^{X_{Q_0}}(E)}{\mathcal{P}_\mathcal{F}\omega_{L_3}^{X_{Q_0}}(Q_0)}\geq\mathcal{P}_\mathcal{F}\omega_{L_3}^{X_{Q_0}}(F)\geq c_{\varepsilon}
\mathcal{P}_\mathcal{F}\omega_{L_2}^{X_{Q_0}}(F)\gtrsim c_\varepsilon\frac{\mathcal{P}_\mathcal{F}\omega_{L_2}^{X_{Q_0}}(F)}{\mathcal{P}_\mathcal{F}\omega_{L_2}^{X_{Q_0}}(Q_0)}\geq
c_\varepsilon\Big(\frac{\varepsilon}{2}\Big)^\theta=:\frac{1}{C_\varepsilon},
$$
where we have used Lemma \ref{bourgain}, \eqref{utpasant}, and the fact that $\mathcal{P}_\mathcal{F}\omega_{L_k}^{X_{Q_0}}(Q_0)=\omega_{L_k}^{X_{Q_0}}(Q_0)$ for $k=2,3$.

Let us then show our claim. First, since $L_2\equiv L_3$ in $T_{Q_0}$ and $\widetilde{Q_0}\subset Q_0\setminus\Sigma_\tau$, Lemma \ref{comp3} yields
\begin{equation}\label{mod39}
k_{L_2}^{X_{Q_0}}(y)\approx_\tau k_{L_3}^{X_{Q_0}}(y),\quad\text{ for }\sigma\text{-a.e. }y\in\widetilde{Q_0}.
\end{equation}
This and the fact that $F\subset\widetilde{Q_0}$ give
$$
\omega_{L_2}^{X_{Q_0}}\Big(F\setminus\bigcup_{Q_i\in\mathcal{F}}Q_i\Big)\approx_\tau
\omega_{L_3}^{X_{Q_0}}\Big(F\setminus\bigcup_{Q_i\in\mathcal{F}}Q_i\Big),
$$
which in turn yields
\begin{multline}\label{deesigualdaa}
\mathcal{P}_\mathcal{F}\omega_{L_3}^{X_{Q_0}}(F)
=
\omega_{L_3}^{X_{Q_0}}\Big(F\setminus\bigcup_{Q_i\in\mathcal{F}}Q_i\Big)+\sum_{Q_i\in\mathcal{F}}\frac{\sigma(F\cap Q_i)}{\sigma(Q_i)}\,\omega_{L_3}^{X_{Q_0}}(Q_i)
\\
\geq
c_\tau\omega_{L_2}^{X_{Q_0}}\Big(F\setminus\bigcup_{Q_i\in\mathcal{F}}Q_i\Big)+\sum_{Q_i\in\mathcal{F}}\frac{\sigma(F\cap Q_i)}{\sigma(Q_i)}\,\omega_{L_3}^{X_{Q_0}}(Q_i).
\end{multline}
It remains to estimate the second term.
Note that in the sum we can restrict ourselves to those cubes $Q_i\in\mathcal{F}$ such that $F\cap Q_i\neq\emptyset$. We consider two cases. If
$Q_i\subset\widetilde{Q_0}$, using \eqref{mod39} we have that $\omega_{L_3}^{X_{Q_0}}(Q_i)\approx_\tau\omega_{L_2}^{X_{Q_0}}(Q_i)$.
Otherwise, if $Q_i\setminus\widetilde{Q_0}\neq\emptyset$, there exists $Q'\in\mathcal{I}_\tau$ such that $Q_i\cap Q'\neq\emptyset$. Then $Q'\subsetneq Q_i$ (if $Q_i\subset Q'$ then  $Q_i\subset Q_0\setminus\widetilde{Q_0}$, contradicting that $F\cap Q_i\neq\emptyset$ and $F\subset\widetilde{Q_0}$) and, in particular, $\ell(Q_i)>\tau\ell(Q_0)$. Let $x_{Q_i}$ be the center of $Q_i$, and let $\Delta_{Q_i}=\Delta(x_{Q_i},r_{Q_i})$ with $r_{Q_i}\approx\ell(Q_i)$ as in \eqref{deltaQ2}. Take $\widehat{Q_i}\in\mathbb{D}_{Q_i}$ with $x_{Q_i}\in\widehat{Q_i}$, $\ell(\widehat{Q_i})=2^{-M}\ell(Q_i)$ and $M>1$ to be chosen. Notice that
$\diam(\widehat{Q_i})\approx 2^{-M}\ell(Q_i)\approx 2^{-M}r_{Q_i}$ and clearly
\begin{multline*}
r_{Q_i}\leq\dist(x_{Q_i},\partial\Omega\setminus\Delta_{Q_i})
\leq
\diam(\widehat{Q_i})+\dist(\widehat{Q_i},\partial\Omega\setminus\Delta_{Q_i})
\\
\approx 2^{-M}r_{Q_i}+\dist(\widehat{Q_i},\partial\Omega\setminus\Delta_{Q_i}).
\end{multline*}
Taking $M\gg 1$ large enough (depending on the $\mathrm{AR}$ constant), we conclude that
$c\tau\ell(Q_0)<\dist(\widehat{Q_i},\partial\Omega\setminus\Delta_{Q_i})\le \dist(\widehat{Q_i}, \partial\Omega\setminus Q_0)$ and hence $\widehat{Q_i}\subset Q_0\setminus\Sigma_{c\tau}$. Again, using Lemma \ref{comp3} and the fact that
$\omega_{L_2}^{X_{Q_0}}$ is doubling in $Q_0$ (which is a consequence of Lemma \hyperref[proppde]{\ref*{proppde}(b)} and Harnack's inequality), we obtain
$$
\omega_{L_3}^{X_{Q_0}}(Q_i)\geq\omega_{L_3}^{X_{Q_0}}(\widehat{Q_i})\approx_{\tau}\omega_{L_2}^{X_{Q_0}}(\widehat{Q_i})
\gtrsim
\omega_{L_2}^{X_{Q_0}}(Q_i).
$$
In the two cases, since $\tau=\tau(\varepsilon)$, \eqref{deesigualdaa} turns into
$$
\mathcal{P}_\mathcal{F}\omega_{L_3}^{X_{Q_0}}(F)\gtrsim_{\varepsilon}
\omega_{L_2}^{X_{Q_0}}\Big(F\setminus\bigcup_{Q_i\in\mathcal{F}}Q_i\Big)+
\sum_{Q_i\in\mathcal{F}}\frac{\sigma(Q_i\cap
F)}{\sigma(Q_i)}\omega_{L_2}^{X_{Q_0}}(Q_i)
=
\mathcal{P}_\mathcal{F}\omega_{L_2}^{X_{Q_0}}(F),
$$
completing the proof of our claim.

\medskip

Recalling that $\widetilde{L}\equiv L_3$, the previous argument proves the following proposition:
\begin{proposition}\label{prop:Step 3}
There exists $\varepsilon_0>0$ (depending only on dimension, $p$, the 1-sided $\mathrm{CAD}$ constants, the ellipticity of $L_0$ and $L$, and the constant in $\omega_{L_0}\in RH_p(\partial\Omega)$) such that the following property holds: given $\varepsilon\in(0,1)$, there exists $C_\varepsilon>1$ such that for every $Q_0\in\mathbb{D}(\partial\Omega)$ with $\ell(Q_0)<\diam(\partial\Omega)/M_0$ and every $\mathcal{F}=\{Q_i\}\subset\mathbb{D}_{Q_0}$ with $\|\mathfrak{m}_\mathcal{F}\|_{\mathcal{C}(Q_0)}\leq\varepsilon_0$, there holds
\begin{equation}\label{concpaso33}
E\subset Q_0,\quad\frac{\sigma(E)}{\sigma(Q_0)}\geq\varepsilon\implies
\frac{\mathcal{P}_\mathcal{F}\omega_{\widetilde{L}}^{X_{Q_0}}(E)}{\mathcal{P}_\mathcal{F}\omega_{\widetilde{L}}^{X_{Q_0}}(Q_0)}\geq\frac{1}{C_\varepsilon},
\end{equation}
where $\widetilde{L}=L^j$ is the operator defined in \eqref{definicionLtilde} and $j\in\mathbb{N}$ is arbitrary.
\end{proposition}

\medskip

\subsection{Step 4}
What we have proved so far does not allow us to apply Lemma \ref{extrapolation}. We have to be able to fix the pole relative to $Q_0$, and show that \eqref{concpaso33} also holds for all $Q\in\mathbb{D}_{Q_0}$.

\begin{proposition}
Let $\varepsilon_0$ be the parameter obtained in Proposition \ref{prop:Step 3}. Given $\varepsilon\in(0,1)$, there exists $C_\varepsilon>1$ such that for every $Q_0\in\mathbb{D}(\partial\Omega)$ with $\ell(Q_0)<\diam(\partial\Omega)/M_0$, every $Q\in\mathbb{D}_{Q_0}$, every $\mathcal{F}=\{Q_i\}\subset\mathbb{D}_{Q}$ with $\|\mathfrak{m}_\mathcal{F}\|_{\mathcal{C}(Q)}\leq\varepsilon_0$, there holds
\begin{equation}\label{propstep4}
E\subset Q,\quad\frac{\sigma(E)}{\sigma(Q)}\geq\varepsilon\implies\frac{\mathcal{P}_\mathcal{F}\omega_{\widetilde{L}}^{X_{Q_0}}(E)}{\mathcal{P}_\mathcal{F}\omega_{\widetilde{L}}^{X_{Q_0}}(Q)}\geq\frac{1}{C_\varepsilon},
\end{equation}
where $\widetilde{L}=L^j$ is the operator defined in \eqref{definicionLtilde} and $j\in\mathbb{N}$ is arbitrary. Consequently, there exists $1<q<\infty$ such that $\omega_{\widetilde{L}}^{X_{Q_0}}\in RH_q^{\rm dyadic}(Q_0)$ uniformly in $Q_0\in\mathbb{D}(\partial\Omega)$ provided $\ell(Q_0)<\diam(\partial\Omega)/M_0$, and moreover $\omega_{\widetilde{L}}\in RH_q(\partial\Omega)$
.
\end{proposition}

\begin{proof}
Fix $Q_0\in\mathbb{D}(\partial\Omega)$ with $\ell(Q_0)<\diam(\partial\Omega)/M_0$. Let $0<\varepsilon<1$, $Q\in\mathbb{D}_{Q_0}$. Let $\mathcal{F}=\{Q_i\}\subset\mathbb{D}_{Q}$ be such that $\|\mathfrak{m}_\mathcal{F}\|_{\mathcal{C}(Q)}\leq\varepsilon_0$ and let $E\subset Q$ satisfy $\sigma(E)\geq\varepsilon\sigma(Q)$. By Lemma \hyperref[proppde]{\ref*{proppde}(c)} (see also Remark \ref{dyadicchangepole}) and the fact that $\mathcal{P}_\mathcal{F}\omega_{\widetilde{L}}^{X_Q}(Q)=\omega_{\widetilde{L}}^{X_Q}(Q)\approx 1$ by Lemma \ref{bourgain}, we see that
\begin{multline*}
\frac{\mathcal{P}_{\mathcal{F}}\omega_{\widetilde{L}}^{X_{Q_0}}(E)}{\mathcal{P}_{\mathcal{F}}\omega_{\widetilde{L}}^{X_{Q_0}}(Q)}
=
\frac{\mathcal{P}_{\mathcal{F}}\omega_{\widetilde{L}}^{X_{Q_0}}(E)}{\omega_{\widetilde{L}}^{X_{Q_0}}(Q)}
\approx
\omega_{\widetilde{L}}^{X_Q}\Big(E\setminus\bigcup_{Q_i\in\mathcal{F}}Q_i\Big)
+\sum_{Q_i\in\mathcal{F}}\frac{\sigma(E\cap Q_i)}{\sigma(Q_i)}\,\omega_{\widetilde{L}}^{X_Q}(Q_i)
\\
=
\mathcal{P}_\mathcal{F}\omega_{\widetilde{L}}^{X_Q}(E)
\approx
\frac{\mathcal{P}_\mathcal{F}\omega_{\widetilde{L}}^{X_Q}(E)}{\mathcal{P}_\mathcal{F}\omega_{\widetilde{L}}^{X_Q}(Q)}\geq\frac{1}{C_\varepsilon},
\end{multline*}
where in the last inequality we have applied Proposition \ref{prop:Step 3} to $Q$ (replacing $Q_0$) satisfying
$\ell(Q)<\diam(\partial\Omega)/M_0$. This shows \eqref{propstep4}, which together with Lemma \ref{carlesondisc} and our choice of $M_0$, allows us to invoke Lemma \ref{extrapolation} and eventually conclude that $\omega_{\widetilde{L}}^{X_{Q_0}}\in A_\infty^{\rm dyadic}(Q_0)$.

We have then proved that $\omega_{\widetilde{L}}^{X_{Q_0}}\in A_\infty^{\rm dyadic}(Q_0)$ uniformly in $Q_0$, provided  $\ell(Q_0)<\diam(\partial\Omega)/M_0$. Thus, there exists $1<q<\infty$, such that $\omega_{\widetilde{L}}^{X_{Q_0}}\in RH_q^{\rm dyadic}(Q_0)$ uniformly in $Q_0$ for the same class of cubes and, in particular,
\begin{equation}\label{eq:step4-end}
\int_{Q_0}k_{\widetilde{L}}^{X_{Q_0}}(y)^q\,d\sigma(y)
\lesssim
\sigma(Q_0)^{1-q},
\qquad
Q_0\in\mathbb{D}(\partial\Omega),\quad \ell(Q_0)<\frac{\diam(\partial\Omega)}{M_0}.
\end{equation}

When $\diam(\partial\Omega)<\infty$, we need to extend the previous estimate to all cubes with sidelength of the order of $\diam(\partial\Omega)$. Let us then take $Q_0\in\mathbb{D}(\partial\Omega)$ with $\ell(Q_0)\geq\diam(\partial\Omega)/M_0$ and define the collection
$$
\mathcal{I}_{Q_0}=\Big\{Q\in\mathbb{D}_{Q_0}:\:\frac{\diam(\partial\Omega)}{2M_0}\leq\ell(Q)<\frac{\diam(\partial\Omega)}{M_0}\Big\}.
$$
Note that $Q_0=\bigcup_{Q\in\mathcal{I}_{Q_0}}Q$ is a disjoint union and using the $\mathrm{AR}$ property we have that
$$
\#\mathcal{I}_{Q_0}\bigg(\frac{\diam(\partial\Omega)}{2M_0}\bigg)^n
\leq
\sum_{Q\in\mathcal{I}_{Q_0}}\ell(Q)^n
\approx
\sum_{Q\in\mathcal{I}_{Q_0}} \sigma(Q)
=
\sigma(Q_0)
\approx
\ell(Q_0)^n
\lesssim
\diam(\partial\Omega)^n,
$$
which implies $\#\mathcal{I}_{Q_0}\lesssim M_0^n$. We can use Harnack's inequality to move the pole from $X_{Q_0}$ to $X_{Q}$ for any $Q\in \mathcal{I}_{Q_0}$ (with constants depending on $M_0$, which is already fixed), since $\delta(X_{Q_0})\approx\ell(Q_0)>\ell(Q)$, $\delta(X_{Q})\approx\ell(Q)$ and $|X_{Q_0}-X_{Q}|\lesssim M_0\ell(Q)$. Hence, we obtain
\begin{multline*}
\int_{Q_0}k_{\widetilde{L}}^{X_{Q_0}}(y)^q\,d\sigma(y)
\approx
\sum_{Q\in\mathcal{I}_{Q_0}}\int_{Q}k_{\widetilde{L}}^{X_{Q}}(y)^q\,d\sigma(y)
\lesssim
\sum_{Q\in\mathcal{I}_{Q_0}}\sigma(Q)^{1-q}
\\
\lesssim
\#\mathcal{I}_{Q_0}\diam(\partial\Omega)^{(1-q)n}\lesssim\sigma(Q_0)^{1-q},
\end{multline*}
where we have used \eqref{eq:step4-end} for $Q$ since  $\ell(Q)<\diam(\partial\Omega)/M_0$, and the AR property.
Therefore, we have extended \eqref{eq:step4-end} to all $Q_0\in\mathbb{D}(\partial\Omega)$ and
Remark \ref{obsrhpq0} yields that $\omega_{\widetilde{L}}\in RH_q(\partial\Omega)$, where $\widetilde{L}=L^j$ and the implicit constants are independent of $j\in\mathbb{N}$.
\end{proof}

\medskip

\subsection{Step 5}
In the previous step we have proved that $\omega_{\widetilde{L}}\in RH_q(\partial\Omega)$ where $\widetilde{L}=L^j$ and the implicit constants are all uniform in $j$. To complete the proof of Theorem \hyperref[perturbationtheo]{\ref*{perturbationtheo}(a)} we show that $\omega_L\in RH_q(\partial\Omega)$ using the following result:

\begin{proposition}\label{prop:limiting}
Let $\Omega\subset\re^{n+1}$, $n\ge 2$, be a 1-sided $\mathrm{CAD}$. Let $L$ and $L_0$ be real symmetric elliptic operators with matrices $A$ and $A_0$ respectively. For every $j\in\mathbb{N}$, let $L^ju= -\div(A^j \nabla u)$, with $A^j(Y)=A(Y)$ if $\delta(Y)\geq 2^{-j}$ and $A^j(Y)=A_0(Y)$  if $\delta(Y)< 2^{-j}$. Assume that there exists $1<q<\infty$ such that  $\omega_{L^j}=\omega_{L^j,\Omega}\in RH_q(\partial\Omega)$ uniformly in $j$, for every $j\ge j_0$. That is, $\omega_{L^j,\Omega}\ll\sigma$ and
there exists $C$ such that
\begin{equation}\label{rhplj}
\int_{\Delta}k_{L^j,\Omega}^{X_\Delta}(y)^q\,d\sigma(y)\leq C\sigma(\Delta)^{1-q},\qquad k_{L^j,\Omega}^{X_\Delta}:=d\omega_{L^j,\Omega}^{X_\Delta}/d\sigma,
\end{equation}
for every $j\ge j_0$ and every $\Delta(x,r)$ with $x\in\partial\Omega$ and $0<r<\diam(\partial\Omega)$.
Then  $\omega_{L,\Omega}\in RH_q(\partial\Omega)$.
\end{proposition}

\begin{proof}
Fix $B_0=B(x_0,r_0)$ with $x_0\in\partial\Omega$ and $0<r_0<\diam(\partial\Omega)/25$, set $\Delta_0=B_0\cap\partial\Omega$, and consider the subdomain $\Omega_\star:=T_{20\Delta_0}$. Using \cite[Lemma 3.61]{hofmartell} we know that $\Omega_\star$ is a bounded 1-sided $\mathrm{CAD}$, with constants depending only on those of $\Omega$. Applying Lemma \hyperref[proppde]{\ref*{proppde}(d)} it follows that $\omega_{L^j, \Omega_\star}\ll\sigma$ in $4\Delta_0$ and also
$$
k_{L^j,\Omega}^{X_{4\Delta_0}}(y)\approx k_{L^j,\Omega_\star}^{X_{4\Delta_0}}(y),\quad\text{ for }\sigma\text{-a.e. }y\in 4\Delta_0.
$$
Recalling \eqref{definicionkappa0} we know that $25 B_0\cap\Omega\subset \Omega_\star$. In particular,
$10B_0\cap\partial\Omega= 10 B_0\cap\partial\Omega_\star$ and $\sigma_\star:=H^n\rest{\partial\Omega_\star}$ coincides with $\sigma$ in $4\Delta_0$. Therefore, \eqref{rhplj} gives
\begin{equation}\label{ecimpor}
\int_{4\Delta_0}k_{L^j,\Omega_\star}^{X_{4\Delta_0}}(y)^q\,d\sigma_\star(y)\approx
\int_{4\Delta_0}k_{L^j,\Omega}^{X_{4\Delta_0}}(y)^q\,d\sigma(y)\lesssim\sigma(\Delta_0)^{1-q}
\end{equation}
uniformly in $j\in\mathbb{N}$. Note also that $\delta_{\star}(X_{4\Delta_0})=\delta(X_{4\Delta_0})$, where $\delta_{\star}(Y)=\dist(Y,\partial\Omega_{\star})$:
$$
\delta_\star(X_{4\Delta_0})
=
\dist(X_{4\Delta_0},10 B_0\cap \partial\Omega_\star)
=
\dist(X_{4\Delta_0},10 B_0\cap \partial\Omega)=\delta(X_{4\Delta_0}).
$$

Define, for every $g\in C_c (\partial\Omega_\star)$
$$
\Phi(g)
:=
\int_{\partial\Omega\star}g(y)\,d\omega_{L,\Omega_\star}^{X_{4\Delta_0}}(y).
$$
Let $g\in \mathrm{Lip}_c(\partial\Omega)$ be such that $\supp(g)\subset 4\Delta_0$ and extend $g$ by zero to $\partial\Omega_\star\setminus 4\Delta_0$ (by a slight abuse of notation we will call the extension $g$) so that $g\in \mathrm{Lip}_c(\partial\Omega_\star)$ and define
$$
u(X)=\int_{\partial\Omega_\star}g(y)\,d\omega_{L,\Omega_\star}^X(y),\qquad u_j(X)=\int_{\partial\Omega_\star}g(y)\,d\omega_{L^j,\Omega_\star}^X(y),\qquad X\in\Omega_\star.
$$
Since $g\in \mathrm{Lip}_c(\partial\Omega_\star)\subset H^{1/2}(\partial\Omega_\star)\cap C_c(\partial\Omega_\star)$, using Lemma \ref{proprepresentacotado} with $\Omega_\star$ and slightly moving $X_{4\Delta_0}$ if needed, we can write
$$
u(X_{4\Delta_0})-u_j(X_{4\Delta_0})=\iint_{\Omega_\star}(A^j-A)(Y)\nabla_Y G_{L,\Omega_\star}(X_{4\Delta_0},Y)\cdot\nabla u_j(Y)\,dY.
$$
Set $\Sigma_j:=\{Y\in\Omega:\,\delta(Y)<2^{-j}\}$, $\widehat{B}_0:=B(X_{4\Delta_0},\delta(X_{4\Delta_0})/2)$ take $j_1\geq j_0$ large enough so that $\widehat{B}_0\cap \Sigma_{j_1}=\emptyset$. For every $j\ge j_1$, it is clear that $|A^j-A|\lesssim\mathbf{1}_{\Sigma_j}$, with constants depending only on the ellipticity of $L_0$ and $L$. Also we have the a priori estimate $\|\nabla u_j\|_{L^2(\Omega_\star)}\lesssim\|g\|_{H^{1/2}(\partial\Omega_\star)}$ (see \cite{hofmartellgeneral}), where the implicit constant  depends on dimension, the $\mathrm{AR}$ constant, the ellipticity of $L_0$ and $L$, and also of $\diam(\partial\Omega_\star)\approx r_0$). All these and H\"older's inequality yield
\begin{align}\label{difuju}
|u(X_{4\Delta_0})-u_j(X_{4\Delta_0})|
&\lesssim
\iint_{\Omega_\star\cap\Sigma_j}|\nabla_Y G_{L,\Omega_\star}(X_{4\Delta_0},Y)||\nabla u_j(Y)|\,dY
\\
&\lesssim
\|\nabla G_{L,\Omega_\star}(X_{4\Delta_0},\cdot)\mathbf{1}_{\Sigma_j}\|_{L^2(\Omega_\star\setminus \widehat{B}_0)}\|g\|_{H^{1/2}(\partial\Omega_\star)}.\nonumber
\end{align}
Since $\Omega_\star$ is bounded, our Green function coincides with the one defined in \cite{gruterwidman}, hence $\nabla G_{L,\Omega_\star}(X_{4\Delta_0},\cdot)\in L^2(\Omega_\star\setminus \widehat{B}_0)$ (see \eqref{obsgreen1}).
Using the dominated convergence theorem, the first factor of the right hand side of \eqref{difuju} tends to zero, hence $u_j(X_{4\Delta_0})\rightarrow u(X_{4\Delta_0})$.
Recalling then \eqref{ecimpor} we have that
$$
|u(X_{4\Delta_0})|
=\lim_{j\to\infty} |u_j(X_{4\Delta_0})|
\leq
\|g\|_{L^{q'}(4\Delta_0)}\sup_{j\in\mathbb{N}}\|k_{L^j,\Omega_\star}^{X_{4\Delta_0}}\|_{L^q(4\Delta_0)}
\\
\lesssim\|g\|_{L^{q'}(4\Delta_0)}\sigma(\Delta_0)^{-1/q'}.
$$
and hence
\begin{equation}\label{eq:phi-g-cont}
|\Phi(g)|\lesssim
\|g\|_{L^{q'}(4\Delta_0)}\sigma(\Delta_0)^{-1/q'},
\qquad
 g\in \mathrm{Lip}_c(\partial\Omega),\quad \supp(g)\subset 4\Delta_0.
\end{equation}

Suppose now that $g\in L^{q'}(2\Delta_0)$ is such that $\supp(g)\subset 2\Delta_0$, and for $0<t< r_0$ set $g_t=P_t g$ with $P_t$ as in Lemma \ref{densidad}. Since $g_t\in\mathrm{Lip}(\partial\Omega)$ satisfies $\supp(g_t)\subset 4\Delta_0$, we have by \eqref{eq:phi-g-cont}
\begin{align*}
|\Phi(g_t)-\Phi(g_s)|
&=|\Phi(g_t-g_s)|\lesssim\|g_t-g_s\|_{L^{q'}(4\Delta_0)}\sigma(\Delta_0)^{-1/q'}
\\
&\lesssim
\sigma(\Delta_0)^{-1/q'}\big(\|P_tg-g\|_{L^{q'}(\partial\Omega)}+\|P_sg-g\|_{L^{q'}(\partial\Omega)}\big)
\end{align*}
for $0<t,s<r_0$. Hence $\{\Phi(g_t)\}_{t>0}$ is a Cauchy sequence, and we can define $\widetilde{\Phi}(g):=\lim_{t\to 0}\Phi(g_t)$. Clearly, $\widetilde{\Phi}$ is a well-defined linear operator and $\widetilde{\Phi}\in L^{q'}(2\Delta_0)^*$:
\begin{equation}\label{eq:est-phitilde}
|\widetilde{\Phi}(g)|\leq\sup_{0<t<r_0}|\Phi(g_t)|\lesssim\sigma(\Delta_0)^{-1/q'}\sup_{0<t<r_0}\|P_tg\|_{L^{q'}(4\Delta_0)}\lesssim\sigma(\Delta_0)^{-1/q'}\|g\|_{L^{q'}(2\Delta_0)},
\end{equation}
where we have used \eqref{eq:phi-g-cont} and Lemma \ref{densidad}. Consequently, there exists $h\in L^q(2\Delta_0)$ with $\|h\|_{L^q(2\Delta_0)}\lesssim\sigma(\Delta_0)^{-1/q'}$ in such a way that $\widetilde{\Phi}(g)=\int_{2\Delta_0}g(y)h(y)\,d\sigma(y)$ for every $g\in L^{q'}(2\Delta_0)$ such that $\supp(g)\subset 2\Delta_0$.

Let $g\in C_c (\partial\Omega)$ with $\supp(g)\subset 2\Delta_0$  and we extend $g$ by zero to $\partial\Omega_\star$ so that that $g\in C_c (\partial\Omega_\star)$. From Lemma \ref{densidad} applied to $\Omega_\star$,
$\|P_t g\|_{L^\infty(\partial\Omega_\star)}\leq\|g\|_{L^{\infty}(2\Delta_0)}$ and $P_tg(x)\rightarrow g(x)$ as $t\to 0^+$ for every $x\in\partial\Omega_\star$. These, the definition of $\widetilde{\Phi}(g)$ and the dominated convergence theorem with respect to $\omega_{L,\Omega_\star}^{X_{4\Delta_0}}$, shows
\begin{equation}\label{eq:phi-phitilde}
\widetilde{\Phi}(g)
=
\lim_{t\to 0^+} \Phi(P_t g)
=
\lim_{t\to 0^+}\int_{\partial\Omega_\star}P_tg(y)\,d\omega_{L,\Omega_\star}^{X_{4\Delta_0}}(y)=\int_{\partial\Omega_\star}g(y)\,d\omega_{L,\Omega_\star}^{X_{4\Delta_0}}(y)=\Phi(g),
\end{equation}
hence $\widetilde{\Phi}(g)=\Phi(g)$ for every $g\in C_c (\partial\Omega)$ with $\supp(g)\subset 2\Delta_0$.

Next, we see that $\widehat{\omega}:=\omega_{L,\Omega_\star}^{X_{4\Delta_0}} \ll\sigma_\star=\sigma$ in $\tfrac{5}{4}\Delta_0$. Let $E\subset\tfrac{5}{4}\Delta_0$ and let $\varepsilon>0$. Since $\widehat{\omega}$ and $\sigma$ are both regular measures, there exist $K\subset E\subset U \subset\tfrac{3}{2}\Delta_0$ with $K$ compact and $U$ open such that $\widehat{\omega}(U\setminus K)+\sigma(U\setminus K)<\varepsilon$. Using Urysohn's lemma we construct $g\in C_c(\partial\Omega)$ such that $\mathbf{1}_K\leq g\leq\mathbf{1}_U$ and $\supp(g)\subset 2\Delta_0$. Thus, by  \eqref{eq:phi-phitilde} and \eqref{eq:est-phitilde},
\begin{multline*}
\widehat{\omega}(E)
\le
\varepsilon+\widehat{\omega}(K)
\leq
\varepsilon
+\int_{\partial\Omega_\star}g(y)\,d\widehat{\omega}(y)
=
\varepsilon+\Phi(g)
=
\varepsilon+\widetilde{\Phi}(g)
\\
\leq
\varepsilon+\|g\|_{L^{q'}(2\Delta_0)}\|h\|_{L^q(2\Delta_0)}\lesssim \varepsilon+(\varepsilon+\sigma(E))^{1/q'}\sigma(\Delta_0)^{-1/q'}.
\end{multline*}
Letting $\varepsilon\to 0^+$ we conclude that $\widehat{\omega}(E)\lesssim \sigma(E)^{1/q'}\sigma(\Delta_0)^{-1/q'}$ and in particular $\widehat{\omega}\ll\sigma$  in $\tfrac{5}{4}\Delta_0$. Writing then $\widehat{k}=d\widehat{\omega}/d\sigma\in L^1(\tfrac{5}{4}\Delta_0)$ we have that
\begin{equation}\label{eq:dinal-density}
\int_{\frac54 \Delta_0}g(y)h(y)\,d\sigma(y)
=
\widetilde{\Phi}(g)
=
\Phi(g)
=
\int_{\partial\Omega_\star}g(y)\,d\widehat{\omega}(y)
=
\int_{\frac54 \Delta_0}g(y)\,\widehat{k}(y)\, d\sigma(y),
\end{equation}
for every $g\in C_c(\partial\Omega)$ with $\supp(g)\subset \frac54\,\Delta_0$.
Since $(h-\widehat{k})\mathbf{1}_{\frac54 \Delta_0}\in L^1(\partial\Omega)$ by Lemma \ref{densidad} it follows that $P_t((h-\widehat{k})\mathbf{1}_{\frac54 \Delta_0})\longrightarrow (h-\widehat{k})\mathbf{1}_{\frac54 \Delta_0}$ in $L^1(\partial\Omega)$ as $t\to 0^+$. Moreover, for any $x\in \Delta_0$, if we let $0<t<r_0/8$ so that $\supp(\varphi_t(x,\cdot))\subset \frac54\,\Delta_0$, then \eqref{eq:dinal-density} applied to $g=\varphi_t(x,\cdot)$
yields that $P_t((h-\widehat{k})\mathbf{1}_{\frac54 \Delta_0})(x)=0$. All these allow to conclude that $\widehat{k}=h$ $\sigma$-a.e.~in $\Delta_0$, hence $\|\widehat{k}\|_{L^q(\Delta_0)}\leq\|h\|_{L^q(2\Delta_0)}\lesssim\sigma(\Delta_0)^{-1/q'}$.

Note that we showed before that $\widehat{\omega}:=\omega_{L,\Omega_\star}^{X_{4\Delta_0}} \ll\sigma$ in $\Delta_0$, Lemma \hyperref[proppde]{\ref*{proppde}(d)} and Harnack's inequality give $\omega_{L,\Omega}^{X_{\Delta_0}}\ll\sigma$ in $\Delta_0$, and
$$
\int_{\Delta_0}k_{L,\Omega}^{X_{\Delta_0}}(y)^q\,d\sigma(y)\approx\int_{\Delta_0}k_{L,\Omega_\star}^{X_{\Delta_0}}(y)^q\,d\sigma(y)\approx\int_{\Delta_0}\widehat{k}(y)^q\,d\sigma(y)\lesssim\sigma(\Delta_0)^{1-q},
$$
Since $\Delta_0=\Delta(x_0,r_0)$ with $x_0\in\partial\Omega$ and $0<r_0<\diam(\partial\Omega)/25$ was arbitrary, we have proved that $\omega_L\ll\sigma$ and
\begin{equation}\label{rhqpeq}
\int_{\Delta}k_{L,\Omega}^{X_{\Delta}}(y)^q\,d\sigma(y)\leq C\sigma(\Delta)^{1-q},\qquad\Delta=\Delta(x,r),\qquad 0<r<\frac{\diam(\partial\Omega)}{25},
\end{equation}
for $C>1$ depending only on dimension, $p$, the 1-sided $\mathrm{CAD}$ constants, the ellipticity of $L_0$ and $L$, and the constant in $\omega_{L_0}\in RH_p(\partial\Omega)$. By a standard covering argument and Harnack's inequality, \eqref{rhqpeq} extends to all $0<r<\diam(\partial\Omega)$. Using Lemma \ref{solvability}, we have shown that $\omega_L=\omega_{L,\Omega}\in RH_q(\partial\Omega)$ completing the proof of Proposition \ref{prop:limiting}.
\end{proof}

\section{Proof of Theorem \texorpdfstring{\hyperref[perturbationtheo]{\ref*{perturbationtheo}\textnormal{(b)}}}{1.1(b)}}\label{section:main-b} 

%Recall that by Lemma \ref{solvability}, the condition $\omega_{L_0}\in RH_p(\partial\Omega)$ is equivalent to the fact that $\omega_{L_0}\ll\sigma$ and that there exists a uniform constant $C$ such that for every $\Delta=\Delta(x,r)$ with $x\in\partial\Omega$ and $0<r<\diam(\partial\Omega)$, there holds
%\begin{equation}\label{reverseholderL0}
%\int_{\Delta}k_{L_0}^{X_{\Delta}}(y)^p\,d\sigma(y)
%\leq C\sigma(\Delta)^{1-p},
%\qquad k_{L_0}^X=d\omega_{L_0}^X/d\sigma,
%\end{equation}
%or equivalently $\|k_{L_0}^{X_{\Delta}}\|_{L^p(\Delta)}\lesssim\sigma(\Delta)^{-1/p'}$.

We first note that by Theorem \hyperref[perturbationtheo]{\ref*{perturbationtheo}(a)}, the fact that $\vertiii{\varrho(A,A_0)}\leq \varepsilon_0$ gives that $\omega_L\in RH_q(\partial\Omega)$ for some $1<q<\infty$, and in particular $\omega_{L}\ll\sigma$. The goal of Theorem \hyperref[perturbationtheo]{\ref*{perturbationtheo}(b)} is to see that if $\varepsilon_0>0$ is taken sufficiently small, then we indeed have that $\omega_L\in RH_p(\partial\Omega)$, that is, $L_0$ and $L$ are in the same reverse Hölder class. To this aim, we split the proof in several steps.

\medskip

We choose $M_0>400\,\kappa_0/c$ (which will remain fixed during the proof) where $c$ is the corkscrew constant and $\kappa_0$ as in  \eqref{definicionkappa12}. Given an arbitrary ball $B_0=B(x_0,r_0)$ with $x_0\in \partial\Omega$ and $0<r_0<\diam(\partial\Omega)/M_0$, let $\Delta_0=B_0\cap\partial\Omega$ and take $X_{M_0\Delta_{Q_0}}$ the corkscrew point relative to $M_0\Delta_{Q_0}$ (note that $M_0 r_0<\diam(\partial\Omega)$). If  $Q_0\in\mathbb{D}^{\Delta_0}$ then $\ell(Q_0)< 400\,r_0<\diam(\partial\Omega)/\kappa_0$. Also $\delta(X_{M_0\Delta_0})\geq c M_0r_0>2\kappa_0r_0$, and by by \eqref{definicionkappa12},
\begin{equation}\label{eq:X0-TDelta0}
X_{M_0\Delta_0}\in\Omega\setminus 2\kappa_0 B_0\subset \Omega\setminus T_{\Delta_0}^{**}.
\end{equation}

\subsection{Step 0}
As done in Step 0 of the proof of Theorem \hyperref[perturbationtheo]{\ref*{perturbationtheo}(a)}, we let work with $\widetilde{L}=L^j$, associated with the matrix $\widetilde{A}=A^j$ defined in \eqref{definicionLtilde}.
As there we have that $\omega_{\widetilde{L}}\ll\sigma$, hence we let $k_{\widetilde{L}}^X:=d\omega_{\widetilde{L}}^X/d\sigma$. This qualitative property will be essential in the first two steps.
At the end of Step 2 we will have obtained the desired conclusion for the operator $\widetilde{L}=L^j$, with constants independent of $j\in\mathbb{N}$, and in Step 3 we will transfer it to $L$ via a limiting argument. From now on, $j\in\mathbb{N}$ will be fixed and we will focus on the operator $\widetilde{L}=L^j$.
\medskip

\subsection{Step 1}
We start by fixing $B_0=B(x_0,r_0)$ with $x_0\in\partial\Omega$, $0<r_0<\diam(\partial\Omega)/M_0$
and $M_0$ as chosen above.  Set $\Delta_0=B_0\cap\partial\Omega$ and $X_{0}:=X_{M_0\Delta_{Q_0}}$ so that \eqref{eq:X0-TDelta0} holds.
We define the operator $L_1u=L_1^{\Delta_0}u=-\div(A_1\nabla u)$ where
$$
A_1(Y):=\left\{
             \begin{array}{ll}
               \widetilde{A}(Y) & \hbox{$\text{if }Y\in T_{\Delta_0},$} \\
               A_0(Y) & \hbox{$\text{if }Y\in\Omega\setminus T_{\Delta_0},$}
             \end{array}
           \right.
$$
and $\widetilde{A}=A^j$ as in \eqref{definicionLtilde}. By construction, it is clear that $\mathcal{E}_1:=A_1-A_0$ verifies $|\mathcal{E}_1|\leq|\mathcal{E}|\mathbf{1}_{T_{\Delta_0}}$, and also $\mathcal{E}_1(Y)=0$ if $\delta(Y)<2^{-j}$. Hence, the support of $A_1-A_0$ is contained in a compact subset of $\Omega$.

In order to simplify the notation, we set $\widehat{\Delta}_0:=\tfrac{1}{2}\Delta_0^*=\Delta(x_0,\kappa_0r_0)$ and let $0\leq g\in L^{p'}(\widehat{\Delta}_0)$ be such that $\|g\|_{L^{p'}(\widehat{\Delta}_0)}=1$. Without loss of generality, we may assume that $g$ is defined in $\partial\Omega$ with $g\equiv 0$ in $\Omega\setminus\widehat{\Delta}_0$. For $0<t<\kappa_0 r_0/2$, we consider $g_t=P_tg\geq 0$ with $P_tg$ defined as in \eqref{defgt}, together with the solutions
$$
u_0^t(X)=\int_{\partial\Omega}g_t(y)\,d\omega_{L_0}^X(y),\qquad u_1^t(X)=\int_{\partial\Omega} g_t(y)\,d\omega_{L_1}^X(y),\qquad X\in\Omega.
$$
By Lemma \ref{densidad},  $g_t\in\mathrm{Lip}(\partial\Omega)$ verifies $\supp(g_t)\subset \Delta_0^*$ and hence $g_t\in\mathrm{Lip}_c(\partial\Omega)\subset H^{1/2}(\partial\Omega)\cap C_c(\partial\Omega)$.
Since $\mathcal{E}_1=A_1-A_0$ verifies $|\mathcal{E}_1|\leq|\mathcal{E}|\mathbf{1}_{T_{\Delta_0}}$ and also $\mathcal{E}_1(Y)=0$ if $\delta(Y)<2^{-j}$, \eqref{eq:X0-TDelta0} and \eqref{definicionkappa12} allow us to invoke  Lemma \ref{proprepresent} (see Remark \ref{modific}) to obtain
\begin{align*}
&F^t(X_0):
=|u_1^t(X_0)-u_0^t(X_0)|
=
\bigg|\iint_{\Omega}(A_0-A_1)(Y)\nabla_Y G_{L_1}(X_0,Y)\cdot\nabla u_0^t(Y)\,dY\bigg|
\\
&\ \ \leq
\sum_{Q_0\in\mathbb{D}^{\Delta_0}}
\sum_{Q\in\mathbb{D}_{Q_0}}\sum_{I\in\mathcal{W}_{Q}^*}\iint_{I^*}|\mathcal{E}(Y)||\nabla_Y G_{L_1}(X_0,Y)||\nabla u_0^t(Y)|\,dY
\\
&\ \ \le
\sum_{Q_0\in\mathbb{D}^{\Delta_0}}
\sum_{Q\in\mathbb{D}_{Q_0}}\sum_{I\in\mathcal{W}_{Q}^*}
\sup_{I^*}|\mathcal{E}|\bigg(\iint_{I^*}|\nabla_Y G_{L_1}(X_0,Y)|^2\,dY\bigg)^{1/2}\bigg(\iint_{I^*}|\nabla u_0^t(Y)|^2\,dY\bigg)^{1/2}.
\end{align*}
Note that for every $Q_0\in\mathbb{D}^{\Delta_0}$ and our choice of $M_0$, we have that $\ell(Q_0)<\diam(\partial\Omega)/\kappa_0$. Thus by Lemma \ref{carlesondisc} the estimate $\vertiii{a} \leq \varepsilon_0$ implies that
 $\mathfrak{m}=\{\gamma_{Q}\}_{Q\in\mathbb{D}(\partial\Omega)}\in\mathcal{C}(Q_0)$ (see \eqref{coefcarleson}) and $\|\mathfrak{m}\|_{\mathcal{C}(Q_0)}\leq\kappa \varepsilon_0$, where $\kappa>0$ depends only on dimension and on the 1-sided $\mathrm{CAD}$ constants. At this point we just need to repeat the arguments in \eqref{hastaqui}--\eqref{eq:conc1} in every $Q_0\in\mathbb{D}^{\Delta_0}$ with $\mathcal{F}=\emptyset$ and hence $\mathbb{D}_{\mathcal{F},Q_0}=\mathbb{D}_{Q_0}$. This ultimately gives
$$
F^t(X_0)\lesssim
\varepsilon_0^{1/2}
\sum_{Q_0\in\mathbb{D}^{\Delta_0}} \|k_{L_1}^{X_0}\|_{L^p(Q_0)}
\lesssim\varepsilon_0^{1/2}\|k_{L_1}^{X_0}\|_{L^p(\widehat{\Delta}_0)},
$$
where the last inequality is justified by the bounded cardinality of $\mathbb{D}^{\Delta_0}$. Therefore,
$$
0\leq u_1^t(X_0)\leq F^t(X_0)+u_0^t(X_0)\lesssim\varepsilon_0^{1/2}\|k_{L_1}^{X_0}\|_{L^p(\widehat{\Delta}_0)}+\|k_{L_0}^{X_0}\|_{L^{p}(\Delta_0^*)},
$$
where we have used that $\|g_t\|_{L^{p'}(\partial\Omega)}\lesssim 1$ and that $\text{supp}(g_t)\subset \Delta_{0}^*$ by Lemma \ref{densidad} and where the implicit constants do not depend on $t$. Next, we write
\begin{multline*}
\int_{\partial\Omega}g(y)k_{L_1}^{X_0}(y)\,d\sigma(y)
=u_1^t(X_0)+\int_{\partial\Omega}(g(y)-g_t(y)) k_{L_1}^{X_0}(y)\,d\sigma(y)
\\
\lesssim
\varepsilon_0^{1/2}\|k_{L_1}^{X_0}\|_{L^p(\widehat{\Delta}_0)}+\|k_{L_0}^{X_0}\|_{L^{p}(\Delta^*_{0})}+\|g-g_t\|_{L^{p'}(\partial\Omega)}\|k_{L_1}^{X_0}\|_{L^p(\Delta_{0}^*)}.
\end{multline*}
Notice that $g_t\rightarrow g$ in $L^{p'}(\partial\Omega)$ by Lemma \ref{densidad}, which along with the fact that $\|k_{L_1}^{X_0}\|_{L^p(\Delta_{0}^*)}\leq C_j<+\infty$, by Lemma \hyperref[proppde]{\ref*{proppde}(e)} and Harnack's inequality, implies
$$
\int_{\partial\Omega}g(y)k_{L_1}^{X_0}(y)\,d\sigma(y)\lesssim\varepsilon_0^{1/2}\|k_{L_1}^{X_0}\|_{L^p(\widehat{\Delta}_0)}+\|k_{L_0}^{X_0}\|_{L^{p}(\Delta^*_{0})}.
$$
Taking the supremum over all $0\leq g\in L^{p'}(\widehat{\Delta}_0)$ with $\|g\|_{L^{p'}(\widehat{\Delta}_0)}=1$ we obtain
$$
\|k_{L_1}^{X_0}\|_{L^p(\widehat{\Delta}_0)}\leq C\varepsilon_0^{1/2}\|k_{L_1}^{X_0}\|_{L^p(\widehat{\Delta}_0)}+C\|k_{L_0}^{X_0}\|_{L^p(\Delta^*_{0})},
$$
where $C$ depends on the allowable parameters. Since $\|k_{L_1}^{X_0}\|_{L^p(\widehat{\Delta}_0)}\leq C_j<\infty$, taking $\varepsilon_0<C^{-2}/4$, we can hide the first term in the left hand side to obtain $\|k_{L_1}^{X_0}\|_{L^p(\widehat{\Delta}_0)}\lesssim\|k_{L_0}^{X_0}\|_{L^p(\Delta^*_{0})}$. Using then that $\omega_{L_0}\in RH_p(\partial\Omega)$ and Harnack's inequality to change the pole from $X_0=X_{M_0\Delta_0}$ to $X_{\Delta_{0}^*}$ (with constants depending on $M_0$, which is already fixed), we conclude that
\begin{multline}\label{rhpL1vt}
\int_{\Delta_0}k_{L_1}^{X_{\Delta_0}}(y)^p\,d\sigma(y)
\lesssim
\int_{\widehat{\Delta}_0}k_{L_1}^{X_0}(y)^p\,d\sigma(y)\lesssim\int_{\Delta_0^*}k_{L_0}^{X_0}(y)^p\,d\sigma(y)
\\
\approx
\int_{\Delta_0^*}k_{L_0}^{X_{\Delta^*_0}}(y)^p\,d\sigma(y)\lesssim\sigma(\Delta^*_0)^{1-p}\approx\sigma(\Delta_0)^{1-p}.
\end{multline}
\medskip

\subsection{Step 2}
Let $L_2:=-\div(A_2\nabla u)$ where
$$
A_2(Y):=\left\{
             \begin{array}{ll}
               A_1(Y) & \hbox{$\text{if }Y\in T_{\Delta_0}$,} \\
               \widetilde{A}(Y) & \hbox{$\text{if }Y\in \Omega\setminus T_{\Delta_0}$,}
             \end{array}
           \right.
$$
and hence $A_2=\widetilde{A}$ in $\Omega$. As seen in Step 0, since $\widetilde{L}\equiv L_0$ in $\{Y\in\Omega:\,\delta(Y)< 2^{-j}\}$, we have that $\omega_{L_2}=\omega_{\widetilde{L}}\ll\sigma$, and there exists $k_{L_2}=d\omega_{L_2}/d\sigma$. Set $B'_0:=B(x_0,r_0/(2\kappa_0))$ and $\Delta_0'=B_0'\cap\partial\Omega$. By \eqref{definicionkappa0}, $2\kappa_0 B'_0\cap\Omega\subset\tfrac{5}{4} B_0\cap\Omega\subset T_{\Delta_0}$ and since $L_2\equiv L_1$ in $T_{\Delta_0}$, Lemma \hyperref[proppde]{\ref*{proppde}(e)} implies
$$
k_{\widetilde{L}}^{X_{\Delta'_0}}(y)=k_{L_2}^{X_{\Delta'_0}}(y)\approx k_{L_1}^{X_{\Delta'_0}}(y),\quad\text{ for }\sigma\text{-a.e. }y\in\Delta'_0.
$$
Consequently, using \eqref{rhpL1vt} and Harnack's inequality (with constants depending on $M_0$, which is already fixed), we obtain
$$
\int_{\Delta'_0}k_{\widetilde{L}}^{X_{\Delta'_0}}(y)^p\,d\sigma(y)
\approx
\int_{\Delta'_0}k_{L_1}^{X_{\Delta'_0}}(y)^p\,d\sigma(y)
\lesssim
\int_{\Delta_0}k_{L_1}^{X_{\Delta_0}}(y)^p\,d\sigma(y)\lesssim\sigma(\Delta_0)^{1-p}\approx\sigma(\Delta'_0)^{1-p}.
$$
Since the surface ball $\Delta_0=\Delta(x_0,r_0)$ with $x_0\in\partial\Omega$ and $r_0<\diam(\partial\Omega)/M_0$ was arbitrary, we have proved that
\begin{equation}\label{cotaL}
\int_{\Delta}k_{\widetilde{L}}^{X_{\Delta}}(y)^p\,d\sigma(y)\lesssim \sigma(\Delta)^{1-p},\qquad\Delta=\Delta(x,r),\quad 0<r<\frac{\diam(\partial\Omega)}{2M_0\kappa_0}.
\end{equation}
By a standard covering argument and Harnack's inequality, \eqref{cotaL} extends to all $\Delta=\Delta(x,r)$ with $0<r<\diam(\partial\Omega)$.
This and Lemma \ref{solvability} show that $\omega_{\widetilde{L}}\in RH_p(\partial\Omega)$ where we recall that $\widetilde{L}=L^j$ is the operator defined in \eqref{definicionLtilde}, $j\in\mathbb{N}$ is arbitrary, and the implicit constant is independent of $j\in\mathbb{N}$.
\medskip

\subsection{Step 3}
Using the previous step and Proposition \ref{prop:limiting} with $q=p$ we conclude as desired that $\omega_{L}\in RH_p(\partial\Omega)$ and the proof of Theorem \hyperref[perturbationtheo]{\ref*{perturbationtheo}(b)} is complete.

\begin{remark}\label{observacionutil}
One can easily see from the previous proof that $\vertiii{a}\leq\varepsilon_0$ could be slightly weakened by simply assuming that $\|\mathfrak{m}\|_{\mathcal{C}(Q_0)}$ is small enough, with $\mathfrak{m}=\{\gamma_Q\}_{Q\in\mathbb{D}(\partial\Omega)}$ and $\gamma_Q$ defined in \eqref{coefcarleson}. Further details are left to the interested reader.
\end{remark}
\medskip

\section{Applications of Theorem \texorpdfstring{\hyperref[perturbationtheo]{\ref*{perturbationtheo}\textnormal{(b)}}}{1.1(b)}}
\label{section:apps} 

Given $L_0$, $L$ elliptic operators with matrices $A_0$, $A$ respectively, we say that their disagreement defined in \eqref{discrepancia} verifies a \emph{vanishing trace} Carleson condition if
\begin{equation}\label{vanishingcarleson}
 \lim_{s\rightarrow 0^+}
\Bigg(\sup_{\substack{x\in\partial\Omega \\ 0<r\leq s<\diam(\partial\Omega)}}\frac{1}{\sigma(\Delta(x,r))}\iint_{B(x,r)\cap\Omega}\frac{\varrho(A,A_0)(X)^2}{\delta(X)}\,dX\Bigg)
=0.
\end{equation}
\medskip

\begin{corollary}\label{vanishingsmallscales}
Suppose that $\Omega\subset\re^{n+1}$ is a 1-sided $\mathrm{CAD}$. Let $L_0$, $L$ be elliptic operators whose disagreement in $\Omega$ is given by the function $\varrho(A,A_0)$ defined in \eqref{discrepancia}. If $\omega_{L_0}\in RH_p(\partial\Omega)$ for some $1<p<\infty$ and the vanishing trace Carleson condition \eqref{vanishingcarleson} holds, then $\omega_L\ll\sigma$ and there exist $C_0>0$ (depending only on dimension, $p$, the 1-sided $\mathrm{CAD}$ constants, the ellipticity of $L_0$ and $L$, and the constant in $\omega_{L_0}\in RH_p(\partial\Omega)$), and $0<r_0<\diam(\partial\Omega)$ (depending on the above parameters and the condition \eqref{vanishingcarleson}), such that
\begin{equation}\label{rhpsmall}
\int_{\Delta}k_L^{X_{\Delta}}(y)^p\,d\sigma(y)\leq C_0\,\sigma(\Delta)^{1-p},\quad\Delta=\Delta(x,r),\quad x\in\partial\Omega,\quad 0<r\leq r_0.
\end{equation}
\end{corollary}

\begin{proof}
Take $\varepsilon_0>0$ from Theorem \hyperref[perturbationtheo]{\ref*{perturbationtheo}(b)} and let $M>1$ to be chosen. Thanks to \eqref{vanishingcarleson}, there exists $s_0=s_0(\varepsilon_0,M)<\diam(\partial\Omega)$ such that for every $\Delta=\Delta(x,r)$ with $x\in\partial\Omega$ and $0<r\leq s_0$, we have that
\begin{equation}\label{definicionj}
\frac{1}{\sigma(\Delta(x,r))}\iint_{B(x,r)\cap\Omega}\frac{a(X)^2}{\delta(X)}\,dX\leq\frac{\varepsilon_0}{M},
\end{equation}
where $a:=\varrho(A,A_0)$. 
Given $s>0$, set $\Sigma_s:=\{Y\in\Omega:\,\delta(Y)< s\}$ and consider the operator $\widetilde{L}u=-\div(\widetilde{A}\nabla u)$ with
$$
  \widetilde{A}(Y):=\left\{
             \begin{array}{ll}
               A_0(Y) & \hbox{$\text{if }Y\in\Omega\setminus\Sigma_{s_0/4},$} \\
               A(Y) & \hbox{$\text{if }Y\in\Sigma_{s_0/4}.$}
             \end{array}
           \right.
$$
Note that $\widetilde{A}$ is uniformly elliptic with constant $\widetilde{\Lambda}=\max\{\Lambda_A,\Lambda_{A_0}\}$, where $\Lambda_A$ and $\Lambda_{A_0}$ are the ellipticity constants of $A$ and $A_0$ respectively. Setting $\widetilde{\mathcal{E}}:=\widetilde{A}(Y)-A_0(Y)$  and $\widetilde{a}(X):=\sup_{|X-Y|<\delta(X)/2}|\widetilde{\mathcal{E}}(Y)|$, it is clear that $\widetilde{\mathcal{E}}(Y)=\mathcal{E}(Y)\mathbf{1}_{\Sigma_{s_0/4}}(Y)$. Therefore, since $B(X,\delta(X)/2)\subset\Omega\setminus\Sigma_{s_0/4}$ for each $X\in\Omega\setminus\Sigma_{s_0/2}$, we have that
\begin{equation}\label{atilda}
\widetilde{a}(X)\leq a(X)\mathbf{1}_{\Sigma_{s_0/2}}(X),\qquad X\in\Omega.
\end{equation}
Now, we claim that
\begin{equation}\label{claimatilda}
\vertiii{\widetilde{a}}=\sup_{\substack{x\in\partial\Omega \\ 0<r\leq s<\diam(\partial\Omega)}}\frac{1}{\sigma(\Delta(x,r))}\iint_{B(x,r)\cap\Omega}\frac{\widetilde{a}(X)^2}{\delta(X)}\,dX\leq\varepsilon_0,
\end{equation}
provided $M$ is chosen large enough depending only on dimension and the AR constant. To prove the claim we take $B=B(x,r)$ with $x\in\partial\Omega$ and $0<r<\diam(\partial\Omega)$. Suppose first that $0<r\leq s_0$, using \eqref{definicionj} and \eqref{atilda}, we obtain
$$
\frac{1}{\sigma(\Delta(x,r))}\iint_{B(x,r)\cap\Omega}\frac{\widetilde{a}(X)^2}{\delta(X)}\,dX
\leq
\frac{1}{\sigma(\Delta(x,r))}\iint_{B(x,r)\cap\Omega}\frac{a(X)^2}{\delta(X)}\,dX\leq\frac{\varepsilon_0}{M}\leq\varepsilon_0.
$$
On the other hand, if $r>s_0$, using \eqref{atilda} we have that
$$
\iint_{B(x,r)\cap\Omega}\frac{\widetilde{a}(X)^2}{\delta(X)}\,dX
\leq
\iint_{B(x,r)\cap\Sigma_{s_0/2}}\frac{a(X)^2}{\delta(X)}\,dX.
$$
By a standard Vitali type covering argument, there exists a family $\{\Delta_j\}_j$ of disjoint surface balls $\Delta_j=\Delta(x_j,s_0/2)$ with $x_j\in\Delta(x,2r)$, satisfying $\Delta(x,2r)\subset\bigcup_j 3\Delta_j$ and $\Delta_j\subset\Delta(x,3r)$. Note that by construction, $B(x,r)\cap\Sigma_{s_0/2}\subset\bigcup_j B(x_j,s_0)$, hence by \eqref{definicionj}, we have that
\begin{multline*}
\iint_{B(x,r)\cap\Sigma_{s_0/2}}\frac{a(X)^2}{\delta(X)}\,dX\leq\sum_j\iint_{B(x_j,s_0)\cap\Omega}\frac{a(X)^2}{\delta(X)}\,dX\leq\frac{\varepsilon_0}{M}\sum_j\sigma(\Delta(x_j,s_0))
\\
\approx
\frac{\varepsilon_0}{M}\sum_j\sigma(\Delta_j)\leq\frac{\varepsilon_0}{M}\sigma(\Delta(x,3r))\approx\frac{\varepsilon_0}{M}\sigma(\Delta(x,r))\leq\varepsilon_0\sigma(\Delta(x,r)),
\end{multline*}
for $M$ sufficiently large, depending only on dimension and on the $\mathrm{AR}$ constant. Gathering the above estimates, we have proved as desired \eqref{claimatilda}.

Next we apply Theorem \hyperref[perturbationtheo]{\ref*{perturbationtheo}(b)} to $L_0$ and $\widetilde{L}$, to conclude that $\omega_{\widetilde{L}}\in RH_p(\partial\Omega)$ and, in particular,
\begin{equation}\label{rhpLtilda}
\int_{\Delta}k_{\widetilde{L}}^{X_{\Delta}}(y)^p\,d\sigma(y)\lesssim \sigma(\Delta),\qquad\Delta=\Delta(x,r),\quad x\in\partial\Omega,\quad 0<r<\diam(\partial\Omega).
\end{equation}
Set $r_0:=s_0/(8\kappa_0)$ and let $\Delta=\Delta(x,r)$ with $x\in\partial\Omega$ and $0<r\leq r_0$. Note that $B(x,2\kappa_0r)\cap\Omega\subset B(x,s_0/4)\cap\Omega\subset\Sigma_{s_0/4}$, hence $\widetilde{L}\equiv L$ in $B(x,2\kappa_0r)\cap\Omega$. Using Lemma \hyperref[proppde]{\ref*{proppde}(e)} we have that $\omega_L\ll\sigma$ in $\Delta$ and
$$
k_L^{X_{\Delta}}(y)\approx k_{\widetilde{L}}^{X_{\Delta}}(y),\quad\text{ for }\sigma\text{-a.e. }y\in \Delta.
$$
This and \eqref{rhpLtilda} proves \eqref{rhpsmall} and the proof is complete.
\end{proof}
\medskip

\begin{corollary}
Suppose that $\Omega\subset\re^{n+1}$ is a \textbf{bounded} 1-sided $\mathrm{CAD}$. Let $L_0$, $L$ be elliptic operators whose disagreement in $\Omega$ is given by the function $a(X)$ defined in \eqref{discrepancia}, and suppose that $\omega_{L_0}\in RH_p(\partial\Omega)$ for some $1<p<\infty$. If the vanishing trace Carleson condition \eqref{vanishingcarleson} holds, then we have that $\omega_L\in RH_p(\partial\Omega)$, with constants depending on $\diam(\partial\Omega)$, dimension, $p$, the condition \eqref{vanishingcarleson}, the 1-sided $\mathrm{CAD}$ constants, the ellipticity of $L_0$ and $L$, and the constant in $\omega_{L_0}\in RH_p(\partial\Omega)$.
\end{corollary}

%\nocite{*}
%\bibliographystyle{alphanum}
%\bibliography{referenceschm}

\begin{thebibliography}{AHM{\etalchar{+}}2I}

\bibitem[AHLT]{MR1879847}
P.~Auscher, S. Hofmann, J.~L.~Lewis, and P.~Tchamitchian.
\newblock Extrapolation of {C}arleson measures and the analyticity of {K}ato's
  square-root operators.
\newblock {\em Acta Math.}, {\bf 187}(2):161--190, 2001.

\bibitem[AHM{\etalchar{+}}1]{MR1934198}
P.~Auscher, S.~Hofmann, C.~Muscalu, T.~Tao, and C.~Thiele.
\newblock Carleson measures, trees, extrapolation, and {$T(b)$} theorems.
\newblock {\em Publ. Mat.}, {\bf 46}(2):257--325, 2002.

\bibitem[AHM{\etalchar{+}}2]{MR3626548}
J.~Azzam, S.~Hofmann, J.~M.~Martell, K.~Nystr\"om, and T.~Toro.
\newblock A new characterization of chord-arc domains.
\newblock {\em J. Eur. Math. Soc. (JEMS)}, {\bf 19}(4):967--981, 2017.

\bibitem[Bou]{bou}
J.~Bourgain.
\newblock On the {H}ausdorff dimension of harmonic measure in higher dimension.
\newblock {\em Invent. Math.}, {\bf 87}(3):477--483, 1987.

\bibitem[Car]{Car} L.~Carleson.  Interpolation by bounded analytic functions and the corona problem.
{\em Ann. of Math. (2)}, {\bf 76}:547--559,  1962.

\bibitem[CG]{CG} L.~Carleson and J.~Garnett. Interpolating sequences and separation properties.
{\em J. Analyse Math.}, {\bf 28}:273--299,  1975.


\bibitem[CHMT]{CHMT}
J.~Cavero, S.~Hofmann and J.M.~Martell, and T.~Toro.
Perturbations of elliptic operators in 1-sided chord-arc domains. Part II: Non-symmetric operators.
{\em Work in progress}, 2017. 


\bibitem[Chr]{christ}
M.~Christ.
\newblock A {$T(b)$} theorem with remarks on analytic capacity and the {C}auchy
  integral.
\newblock {\em Colloq. Math.}, {\bf 60/61}(2):601--628, 1990.


\bibitem[CF]{coifman1974}
R.~R.~Coifman and C.~L.~Fefferman.
\newblock Weighted norm inequalities for maximal functions and singular
  integrals.
\newblock {\em Studia Mathematica}, {\bf 51}(3):241--250, 1974.



\bibitem[CMS]{coifmanmeyerstein}
R.~R.~Coifman, Y.~Meyer, and E.~M.~Stein.
\newblock Some new function spaces and their applications to harmonic analysis.
\newblock {\em J. Funct. Anal.}, {\bf 62}(2):304--335, 1985.

\bibitem[Dah]{MR859772}
B.~E.~J.~Dahlberg.
\newblock On the absolute continuity of elliptic measures.
\newblock {\em Amer. J. Math.}, {\bf 108}(5):1119--1138, 1986.

\bibitem[DJ]{DJe} G.~David and D.~Jerison. Lipschitz approximation to hypersurfaces, harmonic measure, and singular integrals. {\em Indiana Univ. Math. J.}, {\bf  39}(3):831--845, 1990.

\bibitem[DS1]{davidsemmes1}
G.~David and S.~Semmes.
\newblock Singular integrals and rectifiable sets in {${\bf R}^n$}: {B}eyond
  {L}ipschitz graphs.
\newblock {\em Ast\'erisque}, {\bf 193}, 1991.

\bibitem[DS2]{davidsemmes2}
G.~David and S.~Semmes.
\newblock {\em Analysis of and on uniformly rectifiable sets}, volume~38 of
  {\em Mathematical Surveys and Monographs}.
\newblock American Mathematical Society, Providence, RI, 1993.


\bibitem[Esc]{MR1394581}
L.~Escauriaza.The $L^p$ Dirichlet problem for small perturbations of the Laplacian. {\em Israel J. Math.}, {\bf 94}:353--366, 1996.

\bibitem[FKP]{MR1114608}
R.~A.~Fefferman, C.~E.~Kenig, and J.~Pipher.
\newblock The theory of weights and the {D}irichlet problem for elliptic
  equations.
\newblock {\em Ann. of Math. (2)}, {\bf 134}(1):65--124, 1991.

\bibitem[GR]{MR807149}
J.~Garc{\'{\i}}a-Cuerva and J.~L.~{Rubio de Francia}.
\newblock {\em Weighted norm inequalities and related topics}, volume 116 of
  {\em North-Holland Mathematics Studies}.
\newblock North-Holland Publishing Co., Amsterdam, 1985.
\newblock Mathematical Notes, 104.

\bibitem[GW]{gruterwidman}
M.~Gr{\"u}ter and K.-O.~Widman.
\newblock The {G}reen function for uniformly elliptic equations.
\newblock {\em Manuscripta Math.}, {\bf 37}(3):303--342, 1982.

\bibitem[HKM]{HKM}
J.~Heinonen, T.~Kilpel\"ainen, and O.~Martio.
\newblock {\em Nonlinear potential theory of degenerate elliptic equations}.
\newblock Oxford University Press, New York, 1993.

\bibitem[HLMN]{hoflemartellnystrom}
S.~Hofmann, P.~Le, J.~M.~Martell, and K.~Nystr\"om.
\newblock The weak-{$A_\infty$} property of harmonic and $p$-harmonic measures
implies uniform rectifiability.
\newblock {\em Anal. PDE.}, {\bf  10}(3):513--558, 2017.

\bibitem[HL]{MR1828387}
S.~Hofmann and J.~L.~Lewis.
\newblock The {D}irichlet problem for parabolic operators with singular drift
  terms.
\newblock {\em Mem. Amer. Math. Soc.}, {\bf 151}(719):viii+113, 2001.


\bibitem[HM1]{MR2655385}
S.~Hofmann and J.M.~Martell.
\newblock A note on {$A_\infty$} estimates via extrapolation of {C}arleson
  measures.
\newblock In {\em The {AMSI}-{ANU} {W}orkshop on {S}pectral {T}heory and
  {H}armonic {A}nalysis}, volume~44 of {\em Proc. Centre Math. Appl. Austral.
  Nat. Univ.}, pages 143--166. Austral. Nat. Univ., Canberra, 2010.

\bibitem[HM2]{MR2833577}
S.~Hofmann and J.M.~Martell.
\newblock {$A_\infty$} estimates via extrapolation of {C}arleson measures and
  applications to divergence form elliptic operators.
\newblock {\em Trans. Amer. Math. Soc.}, {\bf 364}(1):65--101, 2012.

\bibitem[HM3]{hofmartell}
S.~Hofmann and J.M.~Martell.
\newblock Uniform rectifiability and harmonic measure {I}: Uniform
  rectifiability implies {P}oisson kernels in ${L}^p$.
\newblock {\em Ann. Sci. École Norm. Sup.}, {\bf 47}(3):577--654, 2014.

\bibitem[HMT1]{hofmartellgeneral}
S.~Hofmann and J.M.~Martell, and T.~Toro.
\newblock General divergence form elliptic operators on domains with ADR boundaries,
and on 1-sided NTA domains.
\newblock {\em Work in progress}, 2014.

\bibitem[HMT2]{HMT-var}
S.~Hofmann and J.M.~Martell, and T.~Toro.
\newblock {$A_\infty$} implies {NTA} for a class of variable coefficient
  elliptic operators.
{\em J. Differential Equations}, 2017. \url{http://dx.doi.org/10.1016/j.jde.2017.06.028}	
	


\bibitem[HMU]{hofmartelltuero}
S.~Hofmann and J.M.~Martell, and I.~Uriarte-Tuero.
\newblock Uniform rectifiability and harmonic measure, {II}: {P}oisson kernels
  in {$L^p$} imply uniform rectifiability.
\newblock {\em Duke Math. J.}, {\bf 163}(8):1601--1654, 2014.

\bibitem[HMMM]{HMMM}
S.~Hofmann, D.~Mitrea, M.~Mitrea, and A.~J.~Morris.
\newblock {$L^p$}-square function estimates on spaces of homogeneous type and
  on uniformly rectifiable sets.
\newblock {\em Mem. Amer. Math. Soc.}, {\bf 245}(1159), 2017.

\bibitem[JK]{MR676988}
D.~S.~Jerison and C.~E.~Kenig.
\newblock Boundary behavior of harmonic functions in nontangentially accessible
  domains.
\newblock {\em Adv. in Math.}, {\bf 46}(1):80--147, 1982.

\bibitem[JW]{jonsson1984function}
A.~Jonsson and H.~Wallin.
\newblock {\em Function Spaces on Subsets of  ${\bf R}^n$}.
\newblock Mathematical reports. Harwood Academic Publishers, 1984.


\bibitem[KKiPT]{KKiPT} 
C.~Kenig, B.~Kirchheim, J.~Pipher, and T.~Toro.
Square functions and the $A_\infty$ property of elliptic measures.  
{\em J. Geom. Anal.} {\bf 26}(3):2383--2410,  2016.

\bibitem[KP]{MR1829584}
C.~E.~Kenig and J.~Pipher.
\newblock The {D}irichlet problem for elliptic equations with drift terms.
\newblock {\em Publ. Mat.}, {\bf 45}(1):199--217, 2001.



\bibitem[LM]{MR1020043}
J.~L.~Lewis and M.~A.~M.~Murray.
\newblock Regularity properties of commutators and layer potentials associated
  to the heat equation.
\newblock {\em Trans. Amer. Math. Soc.}, {\bf 328}(2):815--842, 1991.

\bibitem[MPT1]{MR3107693}
E.~Milakis, J.~Pipher, and T.~Toro.
\newblock Harmonic analysis on chord arc domains.
\newblock {\em J. Geom. Anal.}, {\bf 23}(4):2091--2157, 2013.


\bibitem[MPT2]{MR3204862}
E.~Milakis, J.~Pipher, and T.~Toro.
\newblock Perturbations of elliptic operators in chord arc domains. {\em Harmonic analysis and partial differential equations}, {\em Contemp. Math.}, {\bf 612}:143--161, 2014.
 

\bibitem[Sem]{Se} 
S.~Semmes. Analysis vs. geometry on a class of rectifiable hypersurfaces in $R^n$. {\it Indiana Univ. Math. J.}, {\bf 39}:1005--1035, 1990.

\bibitem[Zha]{ZHAO}
Z.~Zhao.
\newblock {BMO} solvability and the {$A_\infty$} condition of the elliptic
  measure in uniform domains. 
	\newblock {\em  J. Geom. Anal.}, 2017. \url{http://doi.org/10.1007/s12220-017-9845-9}

\end{thebibliography}

\newcommand{\etalchar}[1]{$^{#1}$}

\end{document}